\documentclass[12pt]{amsart}
\usepackage{amssymb}
\usepackage[colorlinks=true, allcolors=blue]{hyperref}
\usepackage{amsmath,amsthm,amsxtra,mathtools}
\usepackage{amsfonts,eucal,mathrsfs}
\usepackage{enumitem}
\usepackage{a4wide}
\usepackage{xcolor}

\usepackage{bbm,stix}
\usepackage{subcaption}

\usepackage{tikz}
\usetikzlibrary{cd}

\newtheorem{mainthm}{Theorem}

\DeclarePairedDelimiter{\abs}{\lvert}{\rvert}

\DeclarePairedDelimiter{\bra}{(}{)}

\DeclarePairedDelimiter{\set}{\{}{\}}

\newcommand{\R}{\mathbb{R}}
\newcommand{\T}{\mathbb{T}}

\newcommand{\Ind}{\mathbb{1}}
\newcommand{\Id}{\text{Id}}

\newcommand{\calB}{\mathcal{B}}
\newcommand{\calC}{\mathcal{C}}
\newcommand{\calD}{\mathcal{D}}
\newcommand{\calE}{\mathcal{E}}
\newcommand{\calF}{\mathcal{F}}

\newcommand{\calH}{\mathcal{H}}
\newcommand{\calI}{\mathcal{I}}
\newcommand{\calJ}{\mathcal{J}}

\newcommand{\calL}{\mathcal{L}}
\newcommand{\calM}{\mathcal{M}}
\newcommand{\calN}{\mathcal{N}}
\newcommand{\calO}{\mathcal{O}}
\newcommand{\calP}{\mathcal{P}}

\newcommand{\calR}{\mathcal{R}}
\newcommand{\calS}{\mathcal{S}}
\newcommand{\calT}{\mathcal{T}}

\newcommand{\calV}{\mathcal{V}}
\newcommand{\calW}{\mathcal{W}}
\newcommand{\calX}{\mathcal{X}}

\newcommand{\Bern}{\mathfrak{b}}

\newcommand{\ssubset}{\subset\joinrel\subset}
\newcommand{\supp}{\text{supp}}
\newcommand{\dnabla}{\overline{\nabla}}

\DeclareMathAlphabet{\mathup}{OT1}{\familydefault}{m}{n}
\newcommand{\dd}[1]{\mathop{}\!\mathup{d} #1}

\newcommand{\ddiv}{\overline{\text{div}}\,}

\newtheorem{theorem}{Theorem}[section]

\newtheorem{lemma}[theorem]{Lemma}
\newtheorem{proposition}[theorem]{Proposition}

\theoremstyle{definition}
\newtheorem{definition}[theorem]{Definition}
\newtheorem{remark}[theorem]{Remark}
\newtheorem{example}[theorem]{Example}

\numberwithin{equation}{section}

\title[Variational convergence of the Scharfetter--Gummel scheme]{Variational convergence of the Scharfetter--Gummel scheme to the aggregation-diffusion equation and\\ vanishing diffusion limit}

\author{Anastasiia Hraivoronska}
\address{Anastasiia Hraivoronska, Universite Claude Bernard Lyon 1, ICJ UMR5208, CNRS, \'Ecole Centrale de Lyon, INSA Lyon, Universit\'e Jean Monnet, 69622 Villeurbanne cedex, France} 
\email{hraivoronska\,@\,math.univ-lyon1.fr}

\author{Andr\'e Schlichting}
\address{Andr\'e Schlichting, Institute for Applied Analysis, Ulm University, Germany} 
\email{andre.schlichting\,@\,uni-ulm.de}

\author{Oliver Tse}
\address{Oliver Tse, Department of Mathematics and Computer Science, University of Technology, 5600 MB Eindhoven, The Netherlands} 
\email{o.t.c.tse\,@\,tue.nl}

\date{\today}

\begin{document}

\begin{abstract}
	In this paper, we explore the convergence of the semi-discrete Scharfetter--Gummel scheme for the aggregation-diffusion equation using a variational approach. Our investigation involves obtaining a novel gradient structure for the finite volume space discretization that works consistently for any non-negative diffusion constant. This allows us to study the discrete-to-continuum and zero-diffusion limits simultaneously. The zero-diffusion limit for the Scharfetter--Gummel scheme corresponds to the upwind finite volume scheme for the aggregation equation. In both cases, we establish a convergence result in terms of gradient structures, recovering the Otto gradient flow structure for the aggregation-diffusion equation based on the 2-Wasserstein distance.
\end{abstract}

\maketitle

\tableofcontents

\section{Introduction}
In this paper, we study the convergence of the semi-discrete Scharfetter--Gummel numerical approximation for the aggregation-diffusion equation
\begin{equation}\label{eq:drift-aggregation-diffusion-equation}\tag{{\bf ADE}}
    \partial_t \rho_t = \text{div} \bigl( \epsilon \nabla \rho_t + \rho_t \nabla V + \rho_t \nabla (W * \rho_t) \bigr) \qquad \text{in } (0, T)\times\Omega,
\end{equation}
which describes the evolution of a curve of Borel probability measures $t\mapsto \rho_t\in\calP(\Omega)$ on a bounded convex domain $\Omega\subset\R^d$, where $\epsilon> 0$ is a diffusion coefficient, $V:\R^d\to \R$ is an external potential, and $W:\R^d\to\R$ is an interaction potential. In addition, we impose the no-flux boundary condition
\begin{equation*}
    \epsilon \partial_\nu \rho_t + \rho_t \partial_\nu (V + W * \rho_t) = 0 \qquad \text{on } \partial \Omega,
\end{equation*}
where $\nu$ denotes the outer normal vector on $\partial \Omega$. 

Our strategy employs a variational approach that not only provides the convergence of the semi-discrete (discrete in space, continuous in time) Schar\-fetter--Gummel scheme but also a generalized gradient structure for the cases $\epsilon>0$ and $\epsilon=0$. In particular, the method allows us to prove the convergence of the Scharfetter--Gummel ($\epsilon>0$) and upwind ($\epsilon=0$) approximation to the Otto gradient flow solutions of \eqref{eq:drift-aggregation-diffusion-equation}, which we outline in detail below.

\medskip

The Scharfetter--Gummel flux approximation originates from \cite{scharfetter1969large}, where the authors construct a numerical scheme for a system modeling semiconductor devices. Their objective was to develop a robust scheme for the system of equations with discontinuities or rapid variations in the potential. Independently, the same type of flux is introduced in \cite{il1969difference} for finite-difference schemes. Thereafter, the Scharfetter--Gummel scheme became the preferred finite-volume scheme for the drift-diffusion or convection-diffusion equations. While the original scheme deals with the one-dimensional problem, it has been generalized to higher dimensional problems \cite{farrell1991scharfetter} and the flux discretization approach became the basis for numerous other generalizations, e.g.\ for equations with nonlinear diffusion \cite{bessemoulin2012finite, eymard2006finite, jungel1995numerical} and to systems with source terms \cite{cheng2021generalised, ten2011finite}. Other works and approaches for the numerical solution are found in \cite{Bailo2020-1,Bailo2020-2,Filbet2012,Carillo2015,Duan2021,Lorenzo2018,Zeng2024}.

 To introduce the Scharfetter--Gummel scheme, we first introduce some common notations for finite-volume methods. Let $\{(\calT^h,\Sigma^h)\}_{h>0}$ be a family of finite (admissible) tessellations of a bounded and convex set $\Omega\subset\R^d$, where $\calT^h$ is the family of cells and $\Sigma^h\subset\calT^h\times\calT^h$ contains pairs $(K, L)$ that share a face, i.e.\ when $K,L\in\calT^h$ share a part of their boundary with positive $(d-1)$-dimensional Hausdorff measure, which we denote by $(K|L)$. With a slight abuse of notation, we adopt the notation $K|L$ to denote adjacent cells $(K,L)\in\Sigma^h$ to distinguish from pairs $(K,L)\in \calT^h\times\calT^h$. We further define $\calT^h_K$ to be the set of cells adjacent $K$. Let $\{x_K\}_{K\in\calT^h}$ be a family of points such that $x_K \in K$ to be selected later. The parameter $h>0$ is the maximal diameter of the cells. We make the definitions precise in Section~\ref{sec:assumptions-SG-tesselation}. For now, one can keep a Voronoi tessellation with its generating points in mind as an example of an admissible tessellation.

We illustrate how the Scharfetter--Gummel flux appears in the finite-volume space discretization of \eqref{eq:drift-aggregation-diffusion-equation}. First, consider the case without interaction potential, i.e.\ $W\equiv 0$. Rewriting \eqref{eq:drift-aggregation-diffusion-equation} as
\[
	\partial_t\rho_t + \text{div} j_t = 0,\qquad j_t = -\epsilon \nabla \rho_t - \rho_t \nabla V,
\]
integrating the first equation over a control volume $K\in\calT^h$, and then applying the divergence theorem yields the discrete continuity equation
\begin{equation}\label{eq:finite-volume-CE}\tag{{\bf CE$_h$}}
    \partial_t \rho^h_K + \ddiv \calJ^{h,\rho}_K  = 0, \qquad\text{with}\qquad  \ddiv \calJ^{h,\rho}_K \coloneq  \sum_{L\in\calT_K^h} \calJ^{h,\rho}_{K|L},
\end{equation}
where the numerical approximation for the flux $\calJ^{h,\rho}_{K|L}$ should be well chosen to approximate the continuous flux $j$. The idea of the Scharfetter--Gummel flux discretization is to solve a cell problem for two adjacent cells $K$ and $L$. Then, the cell problem is the one-dimensional boundary value problem: Find $u\in C^2([x_K,x_L])$ satisfying
\begin{equation}\label{eq:def:cell-problem}
	\begin{cases}
		-\partial_x \bigl(\epsilon \partial_x u + u \, q_{K|L}^h \bigr) = 0 \quad \text{on } [x_K, x_L] \\
		u(x_K) = \rho^h_K /|K|,\;\; u(x_L) = \rho^h_L/|L| 
	\end{cases}
	\qquad \text{for all $(K,L)\in\Sigma^h$},
\end{equation}
where $q_{K|L}^h$ is an approximation for the gradient of the potential term $\nabla V$ in~\eqref{eq:drift-aggregation-diffusion-equation} along a segment connecting $x_K$ and $x_L$. The solution of \eqref{eq:def:cell-problem}, which can be explicitly computed, is then used to define the Scharfetter--Gummel flux \cite{eymard2006finite}, defined for all $(K,L)\in\Sigma^h$ as 
\begin{equation}\label{eq:SG-flux}
    \calJ_{K|L}^{h,\rho} \coloneq \epsilon \tau_{K|L}^h \bigl( \Bern(q_{K|L}^h / \epsilon)\, u^h_K - \Bern(- q_{K|L}^h / \epsilon)\, u^h_L \bigr), \qquad u^h_K \coloneq \frac{\rho^h_K}{|K|},
\end{equation}
where $\tau_{K|L}^h \coloneq |(K|L)| / |x_L - x_K|$ is called the transmission coefficient and $\Bern(s) \coloneq s / (e^s - 1)$ is the {\em Bernoulli} function. The Scharfetter--Gummel scheme then reads
\begin{equation}\label{eq:SG}\tag{{\bf SGE$_h$}}
	\partial_t \rho^h_K + \sum_{L\in\calT_K^h} \calJ^{h,\rho}_{K|L} = 0,\qquad  \calJ^{h,\rho}_{K|L}=\epsilon \tau_{K|L}^h \bigl( \Bern(q_{K|L}^h / \epsilon)\, u^h_K-\Bern(- q_{K|L}^h / \epsilon)\, u^h_L \bigr).
\end{equation}

We are interested in a generalization of semi-discrete the Scharfetter--Gummel scheme \eqref{eq:SG} for \eqref{eq:drift-aggregation-diffusion-equation} that includes the interaction term $W$, which was considered in \cite{schlichting2022scharfetter}. In this case, the form of the flux is the same as in \eqref{eq:SG-flux}, but we include a discrete approximation of the interaction force $\nabla (W * \rho) = \int_\Omega \nabla W (\cdot - y) \rho (\dd y)$ of the form
\[
	q_{K|L}^h \coloneq V^h_L - V^h_K + \sum_{M\in\calT^h} \rho^h_M (W^h_{ML} - W^h_{MK}),\qquad (K,L)\in \Sigma^h,
\]
where $W^h_{MK} \coloneq W(x_K - x_M)$ for any $K, M \in \calT^h\times \calT^h$ such that $K\neq M$. 

The important property of the numerical flux \eqref{eq:SG-flux} is that the Bernoulli function interpolates between appropriate discretizations of the pure diffusion and pure drift problems. In the absence of the potential, i.e., $q_{K|L}^h = 0$, the flux becomes $\epsilon \tau_{K|L}^h ( u_K - u_L )$. More interestingly, in the vanishing diffusion limit $\epsilon\to 0$, the Scharfetter--Gummel scheme converges to
\begin{equation}\label{eq:upwind-scheme}\tag{{\bf Up$_h$}}
    \partial_t \rho^h_K + \sum_{L\in \calT_K^h}\calJ_{K|L}^{h,\rho,\text{Up}} =0,\qquad \calJ_{K|L}^{h,\rho,\text{Up}}= \tau_{K|L}^h \bigl( q_{K|L}^{h,+}u^h_K - q_{K|L}^{h,-} u^h_L \bigr),
\end{equation}
which is the upwind flux discretization for the aggregation equation 
\begin{equation}\label{eq:aggregation}\tag{{\bf AE}}
    \partial_t \rho = \text{div} (\rho \nabla (V + W * \rho))\qquad \text{in } (0, T) \times \Omega.
\end{equation}

The convergence of the Scharfetter--Gummel scheme \eqref{eq:SG} to the weak solutions of \eqref{eq:drift-aggregation-diffusion-equation} in the absence of an external potential is proven in \cite{schlichting2022scharfetter}. Moreover, it was shown there that the discrete solutions satisfy an energy-dissipation inequality along the evolution, which is an important structure-preserving property. We aim to go one step further and prove the convergence of a variational structure for \eqref{eq:SG} to the Otto gradient-flow structure for \eqref{eq:drift-aggregation-diffusion-equation}.

The convergence of \eqref{eq:drift-aggregation-diffusion-equation} to \eqref{eq:aggregation} was also shown to hold in \cite{santambrogio2023} with and without rates, a recent article which we only became aware of at the end of writing this manuscript. The authors proved the convergence with rate $\sqrt{\epsilon}$ when the interaction potential $W$ is $\lambda$-convex  with $\lambda\le 0$. We include proof of the convergence without rates for completeness and for outlining the similarity of the proof with the other proofs within this manuscript.

\subsubsection*{Strategy and outline}
The goal of this paper is to complete the commutative diagram in Figure~\ref{diagram:S-G} below, where the convergence results correspond to the convergence of gradient-flow structures. To make the goal clear, we briefly explain the gradient structures involved and the type of convergences we are interested in.

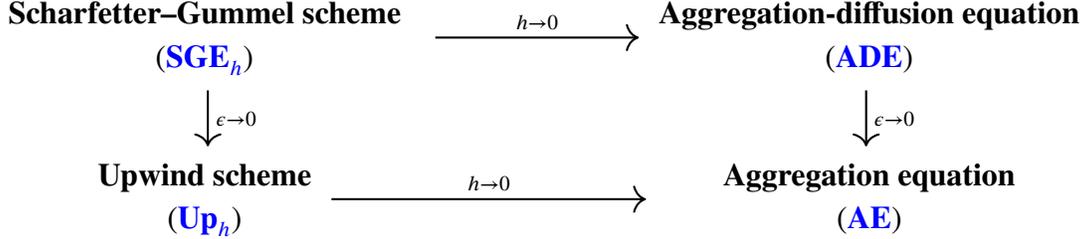
\begin{figure}[ht]
\centering
\begin{tikzcd}
\begin{array}{c}
\textbf{Scharfetter--Gummel scheme} \\[0.2em]
    \eqref{eq:SG}
\end{array}\qquad
\arrow[r, "h\to 0", start anchor={[xshift=-4ex]}, end anchor={[xshift=5ex]}] 
\arrow[d, "\epsilon\to 0", start anchor={[xshift=-2ex]}, end anchor={[xshift=-2ex]}]
& 
\qquad \begin{array}{c}
\textbf{Aggregation-diffusion equation} \\[0.2em]
\eqref{eq:drift-aggregation-diffusion-equation}
\end{array} 
\arrow[d, "\epsilon\to 0", start anchor={[xshift=2ex]}, end anchor={[xshift=2ex]}] \\
\begin{array}{c}
     \textbf{Upwind scheme} \\[0.2em]
     \eqref{eq:upwind-scheme}
     
\end{array}\qquad \arrow[r, "h\to 0", start anchor={[xshift=-5ex]}, end anchor={[xshift=1ex]}] 
&
\qquad\begin{array}{c}
    \textbf{Aggregation equation} \\[0.2em]
     \eqref{eq:aggregation}
\end{array}
\end{tikzcd}
\caption{We are interested in the convergence results indicated by the arrows in the sense of EDP convergence. The arrows with the label "$h\to 0$" indicate the convergence of the discrete approximation to the continuous problem. The arrows with the label "$\epsilon\to 0$" indicate the vanishing diffusion limit.} \label{diagram:S-G}
\end{figure}

The right-hand side of Figure~\ref{diagram:S-G} corresponds to the continuous setting that is rather well understood. The Otto-Wassertein gradient-flow theory \cite{ambrosio2008gradient,JKO1998,Otto2001} provides a gradient-flow formulation for the aggregation-diffusion equation \eqref{eq:drift-aggregation-diffusion-equation} with respect to the $L^2$-Wasserstein metric and the \emph{driving energy} 
\begin{equation}\label{eq:energy-cont-with-interaction}
    \calP(\Omega)\ni\rho\mapsto \calE_\epsilon(\rho) = \begin{cases}
        \displaystyle \epsilon \int_\Omega \phi\biggl( \frac{\dd\rho}{\dd\calL^d} \biggr) \dd\calL^d + \int_\Omega V \dd\rho + \frac{1}{2} \int_\Omega ( W * \rho ) \dd\rho  &  \text{if } \rho \ll \calL^d, \\
        +\infty & \text{otherwise,}
    \end{cases}
\end{equation}
where $\phi(s)=s \log s -s +1$ for $s\in\R_+$, $\calL^d$ denotes the Lebesgue measure on $\R^d$, and $\calE_\epsilon$ is assumed w.l.o.g.\ to be non-negative (cf.\ Section~\ref{sec:assumptions-SG}). Here, we consider gradient flow solutions to \eqref{eq:drift-aggregation-diffusion-equation} in terms of the {\em Energy-Dissipation Balance} (EDB), which we now describe. We begin by recalling that \eqref{eq:drift-aggregation-diffusion-equation} can be expressed as
\begin{align}
	\partial_t\rho_t + \text{div} j_t &= 0\qquad \text{in } (0, T) \times \Omega,\label{eq:CE}\tag{{\bf CE}}\\
	 j_t &= -\rho_t \nabla \calE_\epsilon'(\rho_t),\label{eq:FF}\tag{{\bf KR}}
\end{align}
where \eqref{eq:CE} suggests that  $(\rho,j)$ satisfies the {\em continuity equation}, while \eqref{eq:FF} describes the relationship between the force $-\nabla \calE_\epsilon'(\rho_t)$ and the flux $j_t$, known as the {\em kinetic relation}. Here, $\calE_\epsilon'$ denotes the variational derivative of $\calE_\epsilon$.

By introducing a {\em dual dissipation potential} $\calR^* : \calP(\Omega) \times C_b(\Omega;\R^d) \to \R_+$,
\begin{equation*}
    \calR^*(\rho, \xi) = \frac{1}{2} \int_\Omega |\xi|^2 \dd\rho,
\end{equation*}
the kinetic relation \eqref{eq:FF} may be further expressed as
\[
	j_t = D_2\calR^*(\rho_t, -\nabla \calE_\epsilon'(\rho_t)).
\]
Via Legendre-Fenchel duality, we obtain a variational characterization of the kinetic relation:
\begin{align}\label{eq:force-flux}
	\calR(\rho_t, j_t) + \calR^*(\rho_t, -\nabla \calE_\epsilon'(\rho_t)) = \langle j_t,-\nabla \calE_\epsilon'(\rho_t)\rangle,
\end{align}
where the {\em dissipation potential} $\calR$ is the Legendre dual of $\calR^*$ w.r.t.\ its second argument, i.e.,
\begin{equation}\label{eq:dissipation-potential}
	(\rho,j)\in \calP(\Omega)\times \calM(\Omega;\R^d)\mapsto \calR(\rho,j) = \frac{1}{2}\int_\Omega \biggl|\frac{\dd j}{\dd \rho}\biggr|^2 \dd\rho,
\end{equation}
where $\calM(\Omega;\R^d)$ is the space of finite $\R^d$-valued Radon measures. Under the {\em chain rule}
\begin{align}\label{eq:chain-rule}\tag{{\bf CR}}
	-\frac{\dd }{\dd t}\calE_\epsilon(\rho_t) = \langle j_t,-\nabla \calE_\epsilon'(\rho_t)\rangle, 
\end{align}
along density-flux pairs $(\rho,j)$ satisfying the continuity equation \eqref{eq:CE}, one arrives at a variational expression for the solution of \eqref{eq:drift-aggregation-diffusion-equation}. Indeed, integrating \eqref{eq:force-flux} over arbitrary intervals $[s,t]\in[0,T]$ and employing the chain rule \eqref{eq:chain-rule}, one obtains the Energy-Dissipation Balance:
\begin{equation}\label{eq:Otto-edf}\tag{{\bf EDB}}
   \calI_\epsilon^{[s,t]} (\rho, j) \coloneq \int_s^t \calR(\rho_r,j_r) + \calR^*(\rho_r, - \nabla \calE'_\epsilon(\rho_r))\dd r + \calE_\epsilon(\rho_t) - \calE_\epsilon(\rho_s) = 0.
\end{equation}
Morally, any pair $(\rho,j)$ satisfying the continuity equation \eqref{eq:CE} and \eqref{eq:Otto-edf} is said to be a gradient flow solution of \eqref{eq:drift-aggregation-diffusion-equation} with respect to the structure $(\calE, \calR, \calR^*)$ if it satisfies, additionally, the chain rule \eqref{eq:chain-rule} ensuring the non-negativity of the functional $\calI_\epsilon^{[s,t]}$. Although there are other ways of defining gradient flow solutions to \eqref{eq:drift-aggregation-diffusion-equation}, we choose to use the definition based on EDB since this works well in the generalized gradient flow setting \cite{PRST2022} as seen below.

For $\lambda$-convex functionals $\calE_\epsilon$ w.r.t.\ the Wasserstein distance $W_2$, it is a standard result of evolutionary $\varGamma$-convergence for gradient flows \cite{serfaty2011gamma} that, as $\epsilon\to 0$, the gradient flow solutions of \eqref{eq:drift-aggregation-diffusion-equation} converge to the gradient flow solutions of the corresponding aggregation equation \eqref{eq:aggregation}.

\medskip
The left-hand side of Figure~\ref{diagram:S-G} corresponds to the discrete setting for which the gradient structure is not well understood. For this reason, our first objective is to present a generalized gradient-flow (GGF) formulation for the Scharfetter--Gummel scheme \eqref{eq:SG}. In particular, we show in Section~\ref{sec:cosh-structure-SG} that the scheme fits into the (by now, common) `cosh' gradient-structure framework with the discrete driving energy $\calE_{\epsilon,h}: \calP(\calT^h) \to \R_+$,
\begin{equation}\label{eq:energy-with-interaction}
    \calE_{\epsilon,h}(\rho^h) = \epsilon \sum_{K\in\calT^h} \phi(u^h_K)|K|  + \sum_{K\in\calT^h} V^h_K \rho^h_K + \frac{1}{2} \sum_{(K, L)\in\calT^h\times\calT^h} W^h_{KL} \rho^h_K \rho^h_L, \qquad u^h_K \coloneq \frac{\rho^h_K}{|K|},
\end{equation}
and discrete dual dissipation potential $\calR_{\epsilon,h}^*: \calP(\calT^h) \times \calB(\Sigma^h) \to \R_+$ defined in \eqref{eq:def:cosh:Rstar}, where $\calB(A)$ denotes the set of bounded functions on $A$.

That being said, the `cosh' gradient structure turns out to be ill-suited for proving the desired convergence due to the inclusion of the interaction potential $W$, which gives rise to a dissipation potential that depends on $W$ and $\rho^h$. Such phenomenon is known as \emph{tilt-dependence} of gradient systems and was recently discussed in detail in \cite{peletier2023cosh}, where it was established that \emph{tilt-independent} gradient structures give rise to better convergence properties. Using the de-tilting technique~\cite{peletier2023cosh}, we introduce a \emph{new} tilt-independent gradient structure for the Scharfetter--Gummel scheme in the presence of both external and interaction potentials (cf.\ Section~\ref{sec:tilted-gradient-structure}) and allows us to pass to the $h\to 0$ and $\epsilon\to 0$ limits.

We show in Section~\ref{sec:gf-SG} that the Scharfetter--Gummel scheme \eqref{eq:SG} possesses a gradient structure with driving energy $\calE_{\epsilon,h}$ (cf.\ \eqref{eq:energy-with-interaction}) and the \emph{tilt-independent dual dissipation potential} $\calR_{\epsilon,h}^*$ given by
	\begin{equation}\label{eq:tilted-Rstar}
		\calR_{\epsilon,h}^*(\rho^h, \xi^h) \coloneq 2\sum_{(K,L)\in\Sigma^h} \tau_{K|L}^h \,\alpha_\epsilon^* \bra*{ u^h_K, u^h_L, \frac{\xi^h_{K|L}}{2} }, \qquad u^h_K \coloneq \frac{\rho^h_K}{\abs{K}},
	\end{equation}
	where $\alpha_\epsilon^*:\R_+\times \R_+\times\R \to \R_+$  is defined (see Lemma~\ref{lem:properties-alpha} for more details) for any $\epsilon>0$ by
	\begin{equation}\label{eq:def:alpha*}
		\alpha_\epsilon^*(a, b, \xi) \coloneq \epsilon \int_0^\xi \sinh\bra[\Big]{\frac{x}{\epsilon}} \Lambda_H\bra*{a e^{-x/\epsilon}, b e^{x/\epsilon}} \dd x= \epsilon^2 \alpha_1^* \bra[\Big]{a, b, \frac{\xi}{\epsilon} }.
	\end{equation}
Hereby the \emph{harmonic-logarithmic mean} $\Lambda_H : \R_+ \times \R_+ \to \R_+$ (see also Lemma~\ref{lem:harm_log_mean}) is given as
\begin{equation}\label{eq:def:HarmLogMean}
	\Lambda_H (s, t) \coloneq \frac{1}{\Lambda \left( 1/s, 1/t \right)} \qquad\text{with}\qquad \Lambda(s, t) = \frac{s - t}{\log s - \log t} \qquad\text{for } s\ne t.
\end{equation}
Based on these definitions, the two equations in~\eqref{eq:SG} become a discrete continuity equation for the density-flux pair $(\rho^h,j^h)$ and a kinetic relation providing a force-flux relation:
\begin{align}
    \partial_t \rho^h_t + \overline{\text{div}} &j^h_t = 0 \qquad \text{in } (0,T) \times \calT^h, \label{eq:CEh} \tag{{\bf CE$_h$}}\\
    &j^h_t = D_2 \calR_{\epsilon,h}^* (\rho^h_t, -\dnabla \calE_{\epsilon,h}'(\rho^h_t)), \label{eq:FFh}\tag{{\bf KR$_h$}}
\end{align}
where $\dnabla \varphi(K,L) = \varphi(L)-\varphi(K)$ is the discrete gradient. Together with the discrete chain rule
\begin{align}\label{eq:CRh}\tag{{\bf CR$_h$}}
	-\frac{\dd }{\dd t}\calE_{\epsilon,h}(\rho_t^h) = \langle j_t^h,-\dnabla \calE_{\epsilon,h}'(\rho_t^h)\rangle, 
\end{align}
the pair $(\rho^h,j^h)$ is shown to satisfy the discrete Energy-Dissipation Balance:
\begin{equation}\label{eq:EDPh}\tag{{\bf EDB$_h$}}
    \calI_{\epsilon,h}^{[s,t]} (\rho^h, j^h) \coloneq \int_s^t \calR_{\epsilon,h}(\rho^h_r, j^h_r) + \calR_{\epsilon,h}^*(\rho^h_r, -\dnabla \calE_{\epsilon,h}'(\rho_t^h)) \dd r + \calE_{\epsilon,h}(\rho^h_t) - \calE_{\epsilon,h}(\rho^h_s) = 0,
\end{equation}
for any interval $[s,t]\subset[0,T]$.

\medskip
Our main interest lies in establishing discrete-to-continuum convergence results that connect the left-hand and the right-hand sides of Figure~\ref{diagram:S-G}. For the convergence of \eqref{eq:SG} to \eqref{eq:drift-aggregation-diffusion-equation} (top horizontal arrow), we define the GGF solutions to \eqref{eq:SG} as the minimizers of the energy-dissipation functional $\calI_{\epsilon,h}$ corresponding to the tilt-independent structure defined through~\eqref{eq:tilted-Rstar} (cf.\ Section~\ref{sec:discrete-structures}). We then follow a similar strategy as in \cite{hraivoronska2023diffusive}, which studies the diffusive limit of random walks on tessellations using variational techniques. However, every step of the strategy requires an adaptation to the new gradient structure. The main challenge here is to prove a $\varGamma$-convergence result for the Fisher information, which takes the form
\begin{equation*}
    \calD_{\epsilon,h}(\rho^h) \coloneq \calR_{\epsilon,h}^*(\rho^h, -\dnabla \calE_{\epsilon,h}'(\rho^h)) =\sum_{(K,L)\in\Sigma^h} \beta_\epsilon (u^h_K, u^h_L)\, \tau_{K|L}^h + \calD_{\epsilon,h}^1 (\rho^h) + \calD_{\epsilon,h}^2 (\rho^h),
\end{equation*}
where $\beta_\epsilon(a,b)\coloneq \alpha_\epsilon^* (a, b, -\epsilon \log \sqrt{b/a})$ with $\alpha_\epsilon^*$ from~\eqref{eq:def:alpha*}, and  $\calD_{\epsilon,h}^1$, $\calD_{\epsilon,h}^2$ are defined in Section~\ref{sec:EDP-convergence-tilt}.
The splitting mimics the expanded form of the continuous Fisher information:
\[
 \calD_\epsilon(\rho) \coloneq \calR^*(\rho, - \nabla \calE'_\epsilon(\rho)) = 2\epsilon^2 \int \abs*{ \nabla \sqrt{u} }^2\dd x + \epsilon\int \nabla u \cdot \nabla \mathsf{Q}(\rho)\dd x + \frac{1}{2} \int \abs*{ \nabla \mathsf{Q}(\rho) }^2\, u\dd x,
\]
where $\mathsf{Q}(\rho) = V + W\ast\rho$.
The function $\beta_\epsilon$ depending on $\alpha_\epsilon^*$ in~\eqref{eq:def:alpha*} is only defined by an integral, which makes it more difficult to work with as compared to the Fisher information for the `cosh' structure studied in \cite{hraivoronska2023diffusive}. Nevertheless, it satisfies (see Lemma~\ref{lem:properties-alpha}) the bounds
\[
    \frac{\epsilon^2}{4} \frac{(a - b)^2}{a + b} \leq \beta_\epsilon (a, b) \leq \frac{\epsilon^2}{2} \bra*{\sqrt{b} - \sqrt{a}}^2, \qquad a, b \geq 0, 
\]
thereby allowing us to prove a $\varGamma$-convergence result for $\beta_\epsilon$ (cf.\ Section~\ref{sec:gamma-convergence-fisher-SG}), albeit under more stringent assumptions on the tessellations compared to \cite{hraivoronska2023diffusive}. Additionally, we will need to establish new convergence results for the other parts of $\calD^h$ that depend on the interaction term $q_{K|L}^h$.

The arrow with $\epsilon\to 0$ on the left side of Figure~\ref{diagram:S-G} refers to the convergence of the Scharfetter--Gummel scheme \eqref{eq:SG} to the upwind approximation \eqref{eq:upwind-scheme} as $\epsilon \to 0$ in terms of the generalized gradient structure. Since the state space is a fixed finite tessellation, this result is not difficult to obtain. On the contrary, the {\em direct} convergence of the upwind scheme \eqref{eq:upwind-scheme} to the aggregation equation \eqref{eq:aggregation} appears to be very challenging. The difficulty is described in the literature but is still not well studied. The intuitive idea is that the structure of the tessellation can lead to strong oscillations in the solutions of the discrete continuity equation. More specifically, unlike in the 1-dimensional case, one can not expect propagation of the $\text{BV}$-bound, assuming that the initial data is in $\text{BV}$. Indeed, there is a simple example of a 2-dimensional tessellation consisting of lines of squares with size $h$ alternating with lines of squares with size $h/2$, for which the total variation of the discrete solutions blows up as $h^{-1/2}$ even for a constant velocity field (see details in \cite{despres2004explicit}). 

On the other hand, the convergence results in the strong topology are available on general tessellations for Lipschitz velocity fields \cite{aguillon2018error, merlet2008and, merlet2007error}. When one treats general tessellations and rough velocity fields simultaneously, the convergence is proven in the weak topology \cite{SchlichtingSeis2017,schlichting2018analysis} for time-explicit upwind schemes on Cartesian grids and time-implicit upwind schemes on regular general meshes. A first variational method for Fokker-Planck equations based on upwind dissipation functionals is contained in~\cite{Cances2020}. See also~\cite{EspositoHeinzeSchlichting2023,esposito2024class,EPSS2021} for a study on general graphs and their continuum limits.

A new method for proving regularity estimates for solutions of the discrete continuity equations with non-Lipschitz velocity field and non-Cartesian but periodic tessellations is found in \cite{jabin2022discretizing}, which is significant for future research in this area. Given the state-of-art, at the moment, we cannot expect to prove the discrete-to-continuum convergence of the gradient structure for \eqref{eq:upwind-scheme} for general tessellations. Nevertheless, we obtain a convergence result for the Cartesian grid. We believe that this result is already worthwhile since it does not require any assumptions on the integrability of the initial data, allowing us to include atomic measures as initial data. 

To summarize, the rest of the paper is organized as follows. In Section~\ref{sec:assumptions-and-results}, we specify the assumptions on tessellations and potentials and present the main results. We introduce the gradient structure for \eqref{eq:drift-aggregation-diffusion-equation} and two generalized gradient structures for the Scharfetter--Gummel scheme in Section~\ref{sec:gradient-structures}. The subsequent sections contain the proofs of the convergence results. Section~\ref{sec:EDP-convergence-tilt} is dedicated to the discrete-to-continuum convergence of \eqref{eq:SG} to \eqref{eq:drift-aggregation-diffusion-equation} for $\epsilon>0$. The vanishing diffusion limit $\epsilon\to 0$ from \eqref{eq:SG} to \eqref{eq:upwind-scheme} and from \eqref{eq:drift-aggregation-diffusion-equation} to \eqref{eq:aggregation} are presented in Section~\ref{sec:vanishing-diffusion}. We deal with the convergence of \eqref{eq:upwind-scheme} to \eqref{eq:aggregation} in Section~\ref{sec:upwind-to-aggregation}.

\subsection*{Acknowledgments} A.H.\ and O.T.\ acknowledge support from NWO Vidi grant 016.Vidi.189.102 on "Dynamical-Variational Transport Costs and Application to Variational Evolution". A.S.\ is supported by the Deutsche Forschungsgemeinschaft (DFG, German Research Foundation) under Germany's Excellence Strategy EXC 2044 -- 390685587, \emph{Mathematics M\"unster: Dynamics--Geometry--Structure}.

\section{Assumptions and Main Results}\label{sec:assumptions-and-results}
We specify our assumptions on the family of tessellations in Section~\ref{sec:assumptions-SG-tesselation} and the external and interaction potentials in Section~\ref{sec:assumptions-SG}. The main results of this paper are summarized in Section~\ref{sec:main-results-SG}. 

\subsection{Assumptions on tessellations}\label{sec:assumptions-SG-tesselation}
Let $\Omega\subset\R^d$ be an open bounded convex set. A tessellation $(\calT^h,\Sigma^h)$ covering $\Omega$ consists of a family $\calT^h$ of mutually disjoint cells (usually denoted by $K$ or $L$) that are open convex sets and $\Omega\subset \bigcup_{K\in\calT^h} K$, and a family $\Sigma^h =\{ (K, L)\in\calT^h\times\calT^h : \mathscr{H}^{d-1} (\overline{K}\cap\overline{L}) > 0  \}$ of pairs of cells with a common face. Here, $\mathscr{H}^{d-1}$ denotes the $(d-1)$-dimensional Hausdorff measure.
The common face of a pair $(K,L)\in \Sigma^h$ is denoted by $(K|L)$. We further define $\calT^h_K$ to be the set of cells adjacent $K$. The characterizing size of a tessellation is its maximum diameter:
$$
    h \coloneq \max \bigl\{\text{diam}(K),\, K\in \calT^h\bigr\}.
$$ 
The maximum diameter $h>0$ gives an upper bound on the volumes of the cells $|K|\le C_d h^d$ and faces $|(K|L)| \le C_{d-1} h^{d-1}$, where $C_d$, $C_{d-1}>0$ are universal constants depending only on the spatial dimension $d\ge 1$. In our work, it is also necessary to assume lower bounds on the volumes of the cells to prevent the degeneration of cells, which is guaranteed by the following non-degeneracy assumption. 

\vspace{0.3cm}
\framebox{ \centering \begin{minipage}{0.92\linewidth}
\textbf{Non-degeneracy.}\label{def_zeta_regularity}
    There exists $\zeta \in (0, 1)$ such that
    \begin{enumerate}[label=(\roman*)]
        \item  For each $K\in\calT^h$, there is an inner ball $B(b_K, \zeta h) \subset K$
        with $b_K = \intbar_K x \dd x$;
        \item For every $(K,L)\in\Sigma^h$ it holds that $|(K|L)| \geq \zeta h^{d-1}$.
    \end{enumerate}
\end{minipage}}
\vspace{0.3cm}

A standard assumption, often embedded in the definition of admissible tessellations in the finite-volume setup, is the following orthogonality assumption.

\vspace{0.3cm}
\framebox{ \centering \begin{minipage}{0.92\linewidth}
\textbf{Orthogonality.} There exists a family of points $\calX^h = \{x_K\}_{K\in\calT^h}$ such that $x_K \in K$ and for all $(K, L)\in \Sigma^h$, the face $(K|L)$ is orthogonal to the vector $x_L - x_K$, i.e.,
    \begin{equation}\tag{Ort}\label{ass:orthogonality}
        (K|L) \perp (x_L - x_K).
    \end{equation}
\end{minipage}}
\vspace{0.3cm}

We now summarize the assumptions on the tessellations used within this paper.

\vspace{0.3cm}
\framebox{ \centering \begin{minipage}{0.92\linewidth}
\textbf{Admissible tesselations}. The family of tessellations $\{(\calT^h, \Sigma^h)\}_{h>0}$ satisfy 
\begin{align}\label{ass:tessellation}\tag{A$_\calT$}\hspace{-1em}\left\{\;
\begin{aligned}
    &\text{for any $h>0$, all cells $K\in\calT^h$ are open, convex, and mutually disjoint;} \\
    &\text{$\{(\calT^h, \Sigma^h)\}_{h>0}$ is non-degenerate with some $\zeta\in(0, 1)$ independent of $h$;} \\
    &\text{$\{(\calT^h, \Sigma^h)\}_{h>0}$ satisfies \eqref{ass:orthogonality} for any $h>0$.}
    \end{aligned}\right.
\end{align}
\end{minipage}}
\vspace{0.3cm}

\begin{remark}
    \begin{enumerate}
        \item Voronoi tessellations satisfy \eqref{ass:orthogonality} with $\calX^h$ being the generating points of the Voronoi cells. 
        
        \item As a consequence of the non-degeneracy assumption, the transmission coefficients $\tau_{K|L}^h = |(K|L)| / |x_L - x_K|$ for a family of admissible tesselations $\{(\calT^h, \Sigma^h)\}_{h>0}$ satisfy 
\begin{equation}\label{eq:tau:bound}
    c_\tau \leq \inf_{h>0} \inf_{K\in \calT^h} h^2 \sum_{L: (K,L)\in\Sigma^h} \tau_{K|L}^h \leq \sup_{h>0} \sup_{K\in \calT^h} h^2 \sum_{L: (K,L)\in\Sigma^h} \tau_{K|L}^h 
    \leq C_\tau \,.
\end{equation}
with two universal constants $0<c_\tau\leq C_\tau< +\infty$.

    \item The non-degeneracy assumption implies that the number of adjacent cells is uniformly bounded. More precisely, there exists a universal constant $C_\calN > 0$ such that
    \begin{equation}\label{eq:univ-bound-adjacent}
        \sup_{h>0} \sup_{K\in\calT^h} |\calT^h_K| \leq C_\calN < \infty.
    \end{equation}
    
    \item An important consequence of the orthogonality assumption is the following equality for any $K\in\calT^h$ away from the boundary:
\begin{equation}\label{eq:zero-average}
    \sum_{L\in\calT^h_K} \tau^h_{K|L} (x_L - x_K)
    = \sum_{L\in\calT^h_K} |(K|L)| n_{K|L} 
    = \int_{\partial K} n_{\partial_K} \dd \calH^{d-1} = 0,
\end{equation}
where $n_{K|L} \coloneq (x_L - x_K) / |x_L - x_K|$ and $n_{\partial K}$ is the outer normal vector for the cell $K$.
    \end{enumerate}
\end{remark}

\subsection{Assumptions on potentials}\label{sec:assumptions-SG}

We assume the following properties for the potentials.

\vspace{0.3cm}
\framebox{
\centering
\begin{minipage}{0.92\linewidth}

\textbf{Assumptions on $V$.}\label{ass_potential}
    The external potential $V\in \text{Lip}(\R^d)\cap C^1(\R^d)$ is non-negative.\\[-0.5em]
 
\textbf{Assumptions on $W$.}\label{ass_interaction}
    The interaction potential $W\colon \R^d \to \R$ non-negative, i.e.\ $W(x) \geq 0$ for all $x\in\R^d$ and is symmetric, i.e.\ $W(x) = W(-x)$. In addition, we assume the interaction potential to be either a pointy potential
    \begin{equation}\label{ass:pointy-potential}\tag{Pointy}
        W \in \text{Lip}(\R^d) \cap C^1(\R^d\backslash\{0\}),
    \end{equation}
    or a continuously differential potential
    \begin{equation}\label{ass:smooth-potential}\tag{C$^1$}
        W \in \text{Lip}(\R^d)\cap C^1(\R^d).
    \end{equation}
\end{minipage}}

\medskip

\begin{example} A typical example of interaction potentials appearing in mathematical models of the collective behavior of individuals is the Morse potential
$$
    W(x) = C_r e^{-|x|/\ell_r} - C_a e^{-|x|/\ell_a},
$$
where $\ell_a$ and $\ell_r$ represent the attractive and
repulsive potential ranges and $C_a$ and $C_r$ represent their
respective amplitudes. With the choice $C_r \geq C_a > 0$ and $\ell_a > \ell_r$, it holds that $W(x) \geq 0$ for all $x\in\R^d$ and $W$ satisfies \eqref{ass:pointy-potential}.  
\end{example}

As mentioned above we define the discrete potentials accordingly as 
\[
\begin{aligned}
	 V^h_K &\coloneq V(x_K)\quad\text{for $K\in \calT^h$,\; and}\\
	 W^h_{KL} &\coloneq W(x_L - x_K)\quad \text{for $(K,L)\in\calT^h \times \calT^h$.}
\end{aligned}
\]
We claim in Lemma~\ref{lem:properties-potentials} that the assumptions on $V$ and $W$ indicated above imply
that
\begin{equation}\label{eq:q-to-potentials}
    q_{K|L}^h = \nabla (V + W * \hat{\rho}^h )(x_K) \cdot (x_L - x_K) + o(h)|_{h\to 0}.
\end{equation}
This estimate will play an important role in several statements of this paper. Due to the assumptions on the potentials $V$ and $W$, we further deduce that
\begin{equation}\label{eq:Cw-bound}
    |q_{K|L}^h| \leq c_\text{pot} h\qquad\text{for all $(K,L)\in\Sigma^h$},
\end{equation}
with $c_\text{pot}\coloneq\text{Lip}(V) + \text{Lip}(W)$.

\begin{remark}
We could have also defined 
\[
    V^h_K \coloneq \intbar_K V(x) \dd x,\quad K\in\calT^h,\qquad W^h_{KL} \coloneq \intbar_K \intbar_L W(x - y) \dd x \dd y,\quad (K,L) \in \calT^h\times \calT^h.
\]
One can verify that \eqref{eq:q-to-potentials} remains true. As a consequence, all the results of this paper hold also with these definitions. 
\end{remark}

\subsection{Main results}\label{sec:main-results-SG}
To see the scope of the main results, we indicate the corresponding statements on the arrows in Figure~\ref{diagram:with-theorems}.  

Our first statement is that the Scharfetter--Gummel scheme \eqref{eq:SG} has the generalized gradient structure. This allows us to define the GGF solution to \eqref{eq:finite-volume-CE} as a pair $(\rho^h, j^h)$ satisfying the continuity equation \eqref{eq:CEh}, which is a minimizer for the energy-dissipation functional \eqref{eq:EDPh}. All components of the energy-dissipation functional are made precise in Section~\ref{sec:tilted-gradient-structure} and Lemma~\ref{lem:flux-force-for-tilt} proving that the structure is indeed correct.

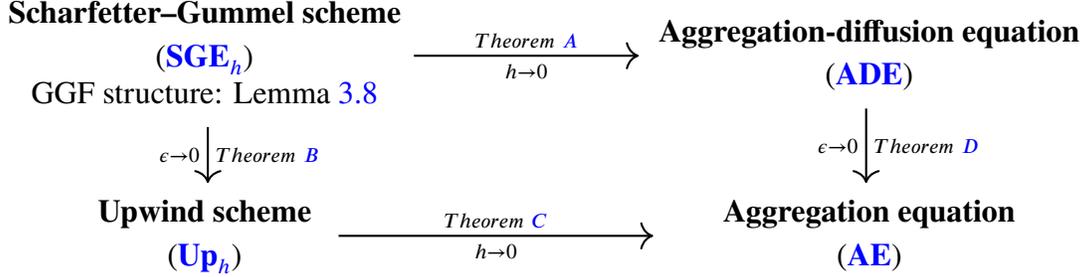
\begin{figure}[ht]
\centering
\begin{tikzcd}
\begin{array}{c}
\textbf{Scharfetter--Gummel scheme} \\[0.2em]
    \eqref{eq:SG} \\
    \text{GGF structure: Lemma~\ref{lem:flux-force-for-tilt}}
\end{array}\qquad
\arrow[r, "h\to 0"', "Theorem~\ref{th:EDP-convergence-tilt}", start anchor={[xshift=-5.5ex]}, end anchor={[xshift=5ex]}] 
\arrow[d, "\epsilon\to 0"', "Theorem~\ref{th:EDP-vanishing-diffusion}", start anchor={[xshift=-2ex]}, end anchor={[xshift=-2ex]}]
& 
\qquad \begin{array}{c}
\textbf{Aggregation-diffusion equation} \\[0.2em]
     \eqref{eq:drift-aggregation-diffusion-equation}
\end{array} 
\arrow[d, "\epsilon\to 0"', "Theorem~\ref{th:vanishing-diffusion-continuous}", start anchor={[xshift=2ex]}, end anchor={[xshift=2ex]}] \\
\begin{array}{c}
     \textbf{Upwind scheme} \\[0.2em]
     \eqref{eq:upwind-scheme}
\end{array}\qquad \arrow[r, "h\to 0"', "Theorem~\ref{th:upwind-to-aggregation}", start anchor={[xshift=-4.5ex]}, end anchor={[xshift=1.5ex]}] 
&
\qquad\begin{array}{c}
    \textbf{Aggregation equation} \\[0.2em]
     \eqref{eq:aggregation}
\end{array}
\end{tikzcd}
\caption{This diagram depicts the main results of this paper. The generalized gradient structure for the Scharfetter--Gummel scheme is established in Lemma~\ref{lem:flux-force-for-tilt}. The labels on the arrows indicate the corresponding convergence statements in the sense of the EDP convergence.} \label{diagram:with-theorems}
\end{figure}

Section~\ref{sec:EDP-convergence-tilt} is devoted to the discrete-to-continuum convergence of the Scharfetter--Gummel scheme as $h\to 0$ for a fixed diffusion coefficient $\epsilon > 0$. To relate the discrete objects with the continuum, we employ the following reconstruction procedure for a density-flux pair $(\rho^h, j^h)$ satisfying \eqref{eq:CEh}
\begin{equation}\label{eq:lifted-pair}
    \frac{\dd\hat{\rho}^h}{\dd\calL^d} \coloneq \sum_{K\in\calT^h} \frac{\rho^h(K)}{|K|} \Ind_K, \qquad \hat{\jmath}^h \coloneq \sum_{(K,L) \in \Sigma^h} j^h_{K|L}\,\sigma_{K|L}^h,
\end{equation}
where $\sigma_{K|L}^h\in\calM(\Omega; \R^d)$ are chosen in a way such that for any $(\rho^h, j^h)$ satisfying the discrete continuity equation \eqref{eq:CEh} the lifted pair $(\hat{\rho}^h, \hat{\jmath}^h)$ satisfies the continuous continuity equation \eqref{eq:CE}. The existence of such measures $\sigma_{K|L}^h\in\calM(\Omega; \R^d)$ was shown in \cite[Lemma~4.1]{hraivoronska2023diffusive}.

\medskip

The main theorems are the following.

\begin{mainthm}\label{th:EDP-convergence-tilt}
Let $\{(\calT_h,\Sigma_h)\}_{h>0}$ be a family of tessellations satisfying \eqref{ass:tessellation}, and assume \eqref{ass:pointy-potential} to hold for the interaction potential $W$. Further, let $\{(\rho^h,j^h)\}_{h>0}$ be a family of GGF-solutions \eqref{eq:SG} with initial data $\{\rho_{\text{in}}^h\}_{h>0}$ having $\sup_{h>0} \calE_h(\rho_{\text{in}}^h) < \infty$, such that there exists $\rho_{\text{in}}\in\text{dom}\,\calE$ with
\[
\frac{\dd\hat\rho_{\text{in}}^h}{\dd\calL^d} \to \frac{\dd\rho_{\text{in}}}{\dd\calL^d} \quad \text{in } L^1(\Omega)
 \quad \text{and } \quad \lim_{h\to 0} \calE_h(\rho_{\text{in}}^h) = \calE(\rho_{\text{in}}).
\]
Then there exists a (not relabelled) subsequence of admissible continuous reconstructions $\{(\hat{\rho}^h, \hat{\jmath}^h)\}_{h>0}$ and a limit pair $(\rho,j)$ such that
\begin{enumerate}[label=(\arabic*)]
    \item $(\rho,j)$ satisfies \eqref{eq:CE} with the density $u \coloneq \dd\rho/\dd\calL^d \in L^1((0, T)\times\Omega)$ and
        \begin{enumerate}[label=(\roman*)]
            \item  $\dd\hat{\rho}^h_t/\dd\calL^d \to u_t$  in $L^1(\Omega)$ for every $t\in [0, T]$;
            \item $\int_\cdot \hat{\jmath}^h_t \dd t \rightharpoonup^* \int_\cdot j_t \dd t$ weakly-$*$ in $\calM((0, T)\times \Omega)$.
        \end{enumerate}
    \item the following liminf estimate holds: For any $[s,t]\subset[0,T]$,
        $$
            \calI_{\epsilon}^{[s,t]}(\rho, j) \le \liminf_{h\to 0} \calI_{\epsilon,h}^{[s,t]}(\rho^h, j^h),
        $$
        where the energy-dissipation functional $\calI_\epsilon$ is given by
        $$
        \calI^{[s,t]}_\epsilon (\rho, j) = \int_s^t \bigl\{\calR(\rho_r,j_r) +  \calD_\epsilon(\rho_r)\bigr\} \dd r + \calE_\epsilon(\rho_t) - \calE_\epsilon(\rho_s),
    $$
    with the dissipation potential $\calR$ given in \eqref{eq:dissipation-potential} and Fisher information $\calD_\epsilon:\calP(\Omega)\to[0,+\infty]$,
    $$
        \calD_\epsilon(\rho) = 2\epsilon^2\int_\Omega \Bigl| \nabla \sqrt{u}\Bigr|^2\dd x + \epsilon\int_\Omega \nabla u \cdot \nabla \mathsf{Q}(\rho)\dd x + \frac{1}{2} \int_\Omega \lvert\nabla \mathsf{Q}(\rho)\rvert^2 \dd\rho
    $$
    if $\rho\ll \calL^d$ with $u=\dd\rho/\dd\calL^d$ and $+\infty$ otherwise. Recall that $\mathsf{Q}(\rho) = V + W\ast\rho$.
        
    \item $(\rho,j)$ is the gradient flow solution of \eqref{eq:drift-aggregation-diffusion-equation} with the energy-dissipation functional $\calI_\epsilon$.
    \end{enumerate}
\end{mainthm}

In Section~\ref{sec:vanishing-diffusion}, we fix a tessellation $(\calT^h, \Sigma^h)$ with some $h>0$ and consider the dependence of the discrete energy-dissipation functional
\begin{equation*}
     \calI_{\epsilon,h}^{[s,t]} (\rho^{h}, j^{h}) = \int_s^t \calR_{\epsilon,h}(\rho^{h}_r, j^{h}_r) + \calD_{\epsilon,h}(\rho^{h}_r) \dd r + \calE_{\epsilon,h}(\rho^{h}_t) - \calE_{\epsilon,h}(\rho^{h}_s),
\end{equation*}
on the diffusion coefficient $\epsilon > 0$. We have the following convergence statement.

\begin{mainthm}\label{th:EDP-vanishing-diffusion}
    Let $(\calT^h, \Sigma^h)$ be a non-degenerate tessellation with a fixed $h>0$. Let $\{ (\rho^{\epsilon,h}, j^{\epsilon,h}) \}_{\epsilon>0}$ be a family of GGF-solutions to \eqref{eq:SG} with initial data $\{\rho_{\text{in}}^{\epsilon,h}\}_{\epsilon>0}$ having $\sup_{\epsilon>0} \calE_{\epsilon,h}(\rho_{\text{in}}^{\epsilon,h}) < \infty$, such that there exists $\rho_{\text{in}}^h\in\text{dom}\,\calE_{\text{up},h}$ with
\[
	\rho_{\text{in}}^{\epsilon,h}(K) \to \rho_{\text{in}}^h(K) \quad \text{for every $K\in\calT^h$}
 \quad \text{and } \quad \lim_{\epsilon\to 0} \calE_{\epsilon,h}(\rho_{\text{in}}^{\epsilon,h}) = \calE_{\text{up},h}(\rho_{\text{in}}^h),
\]
where $\calE_{\text{up},h}:\calP(\calT^h)\to \R$ is given by
\[
	\calE_{\text{up},h}(\rho) = \sum_{K\in\calT^h} V^h_K \rho_K + \frac{1}{2} \sum_{(K,L)\in\calT^h\times\calT^h} W^h_{KL} \rho_K \rho_L.
\]
Then there exists a (not relabelled) subsequence of measure-flux pairs $\{ (\rho^{\epsilon,h}, j^{\epsilon,h}) \}_{\epsilon>0}$ and the limit pair $(\rho^{\text{up},h}, j^{\text{up},h})$ such that
    \begin{enumerate}[label=(\arabic*)]
        \item $(\rho^{\text{up},h}, j^{\text{up},h})$ satisfies \eqref{eq:CEh} and
        \begin{enumerate}[label=(\roman*)]
            \item $\rho^{\epsilon,h}_t \rightharpoonup \rho^{\text{up},h}_t$ weakly in $\calP(\calT^h)$ for all $t\in [0,T]$;
            \item $\int_{\cdot} j^{\epsilon,h}_t \dd t \rightharpoonup^* \int_{\cdot} j^{\text{up},h}_t \dd t$ weakly-$*$ in $\calM((0, T) \times \Sigma^h)$.
        \end{enumerate}
        \item the following liminf estimate holds: For any $[s,t]\subset[0,T]$,
        $$
            \calI_{\text{up},h}^{[s,t]}(\rho^{\text{up},h}, j^{\text{up},h}) \leq \liminf_{\epsilon \to 0} \calI_{\epsilon,h}^{[s,t]}(\rho^{\epsilon,h}, j^{\epsilon,h}),
        $$
        where the energy-dissipation functional $I_{\text{up},h}$ is given by
        \begin{equation}\label{eq:EDP-upwind}
            \calI_{\text{up},h}^{[s,t]}(\rho^{h}, j^{h}) \coloneq \int_s^t \bigl\{\calR_{\text{up},h} (\rho^{h}_r, j^{h}_r) + \calD_{\text{up},h} (\rho^{h}_r)\bigr\} \dd r + \calE_{\text{up},h} (\rho^{h}_t) - \calE_{\text{up},h} (\rho^{h}_s),
        \end{equation}
        with driving energy $\calE_{\text{up},h}$, dissipation potential
        $$
            \calR_{\text{up},h} (\rho^h, j^h) = \sum_{(K,L)\in\Sigma^h} \tau_{K|L}^h\left( u^h_K \left| \frac{j^{h,+}_{K|L}}{\tau_{K|L}^hu^h_K } \right|^2 + u^h_L \left| \frac{j^{h,-}_{K|L}}{\tau_{K|L}^hu^h_L } \right|^2 \right) ,
        $$
        and Fisher information
        $$
            \calD_{\text{up},h} (\rho^h) = \sum_{(K,L)\in\Sigma^h} \tau_{K|L}^h\left( u^h_K \left| \frac{q_{K|L}^{h,+}}{2} \right|^2 + u^h_L \left| \frac{q_{K|L}^{h,-}}{2} \right|^2 \right).
        $$ 
        
        \item $(\rho^{\text{up},h}, j^{\text{up},h})$ is the GGF-solution to the upwind scheme \eqref{eq:upwind-scheme}.
    \end{enumerate}
\end{mainthm}

In Section~\ref{sec:upwind-to-aggregation}, we make a first step towards a convergence result from the upwind scheme \eqref{eq:upwind-scheme} to the aggregation equation \eqref{eq:aggregation}. Here, we use a more restrictive class of tessellations, namely Cartesian grids. Yet, we allow for low regularity of the initial data $\rho_\text{in}$, which can be a measure.

\begin{mainthm}\label{th:upwind-to-aggregation}	
    Let $\{(\calT^h,\Sigma^h)\}_{h>0}$ be a family of Cartesian tessellations with edges of length $h>0$. Let the interaction potential $W$ satisfy \eqref{ass:smooth-potential}. Further, let $\{(\rho^h,j^h)\}_{h>0}$ be a family of GGF-solutions to the upwind scheme \eqref{eq:upwind-scheme} with initial data $\{\rho_{\text{in}}^{h}\}_{h>0}$ having $\sup_{h>0} \calE_{\text{up},h}(\rho_{\text{in}}^{h}) < \infty$, such that there exists $\rho_{\text{in}}\in\text{dom}\,\calE_{\text{agg}}$ with
\[
	\hat{\rho}^h_\text{in} \rightharpoonup ^*\rho_\text{in} \text{\;\;weakly-$*$ in $\calP(\Omega)$}
 \quad \text{and } \quad \lim_{h\to 0} \calE_{\text{up},h}(\rho_{\text{in}}^{h}) = \calE_{\text{agg}}(\rho_{\text{in}}),
\]
where $\calE_{\text{agg}}:\calP(\Omega)\to \R$ is given by
\[
	\calE_{\text{agg}}(\rho) = \int_\Omega V \dd\rho + \frac{1}{2} \int_\Omega (W*\rho) \dd\rho.
\]
Then there exists a (not relabelled) subsequence of admissible continuous reconstructions $\{(\hat{\rho}^h, \hat{\jmath}^h)\}_{h>0}$ and a limit pair $(\rho,j)$ such that
\begin{enumerate}[label=(\arabic*)]
    \item $(\rho, j)$ satisfies \eqref{eq:CE} and
        \begin{enumerate}[label=(\roman*)]
            \item $\hat{\rho}^h_t \rightharpoonup^* \rho_t$ weakly-$*$ in $\calP(\Omega)$ for any $t\in [0, T]$;
            \item $\int_\cdot \hat{\jmath}^h_t \dd t \rightharpoonup^* \int_\cdot j_t \dd t$ weakly-$*$ in $\calM((0, T)\times \Omega)$.
        \end{enumerate}
        
    \item the following liminf estimate holds for any $[s,t]\subset[0,T]$,
        $$
            \calI_{\text{agg}}^{[s,t]}(\rho, j) \le \liminf_{h\to 0} \calI_{\text{up},h}^{[s,t]}(\rho^h, j^h),
        $$
        where the energy-dissipation functional is given by
        \begin{equation}\label{eq:EDF-aggregation}
            \calI_{\text{agg}}^{[s,t]}(\rho, j) = \int_s^t \bigl\{ \calR(\rho_r,j_r) + \calD_{\text{agg}}(\rho_r)  \bigr\} \dd r + \calE_{\text{agg}}(\rho_t) - \calE_{\text{agg}}(\rho_s),
        \end{equation}
    with driving energy $\calE_{\text{agg}}$, dissipation potential $\calR$ given in \eqref{eq:dissipation-potential} and Fisher information 
    \[
    	\calD_{\text{agg}}(\rho) \coloneq \frac{1}{2} \int_\Omega | \nabla\mathsf{Q}(\rho) |^2 \dd\rho,\qquad \mathsf{Q}(\rho) = V + W\ast\rho.
    \]
        
    \item $(\rho,j)$ is the gradient flow solution to the aggregation equation \eqref{eq:aggregation}.
    \end{enumerate}
\end{mainthm}

Finally, and to close the commutative diagram in Figure~\ref{diagram:with-theorems}, we present the vanishing diffusion limit on the continuous level. 
\begin{mainthm}\label{th:vanishing-diffusion-continuous}
    Let the interaction potential $W$ satisfy \eqref{ass:smooth-potential}. Let $\{(\rho^\epsilon, j^\epsilon)\}_{\epsilon>0}$ be a family of the gradient flow solutions to the aggregation-diffusion equation \eqref{eq:drift-aggregation-diffusion-equation} the diffusion coefficients $\epsilon>0$ with initial data $\{\rho_{\text{in}}^{\epsilon}\}_{\epsilon>0}$ having $\sup_{\epsilon>0} \calE_{\epsilon}(\rho_{\text{in}}^{\epsilon}) < \infty$, such that there exists $\rho_{\text{in}}\in\text{dom}\,\calE_{\text{agg}}$ with
\[
	\rho^\epsilon_\text{in} \rightharpoonup ^*\rho_\text{in} \text{\;\;weakly-$*$ in $\calP(\Omega)$}
 \quad \text{and } \quad \lim_{\epsilon\to 0} \calE_{\epsilon}(\rho_{\text{in}}^{\epsilon}) = \calE_{\text{agg}}(\rho_{\text{in}}),
\]
Then there exists a limit pair $(\rho, j)$ and a (not relabelled) subsequence such that
    \begin{enumerate}[label=(\arabic*)]
        \item $(\rho, j)$ satisfies \eqref{eq:CE} and
        \begin{enumerate}[label=(\roman*)]
            \item $\rho^\epsilon_t \rightharpoonup^* \rho_t$ weakly-* in $\calP(\Omega)$ for any $t\in [0,T]$;
            \item $\int_. j^\epsilon_t \dd t \rightharpoonup^* \int_. j_t \dd t$ in $\calM((0, T)\times\Omega)$.
        \end{enumerate}
        \item the following liminf estimate holds for any $[s,t]\subset[0,T]$
        $$
            \calI_{\text{agg}}^{[s,t]} (\rho, j) \leq \liminf_{\epsilon\to 0} \calI_\epsilon^{[s,t]} (\rho^\epsilon, j^\epsilon),
        $$
        with $\calI_{\text{agg}}^{[s,t]}$ defined in \eqref{eq:EDF-aggregation}.
        \item $(\rho,j)$ is the gradient flow solution to the aggregation equation \eqref{eq:aggregation}.
    \end{enumerate}
\end{mainthm}

\section{Gradient Structures: Discrete and Continuous}\label{sec:gradient-structures}

This section is devoted to defining our notion of (generalized) gradient flow solution to each equation of interest. We begin with the continuous case in Section~\ref{sec:otto}, which is the well-known Otto-Wasserstein gradient structure (see \cite{ambrosio2008gradient} for a more extensive study on this). We then introduce, in a similar fashion to the continuous case, generalized gradient structures for general finite volume schemes in Section~\ref{sec:discrete-structures}, and proceed with providing two such structures for the Scharfetter--Gummel scheme in Section~\ref{sec:gf-SG}. We end this section with a summary of the discrete structure we consider in the rest of the article.

\subsection{Otto-Wasserstein gradient structure for diffusion-type equations}\label{sec:otto}

\begin{definition}\label{def_CE} A pair $(\rho, j)$ is said to be in $\mathcal{CE}(0, T)$ if
\begin{itemize}
    \item $\rho\in\calC([0,T];\calP(\Omega))$ is a curve of non-negative finite Radon measures defined on $\Omega$, and
    \item $j=(j_t)_{t\in[0,T]}\subset \calM(\Omega;\R^d)$ is a measurable family of {\em fluxes} with finite {\em action}
    \[
        \int_0^T \int_\Omega \left|\frac{\dd j_t }{\dd\rho_t} \right|^2 \dd\rho_t \dd t < \infty,
    \]
\end{itemize}
    satisfy the continuity equation \eqref{eq:CE} in the following sense: For any $[s,t]\subset [0,T]$,
    \begin{equation}\label{eq_CE_distr}
       \langle \varphi,\rho_t\rangle - \langle \varphi,\rho_s\rangle = \int_s^t \langle \nabla \varphi, j_r\rangle \dd r \qquad\text{for all $\varphi\in \calC_c^1(\R^d)$.}
    \end{equation}
\end{definition}

\begin{remark}\label{rem_W2AC}
    It is known that if $\rho$ solves \eqref{eq:CE} with finite action, then $\rho$ is an absolutely continuous curve in $\calP(\Omega)$ w.r.t.\ the 2-Wasserstein distance \cite[Chapter 8]{ambrosio2008gradient}. 
\end{remark}

\begin{definition}[GF solutions]\label{def_GF_solution}
    A curve $\rho\in\calC([0,T];\calP(\Omega))$ is said to be an $(\calE, \calR, \calR^*)$-gradient flow solution of \eqref{eq:drift-aggregation-diffusion-equation} or \eqref{eq:aggregation} with initial data $\rho_{\text{in}}\in \calP(\Omega)\cap\text{dom}(\calE)$ if
    \begin{enumerate}[label=(\roman*)]
        \item $\rho_0=\rho_{\text{in}}$ in $\calP(\Omega)$;
        \item there is a measurable family $j=(j_t)_{t\in[0, T]} \subset \calM(\Omega; \R^d)$ such that $(\rho, j) \in \mathcal{CE}(0, T)$ with
        $$
            \int_s^t \int_\Omega \calR(\rho_r, j_r) + \calD(\rho_r) \dd r + \calE(\rho_t) = \calE(\rho_s) \quad \text{for all } [s,t]\subset [0,T],
        $$
        where
        \[
            \calD(\rho) \coloneq \inf\Bigl\{ \liminf_{n\to\infty} \calR^*(\rho_n,-\nabla \calE'(\rho_n)) \,:\, \rho_n\rightharpoonup \rho\;\text{weakly in $\calP(\Omega)$},\; \sup\nolimits_{n\ge 0} \calE(\rho_n) <\infty\Bigr\},
        \]
        i.e.\ $\calD$ is a lower-semicontinuous envelope of $\rho\mapsto \calR^*(\rho,-\nabla \calE'(\rho))$;
        \item the following chain rule inequality holds:
        $$
            -\frac{\dd}{\dd t} \calE(\rho_t) \le \calR(\rho_t, j_t) + \calD(\rho_t)\qquad\text{for almost every $t\in(0,T)$.}
        $$
    \end{enumerate}
\end{definition}

\subsection{Generalized gradient structure for finite volume schemes}\label{sec:discrete-structures}
We take the point of view that finite volume schemes can be seen as random walks on the graph induced by tessellations. Hence, we consider a random walk on a graph that corresponds to a tessellation $(\calT^h, \Sigma^h)$. Given an initial law $\rho_0^h = \rho_{\text{in}}^h\in\calP(\calT^h)$, the time marginal law of a random walk satisfies the forward Kolmogorov equation
\begin{equation}\label{eq:Kolmogorov}\tag{{\bf FKE}$_h$}
    \partial_t \rho_t^h = Q^*_h \rho_t^h,
\end{equation}
where $Q^*_h$ is the dual of the generator $Q_h$ defined for all bounded functions $\varphi \in \calB(\calT^h)$ as
$$
    (Q_h \varphi) (K) = \sum_{(K,L)\in\Sigma^h} (\dnabla \varphi) (K,L)\,\kappa^h_{K|L}, \qquad K \in \calT^h,
$$
where $\kappa : \Sigma^h \to \R_+$ is a bounded jump kernel. We restrict ourselves to random walks satisfying detailed balance, i.e.\ random walks admitting a stationary measure $\pi^h\in\calP(\calT^h)$ such that
\begin{equation}\label{eq:detailed-balance}
    \pi^h_K \kappa^h_{K|L} = \pi^h_L \kappa^h_{L|K} \qquad \text{for all } (K, L) \in \Sigma^h.
\end{equation}
We note that the detailed balance implies, by the ergodic theorem for continuous-time Markov chains, the uniqueness of the stationary measure $\pi^h$ (see, for instance, \cite[6.10 (15) Theorem]{grimmett2020probability}).

\begin{definition}\label{def_CE_discrete} 
A pair $(\rho^h, j^j)$ is said to be in $\mathcal{CE}_h(0, T)$ if
\begin{itemize}
    \item $\rho^h\in\calC([0,T];\calP(\calT^h))$ is a curve of finite measures defined on the graph $\calT^h$, and
    \item $j^h = (j_t^h)_{t\in[0,T]} \subset \calM(\Sigma^h)$ is a measurable family of {\em discrete fluxes} with finite action
    \[
    	\int_0^T | j^h_t | ( \Sigma^h) \dd t < \infty,
    \]
\end{itemize}
satisfy the discrete continuity equation \eqref{eq:CEh} in the following sense: For any $[s, t] \subset [0, T]$,
    \begin{equation}\label{eq:theory:CE-discrete-distr}
        \sum_{K\in \calT^h} \varphi^h_K \rho^h_K(t) - \sum_{K\in \calT^h} \varphi^h_K \rho^h_K(s) = \int_{s}^{t} \sum_{(K,L)\in \Sigma^h} (\dnabla \varphi^h)(K, L)\,j^h_{K|L}(r) \dd r\qquad \text{for all $\varphi^h\in \calB(\calT^h)$.}
    \end{equation}
\end{definition}

\begin{definition}[GGF solutions]\label{def_GGF_solution}
    A curve $\rho^h\in\calC([0,T]; \calP(\calT^h))$ is an $(\calE_h, \calR_h, \calR_h^*)$-generalized gradient flow solution of \eqref{eq:Kolmogorov} with initial data $\rho_{\text{in}}^h\in \calP(\calT^h)\cap \text{dom}(\calE_h)$ if
    \begin{enumerate}[label=(\roman*)]
        \item $\rho^h_0= \rho^h_{\text{in}}$ in $\calP(\calT^h)$;
        \item there is a measurable family $j^h=(j^h_t)_{t\in[0, T]} \subset \calM(\Sigma^h)$ such that $(\rho^h, j^h) \in \mathcal{CE}_h(0, T)$ with
        $$
            \int_s^t \calR_h(\rho^h_r, j^h_r) + \calD_h(\rho^h_r) \dd r + \calE_h(\rho^h_t) = \calE_h(\rho^h_s) \quad \text{for all } [s,t]\subset [0,T];
        $$
        where
        \[
            \calD_h(\rho^h) \coloneq \inf\Bigl\{ \liminf_{n\to\infty} \calR_h^*(\rho^h_n,-\dnabla \calE'_h(\rho^h_n)) \,:\, \rho^h_n\rightharpoonup \rho^h\;\text{weakly in $\calP(\calT^h)$}, \; \sup\nolimits_{n\ge 0} \calE_h(\rho^h_n) <\infty \Bigr\},
        \]
        i.e.\ $\calD_h$ is a lower-semicontinuous envelope of $\rho^h\mapsto \calR_h^*(\rho^h,-\dnabla \calE'_h(\rho^h))$.
        
        \item the chain rule inequality holds, i.e.\
        \begin{align*}
            -\frac{\dd}{\dd t}\calE_h(\rho^h_t) \le \calR_h(\rho_t^h,j_t^h) + \calD_h(\rho_t^h)\qquad\text{for almost every $t\in (0,T)$}.
        \end{align*}
    \end{enumerate}
\end{definition}

\subsection{A new gradient structure for the Scharfetter--Gummel scheme}\label{sec:gf-SG}

Since the Scharfetter--Gummel scheme is a finite volume scheme, it defines a random walk on the state space $\calT^h$. Moreover, \eqref{eq:finite-volume-CE} possesses a generalized gradient flow structure if the Scharfetter--Gummel flux \eqref{eq:SG-flux} can be recast as the force-flux relation \eqref{eq:FFh} induced by a dual dissipation potential, i.e.\ if we can express the discrete flux for all $K\in\calT^h$ and $(K, L)\in\Sigma^h$ as
\begin{equation}\label{eq:flux-force-SG}
    \calJ^{h,\rho}_{K|L} = D_2\calR_h^* \bigl(\rho^h,-\dnabla \calE'_{\epsilon,h}(\rho^h) \bigr) (K,L)
\end{equation}
with an appropriate dual dissipation potential $\calR_{\epsilon,h}^*$ and the  driving energy $\calE_{\epsilon,h}$ defined in~\eqref{eq:energy-with-interaction}.

In this section, we introduce the tilt-independent gradient structure encoding the flux-force relationship~\eqref{eq:flux-force-SG}, which we will study in this manuscript and is one of the main contributions of this article. The gist of this structure is that the dual dissipation potential does not depend on potentials $V^h$ and $W^h$ and more importantly also does not degenerate for small diffusivity $\epsilon\ll 1$. 

Based on the cell formula~\eqref{eq:def:cell-problem}, the Scharfetter--Gummel flux in~\eqref{eq:SG-flux} was recast as a kinetic relation for a general force $\xi^h\in\calB(\Sigma^h)$ in~\cite[(2.23)]{schlichting2022scharfetter}, for which we can derive a suitable dual dissipation potential $\calR_{\epsilon,h}^*$. For doing so, we notice that along a solution of the scheme, we have the force
\begin{equation}
\xi^h_{K|L} = - \dnabla \calE'_{\epsilon,h}(\rho^h)(K,L) = - \bra[\bigg]{\epsilon \log \frac{u^h_L}{u^h_K} + q_{K|L}^h }, \qquad (K,L)\in\Sigma^h,
\end{equation}
and therefore, we find the relation 
\begin{equation}\label{eq:qKL:energy}
q_{K|L}^h = \epsilon \log \frac{u^h_K}{u^h_L} - \xi^h_{K|L} = \epsilon \left( \log\bra[\big]{u^h_K e^{-\xi^h_{K|L} / 2\epsilon}} - \log\bra[\big]{u^h_L e^{\xi^h_{K|L} / 2\epsilon} } \right).
\end{equation}
By substituting this relation into~\eqref{eq:SG-flux}, we arrive, after some simplifications, at the identity
\begin{equation}\label{eq:SG:KR}
\calJ^{h,\rho}_{K|L}
= \epsilon \sinh\bra*{\frac{\xi^h_{K|L}}{2\epsilon}} \Lambda_H\bra*{u_K^h e^{-\frac{\xi^h_{K|L}}{2\epsilon}},u_L^h e^{\frac{\xi^h_{K|L}}{2\epsilon}}} |K|
\stackrel{!}{=} D_2 \calR_{\epsilon,h}^* (\rho^h, \xi^h) (K,L)  ,
\end{equation}
where the last equality is a requirement for the new dual dissipation potential and $\Lambda_H$ denotes the harmonic-logarithmic mean defined in~\eqref{eq:def:HarmLogMean}. From the kinetic relation~\eqref{eq:SG:KR} relating the force $\xi^h$ with the flux, one obtains the dissipation potential $\calR_h^*$ as given in~\eqref{eq:tilted-Rstar} with the function $\alpha_\epsilon^*$ in~\eqref{eq:def:alpha*}, by simply integrating over the force. Although $\alpha_\epsilon^*$ is only defined as an integral, it has many beneficial properties, which are essential for the analysis that we collect Lemma~\ref{lem:properties-alpha} in Appendix~\ref{appendix:function-alpha}.

Since the derivation of the kinetic relation~\eqref{eq:SG:KR} might seem to look ad-hoc, we provide a different derivation of the dissipation potential $\calR_{\epsilon,h}^*$ from a `cosh' dissipation potential $\overline\calR_{\epsilon,h}^*$ similar to the one used in~\cite{hraivoronska2023diffusive}. To do so, we perform a \emph{de-tilting} technique from~\cite[Remark 1.17]{peletier2023cosh} and provide the details in Appendix~\ref{appendix:cosh-tilting}. This way, we can show that we arrived at a tilt-independent gradient structure for the Scharfetter--Gummel scheme (see Lemma~\ref{lem:flux-force-for-tilt}).

\section{Variational Convergence for the Tilt-Independent Structure}\label{sec:EDP-convergence-tilt}

The strategy of proving the discrete-to-continuum EDP convergence comprises two main steps:
\begin{enumerate}
    \item Prove compactness for the family of the GGF solutions $(\rho^h, j^h)$ of \eqref{eq:finite-volume-CE} defined in Defintion~\ref{def_GGF_solution}. This allows us to extract a subsequence converging to a limiting pair $(\rho, j)$. 
    \item Prove liminf inequalities for all the functionals in the energy-dissipation functional $\calI_h$ and recover a limiting energy-dissipation functional $\calI$:
    $$
        \calI(\rho, j)\le \liminf_{h\to 0} \calI_h(\rho^h, j^h).
    $$
\end{enumerate}
In Section~\ref{sec:compactness-SG}, we prove the compactness results required by (1). To establish the liminf inequality for $\calI_h$ from (2), the main effort relates to the Fisher information. Thus, Section~\ref{sec:gamma-convergence-fisher-SG} is dedicated to the $\varGamma$-convergence of the Fisher information. We conclude with the proof of Theorem~\ref{th:EDP-convergence-tilt} in Section~\ref{sec:EDP-summary}.

\subsection{Compactness}\label{sec:compactness-SG} We consider a family $\{(\rho^h,j^h)\}_{h>0}$ of $(\calE_h, \calR_h, \calR_h^*)$-generalized gradient flow solutions to \eqref{eq:SG}, where the corresponding functionals are defined in \eqref{eq:energy-with-interaction}, \eqref{eq:tilted-R}, and \eqref{eq:tilted-Rstar} respectively. We also assume the initial data $\{\rho^h_\text{in}\}_{h>0}$ to be well-prepared. We set $J^h\coloneq\int_\cdot \hat\jmath_t^h\dd t$.

\begin{lemma}[Compactness for flux]\label{lemma:compactness-for-flux-tilted}
    The family $\{ J^h \}_{h>0}$ is weakly-$*$ compact in $\calM([0,T]\times\Omega; \R^d)$ and the family $\{ t \mapsto | \hat{\jmath}^h_t | (\Omega) \}_{h>0}$ is equi-integrable. 
    
    In particular, there exists a Borel family $(j_t)_{t\in[0,T]}\subset \calM(\Omega;\R^d)$ such that
    \[
        J^h=\int_{\cdot}\,\hat{\jmath}_t^h\dd t \rightharpoonup^* \int_{\cdot}\, j_t\dd t\quad\text{weakly-$*$ in $\calM([0,T]\times\Omega; \R^d)$}
    \]
    for a (not relabelled) subsequence.
\end{lemma}
\begin{proof} The proof is similar to the proof of the related compactness statement for the `cosh' gradient structure \cite[Lemma~4.4]{hraivoronska2023diffusive}. For completeness, we present the full proof here.

For almost every $t\in(0,T)$, the reconstruction of the flux is defined as
$$
    \hat{\jmath}_t^h = \sum_{(K,L) \in \Sigma^h} j^h_{K|L}(t) \,\sigma_{K|L}^h, 
$$
with $\sigma_{K|L}^h\in\calM(\Omega; \R^d)$ such that $|\sigma_{K|L}^h|(\Omega) \le 2dh$. The existence of the required $\sigma_{K|L}^h$ is proven in \cite[Lemma~4.1]{hraivoronska2023diffusive}. We begin by noticing that for almost every $t\in(0,T)$ and any $\beta\in\R$,
\begin{align*}
	\calR_{\epsilon,h}(\rho_t^h, j_t^h) &= \sup_{\xi^h\in \calB(\Sigma^h)} \left\{ \sum_{(K,L)\in\Sigma^h} \xi_{K|L}^h j_{K|L}^h(t) - 2\sum_{(K,L)\in\Sigma^h} \tau_{K|L}^h \,\alpha_\epsilon^* \bra*{ u^h_K(t), u^h_L(t), \frac{\xi^h_{K|L}}{2} } \right\} \\
	&\ge \beta\, |\hat{\jmath}_t^h|(\Omega) - 2\sum_{(K,L)\in\Sigma^h} \tau_{K|L}^h \,\alpha_\epsilon^* \bra*{ u^h_K(t), u^h_L(t), \frac{\beta\,\text{sign}(j_{J|L}^h)|\sigma_{K|L}^h|(\Omega)}{2}},
\end{align*}
where we simply take $\xi_{K|L}^h=\beta\, \text{sign}(j_{K|L}^h)|\sigma_{K|L}^h|(\Omega)$. Due to Lemma~\ref{lem:properties-alpha}\ref{alpha-cosh}, we obtain
\[
	\alpha_\epsilon^* \bra*{ u^h_K(t), u^h_L(t), \frac{\beta\,\text{sign}(j_{K|L}^h)|\sigma_{K|L}^h|(\Omega)}{2}} \le \frac{1}{4}\sqrt{u_K^h(t)\,u_L^h(t)}\,\Psi_\epsilon^*\bra*{\beta|\sigma_{K|L}^h|(\Omega)},
\]
where $\Psi_\epsilon^*(s) = 4 \epsilon^2 (\cosh(s/2 \epsilon) - 1)$, and consequently,
\[
	\calR_{\epsilon,h}(\rho_t^h, j_t^h) \ge \beta\, |\hat{\jmath}_t^h|(\Omega) - \frac{C_\tau}{2h^2} \Psi_\epsilon^*\bigl(2\beta dh\bigr),
\]
with the constant $C_\tau>0$ as defined in \eqref{eq:tau:bound}.

Since $\Psi_\epsilon^*(s r) \leq s^2 r^2 \cosh (sr / 2\epsilon)$ for $s,r\in\R$, which can be shown using the Taylor series for the $\cosh$ function, for $|r| \leq 1$, we obtain $\Psi_\epsilon^*(s r) \le r^2 \widetilde\Psi_\epsilon^*(s)$ with $\widetilde\Psi_\epsilon^*(s) = s^2 \cosh (s / 2\epsilon)$. Note that $\widetilde\Psi_\epsilon^*$ is a convex function having superlinear growth at infinity. Minimizing the previous inequality over $\beta\in\R$, we obtain
\[
	\calR_{\epsilon,h}(\rho_t^h, j_t^h) \ge \frac{C_\tau}{2} \sup_{\beta\in\R} \left\{ \beta\, \frac{|\hat{\jmath}_t^h|(\Omega)}{d C_\tau} - \widetilde\Psi_\epsilon^*(\beta )\right\} = \frac{C_\tau}{4} \widetilde\Psi_\epsilon\bra*{\frac{|\hat{\jmath}_t^h|(\Omega)}{d C_\tau}},
\]
where $\widetilde\Psi_\epsilon$ is the Legendre dual of $\widetilde\Psi_\epsilon^*$, which is again a convex function having superlinear growth. Since $(j_t^h)_{t\in[0,T]}$ has uniform-in-$h$ finite action, we then obtain
\[
	\sup_{h>0}\int_0^T \widetilde\Psi_\epsilon\bra*{\frac{|\hat{\jmath}_t^h|(\Omega)}{d C_\tau}}\,\dd t \le \frac{2}{C_\tau}\sup_{h>0}\int_0^T\calR_{\epsilon,h}(\rho_t^h, j_t^h)\,\dd t \le \frac{2}{C_\tau}\sup_{h>0} \calE_{\epsilon,h}(\rho^h_\text{in})<\infty,
\]
therewith deducing the equi-integrability of the family $\{ t \mapsto | \hat{\jmath}^h_t | (\Omega) \}_{h>0}$. 

One also easily deduces from the previous inequality that
\[
	\sup_{h> 0}|J^h|([0,T]\times\Omega)  \le 2 d \bra*{\sup_{h> 0}\int_0^T \calR_{\epsilon,h}(\rho_t^h, j_t^h)\dd t + \frac{C_\tau T}{2} \widetilde\Psi_\epsilon^*(1)} <\infty,
\]
which implies the existence of some $J\in\calM((0, T)\times\Omega)$ and some subsequence for which $J^h \rightharpoonup^* J$ weakly-$*$ in $\calM((0, T)\times\Omega)$. Finally, Due to the equi-integrablity of $\{ t \mapsto | \hat{\jmath}^h_t | (\Omega) \}_{h>0}$, we deduce that $J$ has the representation $J = \int_\cdot j_t\dd t$ for a Borel family $(j_t)\subset\calM(\Omega; \R^d)$.
\end{proof}

\begin{lemma}\label{lemma:BV-bound}
    Let $\rho^h \in \calP(\calT^h)$ with $\calD^0_h(\rho^h) < \infty$, where
    \begin{align}\label{eq:def:D0}
    	\calD_{\epsilon,h}^0 (\rho^h) \coloneq 2\sum_{(K,L)\in\Sigma^h} \beta_\epsilon ( u^h_K, u^h_L )\, \tau_{K|L}^h,\qquad u_K^h = \frac{\rho_K^h}{|K|}.
    \end{align}
    Then the reconstructed density $\hat{u}^h$ satisfies
    $$
        | D \hat{u}^h | (\Omega) \leq C \sqrt{\calD^0_{\epsilon,h}(\rho^h)},
    $$
    for some constant $C>0$ independent of $h>0$.
\end{lemma}
\begin{proof}
    Since $\hat{u}^h$ is a piece-wise constant function on the cells $\calT^h$, one can show that
    $$
        D\hat{u}^h = \sum_{(K, L)\in\Sigma^h} u^h_K n_{K|L} \calH^{d-1}|_{(K|L)} = \frac{1}{2}\sum_{(K, L)\in\Sigma^h} (u^h_K - u^h_L) n_{K|L} \calH^{d-1}|_{(K|L)}.
    $$
    Therefore, using the Cauchy-Schwarz inequality and Lemma~\ref{lem:properties-alpha}\ref{item:beta:bound} yields
    \begin{align*}
        |D\hat{u}^h| (\Omega) &\leq \frac{1}{2}\sum_{(K, L)\in\Sigma^h} |u^h_K - u^h_L| |(K|L)|
        \leq \frac{1}{2}\sum_{(K, L)\in\Sigma^h} |u^h_K - u^h_L| h \tau_{K|L}^h \\
        &\leq \Bigg( \sum_{(K,L)\in\Sigma^h} \frac{\left|u^h_L - u^h_K\right|^2}{u^h_L + u^h_K} \tau_{K|L}^h \Bigg)^{1/2} \bra*{ \sum_{(K,L)\in\Sigma^h} (u^h_K + u^h_L)\, h^2 \tau_{K|L}^h }^{1/2} 
        \leq C \sqrt{\calD_{\epsilon,h}^0(\rho^h)},
    \end{align*}
    for some constant $C>0$ independent of $h>0$.
\end{proof}

With Lemma~\ref{lemma:compactness-for-flux-tilted} and Lemma~\ref{lemma:BV-bound} at hand, we can prove the strong compactness result.
\begin{proposition}[Strong compactness]\label{th_compactness}
Let the family of curves $\{\rho^h\}_{h>0}$ be the GGF-solutions of \eqref{eq:SG} with $(\calE_h, \calR_h, \calR_h^*)$ defined in \eqref{eq:energy-with-interaction}, \eqref{eq:tilted-R}, and \eqref{eq:tilted-Rstar} respectively. Let $ \sup_{h>0} \calE_h(\rho^h_\text{in}) < \infty$. Then there exists $u \in L^1( (0, T); L^1(\Omega))$ and a (not relabelled) subsequence such that 
\[
    \hat{u}^h_t \to u_t \quad\text{in $L^1(\Omega)$\; for almost every $t\in(0,T)$.}
\]
\end{proposition}
The proof of the proposition can be found in \cite[Theorem~4.8]{hraivoronska2023diffusive}.

\subsection{\texorpdfstring{$\varGamma$}{Γ}-convergence of the Fisher information}\label{sec:gamma-convergence-fisher-SG}

The aim of this section is to prove a $\varGamma$-convergence result for the discrete Fisher information $\rho^h \mapsto \calD_{\epsilon,h}(\rho^h)\coloneq\calR^*_{\epsilon,h}(\rho^h,-\dnabla \calE'_{\epsilon,h}(\rho^h))$, where
\[
-\dnabla \calE'_{\epsilon,h}(\rho^h)(K,L) = 2 \epsilon \log\sqrt{u^h_K / u^h_L} - q_{K|L}^h.
\] 
It will be crucial that we have the decomposition of $\alpha_\epsilon^*$ from Lemma~\ref{lem:properties-alpha}\ref{alpha-decomposition} to get the representation of $\calD_h$ as the sum of three terms
\begin{align}\label{eq:tilted-Fisher}
	\calD_{\epsilon,h}(\rho^h) &= \calD_{\epsilon,h}^0(\rho^h) + \calD_{\epsilon,h}^1 (\rho^h) + \calD_{\epsilon,h}^2 (\rho^h),
\intertext{where $\calD_{\epsilon,h}^0$ is given in~\eqref{eq:def:D0} and}
	\calD_{\epsilon,h}^1 (\rho^h) &\coloneq \frac{\epsilon}{2} \sum_{(K, L)\in\Sigma^h} (u^h_L - u^h_K) \, q_{K|L}^h \tau_{K|L}^h, \notag\\
	\calD_{\epsilon,h}^2 (\rho^h) &\coloneq \frac{1}{2} \sum_{(K, L)\in\Sigma^h} |q_{K|L}^h|^2 \, \mathbb{h}_\epsilon (u^h_K, u^h_L, q_{K|L}^h) \,\tau_{K|L}^h. \notag
\end{align}
This representation resembles the expansion of the continuous counterpart. Indeed, we expect the limit functional to be
\begin{align*}
    \calD_\epsilon(\rho) &= \calR^*\bra[\big]{\rho, -\nabla (\epsilon \log u + \mathsf{Q}(\rho) )}
    = \frac{\epsilon^2}{2} \int \abs*{\nabla \log \bra*{u e^{\mathsf{Q}(\rho)/\epsilon}} }^2 \dd\rho \\
    &= 2\epsilon^2 \int \abs*{ \nabla \sqrt{u} }^2\dd x + \epsilon\int \nabla u \cdot \nabla \mathsf{Q}(\rho)\dd x + \frac{1}{2} \int \abs*{ \nabla \mathsf{Q}(\rho) }^2 u\dd x \\
    &\eqcolon \calD_\epsilon^0(\rho) + \calD_\epsilon^1(\rho) + \calD^2(\rho),
\end{align*}
where we use the notation $\mathsf{Q}(\rho) = V + W*\rho$ as in the introduction.

\medskip
The main result of this section is the following theorem.
\begin{theorem}\label{th:convergence-tilted-Fisher} Assume that a family of tessellations $\{(\calT^h, \Sigma^h)\}_{h>0}$ satisfies the orthogonality \eqref{ass:orthogonality}.
Up to passing to a subsequence, the family of functionals $\{\calD_{\epsilon,h}\}_{h>0}$ has a $\varGamma$-limit $\calD_\epsilon$ w.r.t.\ the $L^2$-topology taking the form
\begin{equation*}
    \calD_\epsilon(\rho) = \begin{cases}\displaystyle
       2\epsilon^2 \int_\Omega \left| \nabla \sqrt{u}\right|^2 \dd x + \epsilon \int_\Omega \nabla u \cdot \nabla \mathsf{Q}(\rho)\dd x + \frac{1}{2} \int_\Omega \left| \nabla \mathsf{Q}(\rho) \right|^2 \dd\rho  & \text{if } \sqrt{u} \in H^1(\Omega), \\
        +\infty & \text{otherwise.}
    \end{cases}
\end{equation*}  
\end{theorem}

The proof of Theorem~\ref{th:convergence-tilted-Fisher} consists of the $\varGamma$-convergence result for $\calD_{\epsilon,h}^0$ and continuous convergence results for $\calD_{\epsilon,h}^1$ and $\calD_{\epsilon,h}^2$. Although we use the orthogonality assumption \eqref{ass:orthogonality} to get the complete result, the convergence of $\calD_{\epsilon,h}^0$ and $\calD_{\epsilon,h}^2$ can be established without \eqref{ass:orthogonality} at the cost of the tensor $\T$ appearing in the limit. Unfortunately, it is not clear how to identify the limit of $\calD_{\epsilon,h}^1$ without \eqref{ass:orthogonality}.

We begin with $\calD_{\epsilon,h}^0$. According to Lemma~\ref{lem:properties-alpha}\ref{item:beta:bound} the function $\beta$ satisfies the following bounds
\begin{equation}\label{eq:beta:bound}
    \frac{1}{8} \frac{(a - b)^2}{\tfrac{a + b}{2}}\leq \beta_\epsilon(a, b) \leq \frac{1}{2} (\sqrt{a} - \sqrt{b})^2 \qquad\text{for $a,b>0$}.
\end{equation}
The appearance of such bounds is possible to understand intuitively by noting that in the continuous setting, thanks to the chain rule the following two formulations are equivalent
\begin{align*}
	\frac{1}{8}  \frac{|\nabla u |^2 }{u} 
	=
	 \frac{1}{2} \big| \nabla \sqrt{u} \big|^2  \qquad\text{for $\sqrt{u}\in H^1(\Omega)$}.
\end{align*}
We now recognize the lower bound for $\beta_\epsilon$ as a discretization for the second formulation. We can also expect that \eqref{eq:general-functional} has the same $\varGamma$-limit as the quadratic functional. 

The proof of $\varGamma$-convergence for $\calD_{\epsilon,h}^0$ follows the localization method. The corresponding theory is covered in \cite[Chapter~16-20]{maso1993introduction}, and for the application of the localization method in the setting close to ours, see \cite{alicandro2004general, forkert2022evolutionary, hraivoronska2023diffusive}. The method is based on considering the localized version of the functional $\calD_{\epsilon,h}^0$ restricted to an open set $A\subset\Omega$
\begin{equation}\label{eq:general-functional}
    \calF_{\epsilon,h} (v^h, A) \coloneq \sum_{(K, L)\in\Sigma^h|_A} \beta_\epsilon \left( (v^h_K)^2, (v^h_L)^2 \right) \tau_{K|L}^h,
\end{equation}
where $\Sigma^h|_A \coloneq \left\{ (K, L) \in \Sigma^h:\, K,L\in\calT^h|_A \right\}$ and $\calT^h|_A \coloneq \left\{ K \in \calT^h : K\cap A \neq \emptyset \right\}$. 

We define for any open set $A\subset\Omega$
\begin{equation}\label{eq:sup-functional}
    \calF_{\epsilon,\sup}(v, A) \coloneq \varGamma\text{-}\limsup_{h\to 0} \calF_{\epsilon,h} (v, A) = \inf \left\{ \limsup_{h\to 0} \calF_{\epsilon,h} (v_h, A) ~: ~ v_h \to v \right\}.
\end{equation}
In the next lemma, we summarize the properties of $\calF_{\epsilon,\sup}$, which is necessary to apply the representation theorem from \cite[Theorem~2]{bouchitte2002global}. Specifically, we prove that $\calF_{\epsilon,\sup}$ is an inner regular, subadditive, and local functional satisfying the lower and upper Sobolev bounds. The proof follows very closely the strategy from~\cite{hraivoronska2023diffusive} and leverages the quadratic comparison of the function $\beta_\epsilon$ noted above in~\eqref{eq:beta:bound}.

\begin{lemma}\label{lemma:sobolev-bounds-tilted} The functional $\calF_{\epsilon,\sup}$ defined in \eqref{eq:sup-functional} has the following properties
\begin{enumerate}[label=(\roman*)]
    \item Inner regularity: For any $v\in H^1(\Omega, \mu)$ and for any $A\in\calO$ it holds that
    $$
        \sup_{A'\ssubset A} \calF_{\epsilon,\sup}(v, A') = \calF_{\epsilon,\sup}(v, A);
    $$
    \item Subadditivity: For any $v\in H^1(\Omega, \mu)$ and for any $A, A', B, B' \in \calO$ such that $A'\ssubset A$ and $B'\ssubset B$ it holds that:
    $$
        \calF_{\epsilon,\sup}(v, A'\cup B') \leq \calF_{\epsilon,\sup}(v, A) + \calF_{\epsilon,\sup}(v, B);
    $$
    \item Locality: For any $A\in \calO$ and any $v, \psi \in H^1(\Omega, \mu)$ such that $v=\psi$ $\mu$-a.e. on $A$ there holds
    $$
        \calF_{\epsilon,\sup}(v, A) = \calF_{\epsilon,\sup}(\psi, A).
    $$
    \item Sobolev bounds: For any $v\in H^1(\Omega)$ and an open set $A\subset\Omega$
    $$
        c\int_A \left| \nabla v \right|^2 \dd x \leq \calF_{\epsilon,\sup}(v, A) \leq C\int_A \left| \nabla v \right|^2 \dd x,
    $$
for some $c, C>0$ independent of $v$ and $A$.
\end{enumerate}
\end{lemma}
\begin{proof} In this proof, we drop the subscript $\epsilon$ to simplify the notation.

The proof of (\textit{i}), (\textit{ii}), and (\textit{iii}) follows very closely the corresponding proofs in \cite{alicandro2004general, forkert2022evolutionary, hraivoronska2023diffusive}.

We now turn to the proof of (\textit{iv)}. First, we consider the upper bound. By the upper bound shown in Lemma~\ref{lem:properties-alpha}(f), it holds that
    $$
        \calF_{\sup}(v, A) \leq \frac{\epsilon^2}{2} \sum_{(K,L)\in\Sigma^h|_A} \left( v^h_L - v^h_K \right)^2 \tau_{K|L}^h.
    $$
    Then the required upper bound follows from \cite[Lemma~5.8]{hraivoronska2023diffusive}.

Next, we show the lower bound from (\textit{iv}). Let $\{v_h\}_{h>0}\in L^2(\Omega)$ be a sequence with $v_h \to v$ in $L^2(\Omega)$ such that
    $$
        \calF_{\sup}(v, A) = \limsup_{h\to 0} \calF_h(v_h, A).
    $$
    We fix an arbitrary $r>0$ and denote $A_r \coloneq \left\{ x\in A \,:\, \text{dist}(x, \partial A) > r \right\}$. Let $\eta\in\R^d$ be such that $|\eta| < r$, then by the argument as in \cite[Lemma~5.17]{hraivoronska2023diffusive}.
    $$
        \int_{A_r} | v_h(x+\eta) - v_h(x) |^2 \dd x \leq C |\eta|^2 \sum_{(K, L)\in\Sigma^h|_{A_r}} \left| v_h(L) - v_h(K) \right|^2 \tau_{K|L}^h.
    $$
    Using the lower bound for $\beta_\epsilon$ from Lemma~\ref{lem:properties-alpha}\ref{item:beta:bound} 
    \begin{align*}
        \calF(v_h,A_r) =\sum_{(K, L)\in\Sigma^h|_{A_r}} \beta_\epsilon \bigl((v_h(K))^2, (v_h(L))^2\bigr)\, \tau_{K|L}^h \geq \frac{\epsilon^2}{4} \sum_{(K, L)\in\Sigma^h|_{A_r}} | v_h(L) - v_h(K) |^2 \tau_{K|L}^h
    \end{align*}
    and passing to the limit superior as $h\to0$ then yields
    $$
        \calF_{\sup}(v, A_r) \geq c \frac{\left\|v(\cdot+\eta) - v \right\|^2_{L^2(A_r)}}{|\eta|^2} \geq c\int_{A_r} \left| \nabla v \right|^2 \dd x \qquad \text{for } v\in H^1(\Omega).
    $$
    Due to the inner regularity property, we conclude $\calF_{\sup}(v, A) \geq c\int_{A} | \nabla v |^2 \dd x$ for  $v\in H^1(\Omega)$.
\end{proof}

We aim to find an integral representation for $\calF_{\sup}$ in the form
$$
    \calF_{\epsilon,\sup}(v, A) = \int_A f_\epsilon(x, v(x), \nabla v(x)) \dd x, \qquad v\in H^1(A).
$$
We will prove that the functions $\phi^h_{x,w,\xi} (K) = w + \langle \xi, x_K - x \rangle$ with some fixed $x\in\Omega, w\in\R, \xi\in\R^d$ are almost minimizers for $\calF_{\epsilon,h}$. For this reason, we set
\[
    M_{\epsilon,h} (v^h, A) \coloneq \inf \Bigl\{ \calF_{\epsilon,h} (w^h, A) ~ : ~ w^h \text{ on } \calT^h|_A \quad\text{with}\quad  w^h = v^h \text{ on } \calT^h|_{A^c} \Bigr\}.
\]

\begin{lemma}\label{lemma:almost-minimizers} The family of functions $\{\phi^h_{x,w,\xi} \}_{h>0}$ with $x\in\Omega, w\in\R, \xi\in\R^d$ are almost minimizers for $\calF_{\epsilon,h}$, i.e.
\[
    \lim_{h\to 0} \left( \calF_{\epsilon,h} (\phi^h_{x,w,\xi}, Q_r(x)) - M_{\epsilon,h} (\phi^h_{x,w,\xi}, Q_r(x)) \right) = 0,
\]
for a cube $Q_r(x)$ with the edge length $r>0$ and the center in $x\in\Omega$.
\end{lemma}
\begin{proof}
Let $\psi^h$ be the minimizer for $M_{\epsilon,h} (\phi^h_{x,w,\xi}, Q_r(x))$. The convexity of $\calF_{\epsilon,h}$ yields
\begin{align*}
    0 \leq \calF_{\epsilon,h} (\phi^h_{x,w,\xi}, Q_r(x)) - \calF_{\epsilon,h} (\psi^h, Q_r(x))
    \leq D\calF_{\epsilon,h}(\phi^h_{x,w,\xi}, Q_r(x)) [\phi^h_{x,w,\xi} - \psi^h].
\end{align*}
We now calculate the variation of $\calF_{\epsilon,h}(\cdot, A)$ at some $v^h\in\R^{\calT^h}$, for a fixed open set $A\subset\Omega$ in the directions $w^h\in\R^{\calT^h}$ such that $w^h_K = 0$ for $K\in\calT^h|_{A^c}$. Here, we use Lemma~\ref{lem:properties-alpha}\ref{item:beta:bound} that states the existence of the directional derivatives for $\R_+ \times \R_+ \ni (a, b) \mapsto \beta_\epsilon (a^2, b^2)$:
\begin{align*}
    D\calF_{\epsilon,h} (v^h, A)[w^h]
    &= 2\sum_{(K,L)\in\Sigma^h|_A} \Big[ \partial_1 \beta_\epsilon \bigl((v^h_K)^2, (v^h_L)^2\bigr) v^h_K w^h_K + \partial_2 \beta_\epsilon \bigl((v^h_K)^2, (v^h_L)^2\bigr)  v^h_L w^h_L \Big] \tau_{K|L}^h \\
    &= 4 \sum_{(K,L)\in\Sigma^h|_A} w^h_K v^h_K \partial_1 \beta_\epsilon \bigl((v^h_K)^2, (v^h_L)^2\bigr) \tau_{K|L}^h.
\end{align*}
where we used the fact that
$$
    \partial_1 \beta_\epsilon (a, b) = \frac{\epsilon^2}{4} \int_b^a \frac{b}{z \Lambda(z, b)} \dd z = \partial_2 \beta_\epsilon (b, a).
$$
We denote for the moment
$$
    \gamma(a, b) \coloneq a \partial_1 \beta_\epsilon (a^2, b^2)
    = \frac{\epsilon^2}{4} \int_{b^2}^{a^2} \frac{a b^2}{z \Lambda(z, b^2)} \dd z
$$
and perform Taylor expansion in the first variable
\begin{align*}
    \gamma(a, b) = \gamma(b, b) + \partial_1 \gamma(a, b)|_{a=b} (b - a) + \partial_1^2 \gamma(a, b)|_{a=b} (b - a)^2 + o\left( (a-b)^2\right).
\end{align*}
Direct calculations provide
\begin{align*}
    \partial_1 \gamma(a, b) &= \frac{\epsilon^2}{4} \left(\int_{b^2}^{a^2} \frac{b^2}{z \Lambda(z, b^2)} \dd z - \frac{a b^2}{a^2 \Lambda(a^2, b^2)} 2a \right)
    = \frac{\epsilon^2}{4} \left(\int_{b^2}^{a^2} \frac{b^2}{z \Lambda(z, b^2)} \dd z - \frac{2 b^2}{\Lambda(a^2, b^2)} \right),
\end{align*}
and thus, $\partial_1 \gamma(a, b)|_{a=b} = - \epsilon^2 / 2$. Calculating the second derivative, we obtain
\begin{align*}
    \partial_1^2 \gamma(a, b) &= \frac{\epsilon^2}{4} \left( \frac{b^2}{a^2 \Lambda(a^2, b^2)} 2a + \frac{2 b^2}{\Lambda^2(a^2, b^2)} \partial_1 \Lambda(a^2, b^2) 2a \right) \\
    &= \frac{\epsilon^2}{4} \frac{2 b^2}{a \Lambda(a^2, b^2)} \left( 1 - 2 \frac{a^2 - \Lambda(a^2, b^2)}{a^2 - b^2} \right) \overset{a \to b}{\longrightarrow} 0.
\end{align*}
Therefore,
$$
    \gamma(a, b) = -\frac{\epsilon^2}{2} (b - a) + o\left( (a-b)^2 \right).
$$
Inserting this expansion into the variation of $\calF_{\epsilon,h}$ yields
\begin{align*}
    D\calF_{\epsilon,h}(\phi^h_{x,w,\xi}, Q_r(x))[w^h] 
    &= 4 \sum_{(K,L)\in\Sigma^h|_{Q_r(x)}} w^h_K \left( -\frac{\epsilon^2}{2} \langle \xi, x_L - x_K \rangle + o( h^2 ) \right) \tau_{K|L}^h.
\end{align*}
Since for any admissible tessellation, $ \sum_{L\in\calT^h_K} (x_L - x_K)\,\tau_{K|L}^h = 0$ for $K\in\calT^h|_A \backslash \calT^h|_{A^c}$ (see \eqref{eq:zero-average}), we obtain
$$
    \left| D\calF_{\epsilon,h}(\phi^h_{x,w,\xi}, Q_r(x))[w^h] \right| \leq o(1)_{h\to 0} \sum_{K\in\calT^h|_{Q_r(x)}} |w^h_K | |K|,
$$
which proves the assertion.
\end{proof}

We now split the functional $\calF_{\epsilon,h}$ into the quadratic part and the error term, i.e.
\begin{equation}\label{eq:split-general-functional}
    \calF_{\epsilon,h}(v^h) = \frac{\epsilon^2}{2} \sum_{(K,L)\in\Sigma^h} \left( v^h_L - v^h_K \right)^2 \tau_{K|L}^h - \sum_{(K,L)\in\Sigma^h} e_\epsilon (v^h_K, v^h_L) \,\tau_{K|L}^h,
\end{equation}
where we denote $e_\epsilon(a,b)=\frac{\epsilon^2}{2} \left( a - b \right)^2 - \beta_\epsilon(a^2, b^2)$. The first observation to make is that the error term vanishes in the $\varGamma$-limit.
\begin{lemma}\label{lemma:error-term} Let $x\in\Omega, w\in\R, \xi\in\R^d$ be fixed. For the discrete functions $\phi^h_{x,w,\xi} (K) = w + \langle \xi, x_K - x \rangle$ for all $K\in\calT^h$, the following convergence holds
    $$
        \lim_{h\to 0} \sum_{(K,L)\in\Sigma^h|_{Q_r(x)}} e \big( \phi^h_{x,w,\xi}(K), \phi^h_{x,w,\xi}(L) \big) \tau_{K|L}^h = 0,
    $$
for any $r > 0$.
\end{lemma}
\begin{proof} 
We recall that $e_\epsilon(a, b) = \frac{\epsilon^2}{2} (a - b)^2 - \beta_\epsilon (a^2, b^2)$. Lemma~\ref{lem:properties-alpha}\ref{item:beta:bound}  yield the following bound \begin{align*}
    e_\epsilon(a, b) \leq \frac{\epsilon^2}{2} (a - b)^2 - \frac{\epsilon^2}{4} \frac{(a^2 - b^2)^2}{a^2 + b^2}
    = \frac{\epsilon^2}{4} (a - b)^2 \frac{2(a^2 + b^2) - (a + b)^2}{a^2 + b^2}
    = \frac{\epsilon^2}{4} \frac{(a - b)^4}{a^2 + b^2}.
\end{align*}
W.l.o.g., we assume that $w=0$. If $\phi^h_{x,\xi}(K) = \phi^h_{x,\xi}(L) = 0$, then $e( \phi^h_{x,\xi}(K), \phi^h_{x,\xi}(L))=0$, and we do not need to take these terms into account. Thus, we only need to consider the edges $\Sigma^h|_{Q_r(x)}$ for which $\phi^h_{x,\xi}(K) \ge 0, ~ \phi^h_{x,\xi}(L) > 0$ or $ \phi^h_{x,\xi}(K) > 0, ~ \phi^h_{x,\xi}(L) \ge 0 $.
    
Let $\delta > 0$ be arbitrary and define
\[
    \Sigma^h_\delta \coloneq \left\{ (K, L)\in \Sigma^h|_{Q_r(x)}: \min\bra*{ | \phi^h_{x,\xi}(K) |, | \phi^h_{x,\xi}(L) | } > \delta |\xi| \right\}.
\]
Using the non-degeneracy of the tessellation, we get
\begin{align*}
    \sum_{(K,L)\in\Sigma^h_\delta} e \bra*{ \phi^h_{x,\xi}(K), \phi^h_{x,\xi}(L) } \tau_{K|L}^h
    \leq C \frac{\epsilon^2}{4} \sum_{(K,L)\in\Sigma^h_\delta} \frac{|\xi|^4 h^4}{|\xi|^2 \delta^2} h^{d-2}
    \leq C \epsilon^2 \frac{|\xi|^2 h^2}{\delta^2} |\Omega|.
\end{align*}
The remainder of the sum, precisely the sum over $\overline{\Sigma}^h_\delta \coloneq \Sigma^h|_{Q_r(x)} \backslash \Sigma^h_\delta$,  can be bounded with the inequality $e_\epsilon(a, b) \leq \frac{\epsilon^2}{2} (a - b)^2$ to obtain
\begin{align*}
    \sum_{(K,L)\in\overline{\Sigma}^h_\delta } e \bra*{ \phi^h_{x,\xi}(K), \phi^h_{x,\xi}(L) } \tau_{K|L}^h 
    &\leq \frac{\epsilon^2}{2} \sum_{(K,L)\in\overline{\Sigma}^h_\delta} \left| \langle \xi, x_L - x_K \rangle \right|^2 \tau_{K|L}^h 
    \leq C \frac{\epsilon^2}{2} |\xi|^2 h^d \left| \Sigma^h \backslash \Sigma^h_\delta \right|.
\end{align*}
    
If $(K,L)\in\overline{\Sigma}^h_\delta $, then either $|\langle \xi, x_K - x \rangle| \leq |\xi| \delta$ or $|\langle \xi, x_L - x \rangle| \leq |\xi| \delta$, and therefore,
\[
    | \overline{\Sigma}^h_\delta |
    \leq C_\calN \Big| \big\{ K \in \calT^h|_{Q_r(x)} : |\langle \xi, x_K - x \rangle | \leq |\xi| \delta \big\} \Big| \eqcolon C_\calN |\calT^h_\delta|,
\]
where $C_\calN$ is the universal bound on the number of adjacent cells from \eqref{eq:univ-bound-adjacent}. The inequality $|\langle \xi, x_K - x \rangle| \leq |\xi| \delta$ means that the point $x_K$ lies within distance $\delta$ from the line passing through $x$ and having the direction vector $\xi$. Employing the non-degeneracy assumption again, we get
\[
    | \overline{\Sigma}^h_\delta  | \leq C_\calN |\calT^h_\delta| \leq C_\calN \frac{C_{d-1} \delta 2\sqrt{d} r^{d-1}}{C_d (\zeta h)^{d}} = C \frac{\delta r^{d-1}}{h^d}.
\]
Hence, the sum over all $(K, L)\in \Sigma^h|_{Q_r(x)}$ has the following bound
\[
    \sum_{(K,L)\in\Sigma^h|_{Q_r(x)}} e \bra*{ \phi^h_{x,\xi}(K), \phi^h_{x,\xi}(L) } \tau_{K|L}^h \leq C \epsilon^2 |\xi|^2 \left( \frac{h^2}{\delta^2} + \delta r^{d-1} \right).
\]
For $d \geq 2$ we choose $\delta(h) = \sqrt{h}$ for all $h > 0$ to obtain the asserted limit.
\end{proof}

Inserting the functions $\phi^h_{x,w,\xi}$ into the quadratic part of $\calF_{\epsilon,h}$ yields
\begin{align}\label{eq:insert-affine}
    \calF_{\epsilon,h}(\phi^h_{x,w,\xi})
    &= \frac{\epsilon^2}{2} \sum_{K\in\calT^h} \Big\langle \xi, \sum_{L\in\calT^h_K} \frac{\tau_{K|L}^h}{|K|} (x_L - x_K) \otimes (x_L - x_K) \xi \Big\rangle |K| = \frac{\epsilon^2}{2} \int_{\R^d} \langle \xi, \T^h(x) \xi \rangle \dd x
\end{align}
with the tensor
\begin{equation}\label{eq:tensor}
    \T^h(x) \coloneq \sum_{K\in\calT^h} \Ind_K(x) \sum_{L\in\calT^h_K} \frac{\tau_{K|L}^h}{|K|} (x_L - x_K) \otimes (x_L - x_K).
\end{equation}

The properties of $\T^h$ are summarized in the following proposition.
\begin{proposition}[Lemma~5.12 in \cite{hraivoronska2023diffusive}]\label{prop:tensor}
    The diffusion tensor \eqref{eq:tensor} has the following properties:
    \begin{enumerate}[label=(\roman*)]
        \item $\T^h(x)$ is symmetric and positive-definite for any $x\in \Omega$; 
        
        \item $\{\T^h\}_{h>0}$ is bounded in $ L^\infty (\Omega; \R^{d\times d})$: 
        \[
            \text{for all the components $\T^h_{ij}$ it holds that  $\displaystyle \sup_{h>0} \|\T^h_{ij}\|_{L^\infty(\Omega)} < \infty$;}
        \]
        \item $\{\T^h\}_{h>0}$ has a weakly-$*$ limit in the $\sigma(L^\infty, L^1)$ topology, i.e.\ there exist a subsequence and a tensor $\T\in L^\infty (\Omega; \R^{d\times d})$ such that
        $$
            \lim_{h\to 0} \int_{\Omega} \T^h_{ij} f \dd x = \int_{\Omega}\T_{ij} f \dd x \qquad \text{for all } f\in L^1(\Omega).
        $$
    \end{enumerate}
\end{proposition}

Proposition~\ref{prop:tensor} guarantees that there exists a limiting tensor $\T$, but, for an arbitrary tessellation, $\T$ is not necessarily the identity. In the next proposition, we show that \eqref{ass:orthogonality} is a sufficient condition to ensure that a family of tessellations converges to the identity matrix.
\begin{proposition}\label{prop:charact-by-orthogonality}
    Let a family of tessellations $\left\{ (\calT^h, \Sigma^h )\right\}_{h>0}$ satisfy the orthogonality assumption \eqref{ass:orthogonality}, then the family of tensors $\{\T^h\}_{h>0}$ defined in \eqref{eq:tensor} is such that
    $$
        \T^h_{ij} \rightharpoonup^* 2 \delta_{ij}, \qquad \text{weakly-$*$ in } \sigma(L^\infty, L^1)
    $$
    up to a subsequence. Thus, $\T=2\Id$.
\end{proposition}
\begin{proof}
    Consider the following piecewise constant functions
    $$
        \hat{\phi}^{i,h} (x) = \sum_{K\in\calT^h} x^i_K \Ind_K(x),\qquad i=1,\ldots,d.
    $$
    It is not difficult to show that the family $\{\hat{\phi}^{i,h}\}_{h>0}$ is bounded uniformly in $BV(\Omega)$. Firstly,
    $$
        \|\hat{\phi}^{i,h}\|_{L^1(\Omega)} = \sum_{K\in\calT^h} |x^i_K| |K|
        \leq \sup_{x\in\Omega} |x^i| |\Omega|.
    $$
    Secondly, as in the proof of Lemma~\ref{lemma:BV-bound}, we have the uniform bound on translations
    \begin{align*}
        \int_\Omega \psi(x)\bigl( \hat{\phi}^{i,h}(x - \eta) - \hat{\phi}^{i,h}(x) \bigr) \dd x &\leq \sum_{(K,L)\in\Sigma^h} \left|\phi^{i,h}_L - \phi^{i,h}_K\right| |(K|L)| |\eta| \\
        &\leq C |\eta| \sum_{K\in\calT^h} |K| = C |\eta| |\Omega|,
    \end{align*}
    for an arbitrary $\psi\in C^1_c(\Omega)$. Therefore, we can conclude that
    $$
        |D \hat{\phi}^{i,h}|(\Omega) \leq C |\Omega| \qquad \text{for all } h>0,
    $$
    for some constant $C>0$ independent of $h>0$. 
    This BV bound implies that (up to a subsequence) there exists $\Tilde{\phi}^i\in BV(\Omega)$ such that $\hat{\phi}^{i,h} \to \Tilde{\phi}^i$ in $L^1(\Omega)$ and $D\hat{\phi}^{i,h} \rightharpoonup^* D\Tilde{\phi}^i$ weakly-$*$ in $\calM(\Omega; \R^d)$. On the other hand, it is easy to see that $\hat{\phi}^{i,h} \to x^i$ in $L^1(\Omega)$. Therefore,
    \[
        \int_\Omega \varphi\, (D_j\hat\phi^{i,h})(\dd x) = \int_\Omega \partial_j\varphi \,\hat{\phi}^{i,h}\dd x \longrightarrow \int_\Omega \partial_j\varphi \,x^i \dd x = -\int_\Omega \varphi \,\delta_{ij}\dd x
    \]
    for all $\varphi\in C_c^1(\Omega)$, which consequently yields $D_j \tilde{\phi}^i= \delta_{ij}$.
    
    On the other hand, using the piecewise constant structure of $\hat{\phi}^{i,h}$, we can write its distributional derivative explicitly as
    $$
        D \hat{\phi}^{i,h} = \frac{1}{2}\sum_{(K, L)\in\Sigma^h} (x^i_L - x^i_K) \nu_{K|L} \calH^{d-1}|_{(K|L)},
    $$
    where $\nu_{K|L}$ denotes the outer normal of the face $(K|L)$.
    Due to the orthogonality assumption, we have that $\nu_{K|L} = (x_L - x_K)/|x_L - x_K|$, and hence
    $$
         D \hat{\phi}^{i,h} = \frac{1}{2}\sum_{(K, L)\in\Sigma^h} \tau_{K|L}^h (x^i_L - x^i_K) (x_L - x_K) \frac{\calH^{d-1}|_{(K|L)}}{|(K|L)|}.
    $$
    Notice that $D \hat{\phi}^{i,h}$ is related to the tensor $\T^h$ in the following way: For any $\varphi\in C_c^1(\Omega)$, 
    \begin{align*}
        \int_\Omega \varphi(x)\,D_j\hat\phi^{i,h}(\dd x)
        &= \frac{1}{2} \sum_{(K,L)\in \Sigma^h} \tau_{K|L}^h (x^i_L - x^i_K) (x_L^j - x_K^j) \intbar_{(K|L)} \varphi(y)\,\calH^{d-1}(\dd y) \\
        &= \frac{1}{2} \sum_{(K,L)\in \Sigma^h} \tau_{K|L}^h (x^i_L - x^i_K) (x_L^j - x_K^j)\,\varphi(x_K) + o(1) \\
        &= \frac{1}{2} \sum_K \int_K \sum_{L\in\calT^h_K} \frac{\tau_{K|L}^h}{|K|} (x^i_L - x^i_K) (x_L^j - x_K^j)\,\varphi(x)\dd x + o(1) \\
        &= \frac{1}{2} \int_\Omega \T^h_{ij}(x)\,\varphi(x)\dd x + o(1).
    \end{align*}
    Therefore, passing to the limit then yields $\T_{ij} = 2 \delta_{ij}$. In particular, $\T = 2 \Id$.
\end{proof}

We are now in the position to summarize the convergence statement for $\calD_{\epsilon,h}^0$.
\begin{lemma}\label{lem:convergence-D0} Up to a subsequence, the family of functionals $\{\calD_{\epsilon,h}^0\}_{h>0}$ has a $\varGamma$-limit $\calD_\epsilon$ with respect to the $L^2$-topology taking the form
\begin{equation*}
    \calD_\epsilon^0(\rho) = \begin{cases}\displaystyle
        2\epsilon^2\int_\Omega \big| \nabla\sqrt{u} \big|^2 \dd x & \text{if } \sqrt{u}\coloneq \sqrt{\frac{\dd\rho}{\dd x}} \in H^1(\Omega), \\
        +\infty & \text{otherwise.}
    \end{cases}
\end{equation*}
\end{lemma}
\begin{proof} The proof follows from the localization method (cf.\ \cite[Chapter~16-20]{maso1993introduction}) and the results presented above in this section.

    The localized version $\calF_{\epsilon,h}$ of $\frac{1}{2} \calD_{\epsilon,h}^0$ is defined in \eqref{eq:general-functional} and the properties of the $\Gamma$-$\lim\sup$ of $\calF_{\epsilon,h}$ that are necessary to apply the localization method are summarized in Lemma~\ref{lemma:sobolev-bounds-tilted}. Consequently, $\calF_{\epsilon,\text{sup}}$ has an integral representation
    \[
        \calF_{\epsilon,\text{sup}} (v, A) = \int_A f_\epsilon(x, v(x), \nabla v(x)) \dd x, \qquad v\in H^1(A), \quad A \subset \Omega.
    \]
    
    In \eqref{eq:split-general-functional}, we notice that $\calF_{\varepsilon,h}$ splits into the quadratic part and the error term. We prove that the error term vanishes in the $\Gamma$-limit in Lemma~\ref{lemma:error-term}. To obtain the $\Gamma$-limit of the quadratic part, we use Lemma~\ref{lemma:almost-minimizers}, where we establish that the affine functions are almost minimizers for $\calF_{\varepsilon,h}$. Inserting the affine functions into $\calF_{\epsilon,h}$ in \eqref{eq:insert-affine} and employing the convergence result from Proposition~\ref{prop:tensor}, we derive that 
    \[
        f_\epsilon(x, v, \xi) = \frac{\epsilon^2}{2} \langle \xi, \T(x) \xi \rangle.
    \]
    Lastly, Proposition~\ref{prop:charact-by-orthogonality} provides $\T(x) = 2 \Id$, thereby concluding the proof.
\end{proof}

To complete the proof of Theorem~\ref{th:convergence-tilted-Fisher}, we present the continuous convergence results for $\calD_{\epsilon,h}^1$ and $\calD_{\epsilon,h}^2$. As preparation, we establish the relation between $q^h$ and the potentials $V$ and $W$. 

\begin{lemma}\label{lem:properties-potentials} Let $W$ satisfy \eqref{ass:pointy-potential} and the family $\{\rho^h\in\calP(\calT^h)\}_{h>0}$ be such that 
\[
	\frac{\dd\hat{\rho}^h}{\dd\calL^d}\to \frac{\dd\rho}{\dd\calL^d}\quad \text{in $L^1(\Omega)$},\qquad \text{with}\qquad \sup_{h>0} \int_\Omega \phi\bra*{\frac{\dd\hat{\rho}^h}{\dd\calL^d}}\dd\calL^d  < \infty.
\]
where $\phi(s) = s \log s - s  + 1$ is the entropy density.
Then the following relation holds:
\begin{equation*}
    q_{K|L}^h = \nabla\mathsf{Q} (\hat{\rho}^h) (x_{K|L}) \cdot (x_L - x_K) + o(h), \quad \text{for any } x_{K|L}\in K\cup L,
\end{equation*}
 where $\mathsf{Q}(\rho) = V + W\ast\rho$. Moreover, $q_{K|L}^h$ has the following two integral approximations
\begin{equation}\label{eq:q-representation-K}
    q_{K|L}^h = \intbar_K \nabla\mathsf{Q} (\hat\rho^h) (x) \dd x \cdot (x_L - x_K) + o(h)
\end{equation}
and
\begin{equation}\label{eq:q-representation-KL}
    q_{K|L}^h = \intbar_{(K|L)} \nabla\mathsf{Q} (\hat\rho^h) (x) \,\calH^{d-1}(\dd x) \cdot (x_L - x_K) + o(h).
\end{equation}
\end{lemma}
\begin{proof}
Since $\nabla V$ is uniformly continuous on $\overline\Omega$, we obtain that
\begin{align*}
    V(x_L) - V(x_K) = \nabla V(x_{K|L}) \cdot (x_L - x_K) + o(h),
\end{align*}
where $x_{K|L}$ is some point in $K\cup L$.

The part of $q_{K|L}^h$ related to the interaction potential is
\begin{align*}
    &\sum_{M\in\calT^h} \rho^h_M \big( W(x_L - x_M) - W(x_K - x_M) \big) \\
    &\hspace{6em}= \sum_{\substack{M\in\calT^h\\ M\neq K,L}} \rho^h_M \big( W(x_L - x_M) - W(x_K - x_M) \big) \\
    &\hspace{8em}+ (W(x_L - x_K) - W(0)) \rho^h_K + (W(0) - W(x_K - x_L)) \rho^h_L.
\end{align*}
The later terms are bounded as
\begin{align*}
    |W(x_L - x_K) - W(0)| \rho^h_K + |W(0) - W(x_K - x_L)| \rho^h_L \leq 2\,h\, \text{Lip}(W) \sup_{x\in\Omega} \hat{\rho}^h(B_h(x)).
\end{align*}
We intend to show that $\sup_{x\in \Omega} \hat{\rho}^h (B_h(x)) \to 0$. Using the Legendre-duality, we obtain
\[
	\int_\Omega \phi(\hat{u}^h(z)) \dd z \ge \beta\hat\rho^h(B_h(x)) - \phi^*(\beta)\,\calL^d(B_h(x)) \qquad\text{for any $\beta>0$},
\]
where $\phi(s) = s \log s - s  + 1$ is the entropy density. In particular, we obtain
\[
	\sup_{x\in\Omega}\hat\rho^h(B_h(x)) \le \frac{1}{\beta} \left\{\sup_{h>0}\int_\Omega \phi(\hat{u}^h(z)) \dd z + \phi^*(\beta)\,C_d (3h)^d\right\}\qquad\text{for any $\beta>0$}.
\]
Therefore, the limsup as $h\to 0$ yields
\[
	0\le \limsup_{h\to 0}\sup_{x\in\Omega}\hat\rho^h(B_h(x)) \le \frac{1}{\beta} \sup_{h>0}\int_\Omega \phi(\hat{u}^h(z)) \dd z.
\]
Since $\beta>0$ was arbitrary, we can send $\beta\to \infty$ to obtain the required limit, and thus
$$
    (W(x_L - x_K) - W(0)) \rho^h_K + (W(0) - W(x_K - x_L)) \rho^h_L = o(h).
$$
For $M\neq K,L$, we choose an arbitrary $x_{K|L} \in K \cup L$ to obtain
\begin{align*}
    W(x_L - x_M) - W(x_K - x_M) &= \int_0^1 \nabla W ((1-\lambda) x_K + \lambda x_K - x_M) \dd\lambda \cdot (x_L - x_K) \\
    &= \nabla W (x_{K|L} - x_M) \cdot (x_L - x_K) + o(h).
\end{align*}

We now return to the whole expression for $q_{K|L}^h$ and write
\begin{align*}
    q_{K|L}^h &= \nabla V(x_{K|L}) \cdot (x_L - x_K) + \sum_{M\in\calT^h, M\neq K,L} \rho^h_M \intbar_M \nabla W (x_{K|L} - x) ~ \dd x \cdot (x_L - x_K) + o(h) \\
    &= \nabla V(x_{K|L}) \cdot (x_L - x_K) + \int_{\Omega\backslash \overline{K\cup L}} \nabla W (x_{K|L} - x) \, \hat{\rho}^h (\dd x) \cdot (x_L - x_K) + o(h) \\
    &= \nabla \mathsf{Q} (\hat{\rho}^h) (x_{K|L}) \cdot (x_L - x_K) - \int_{K\cup L} \nabla W (x_{K|L} - x) \, \hat{\rho}^h (\dd x) \cdot (x_L - x_K) + o(h).
\end{align*}
In a similar way as above, we obtain
\begin{align*}
    \bigg| \int_{K\cup L} \nabla W (x_{K|L} - x) \, \hat{\rho}^h (\dd x) \bigg| \leq \text{Lip} (W) \sup_{x\in\Omega} \hat{\rho}^h(B_{3h} (x)) \xrightarrow{h\to 0} 0,
\end{align*}
therefore,
$$
    q_{K|L}^h = \nabla\mathsf{Q} (\hat{\rho}^h) (x_{K|L}) \cdot (x_L - x_K) + o(h).
$$

To show the integral representations \eqref{eq:q-representation-K} and \eqref{eq:q-representation-KL}, we note that $\nabla \mathsf{Q} (\hat\rho^h)$ converges uniformly to $\nabla \mathsf{Q} (\rho)$. Indeed,
\begin{align*}
    \big| \nabla \mathsf{Q} (\hat\rho^h)(x) - \nabla \mathsf{Q} (\rho)(x) \big| 
    &\leq \bigg| \int_\Omega \nabla W (x - y) (\hat\rho^h - \rho) (\dd y) \bigg| \leq \text{Lip} (W) \| \hat{u} - u \|_{L^1(\Omega)}.
\end{align*}
The uniform convergence implies that the family $\{ \nabla \mathsf{Q} (\hat\rho^h) \}_{h>0}$ is uniformly equicontinuous. Hence,
\begin{align*}
    \bigg| \nabla\mathsf{Q}(\hat{\rho}^h)(x_{K|L}) - \intbar_K \nabla\mathsf{Q}(\hat{\rho}^h)(x) \dd x \bigg|
    \leq \intbar_K | \nabla\mathsf{Q}(\hat{\rho}^h)(x_{K|L}) - \nabla\mathsf{Q}(\hat{\rho}^h)(x) | \dd x = o(1)
\end{align*}
and \eqref{eq:q-representation-K} follows. The same argument works for \eqref{eq:q-representation-KL}.
\end{proof}

\begin{lemma}\label{lem:convergence-D1} Let the family $
\{\rho^h
\in\calP(\calT^h) \}_{h>0}$ be such that $ \sup_{h>0} \calD_{\epsilon,h}^0(\rho^h) < \infty$. Moreover, suppose that there exists $u\in W^{1,1}(\Omega)$ such that 
\[
	\frac{\dd\hat{\rho}^h}{\dd\calL^d}\to u\eqcolon\frac{\dd \rho}{\dd\calL^d}\quad \text{in $L^1(\Omega)$},\qquad \text{and}\qquad D\hat{u}^h \rightharpoonup^* \nabla u\quad\text{ weakly-$*$ in $\calM(\Omega; \R^d)$}.
\]
Then
    $$
        \lim_{h\to 0} \calD_{\epsilon,h}^1(\rho^h) = \epsilon \int_\Omega \nabla u \cdot \nabla \mathsf{Q}(\rho)\dd x.
    $$
\end{lemma}
\begin{proof}
    First, we show that $\calD_{\epsilon,h}^1$ is uniformly bounded. Using the Cauchy-Schwartz inequality yields
    \begin{align*}
        \calD_{\epsilon,h}^1(\rho^h) &= \frac{\epsilon}{2} \sum_{(K,L)\in\Sigma^h} (u^h_L - u^h_K)\, q_{K|L}^h \tau_{K|L}^h 
        \le c_\text{pot}\sqrt{\calD_{\epsilon,h}^0}\left( \sum_{(K,L)\in\Sigma^h} (u^h_L + u^h_K)\, h^2 \tau_{K|L}^h \right)^{1/2}
    \end{align*}
    where we used the estimate \eqref{eq:Cw-bound}. Since $\sum_{L\in\calT^h_K} h^2 \tau_{K|L}^h \leq C_\tau |K|$, we then obtain the uniform bound.

    Similarly, one can show that
    \begin{equation}\label{eq:bound-on-sum-proof}
        \sup_{h>0} \sum_{(K,L)\in\Sigma^h} |u^h_L - u^h_K| |(K|L)| < \infty.
    \end{equation}
    We aim to rewrite $\calD_{\epsilon,h}^1$ in an integral form, which will be convenient for passing to the limit $h\to 0$. We begin by observing that $\tau_{K|L}^h$ can be rewritten as 
    $$
        \tau_{K|L}^h = \frac{|(K|L)|}{|x_L - x_K|} = \frac{1}{|x_L - x_K|} \calH^{d-1} ((K|L)).
    $$
    Inserting this expression for $\tau_{K|L}^h$ into $\calD_{\epsilon,h}^1$ yields
    $$
        \calD_{\epsilon,h}^1(\rho^h) = \frac{\epsilon}{2} \sum_{(K,L)\in\Sigma^h} (u^h_L - u^h_K)  \frac{q_{K|L}^h}{|x_L - x_K|} \int_{(K|L)} \calH^{d-1} (\dd x).
    $$

    The representation \eqref{eq:q-representation-KL} for $q_{K|L}^h$ derived in Lemma~\ref{lem:properties-potentials} yields
    \begin{align*}
        \frac{q_{K|L}^h}{|x_L - x_K|} \int_{(K|L)} \dd\calH^{d-1}
        = \int_{(K|L)} &\nabla\mathsf{Q} (\hat{\rho}^h) (x) \,\calH^{d-1}(\dd x) \cdot \nu_{K|L} + |(K|L)| o(1)|_{h\to 0},
    \end{align*}
    where $\nu_{K|L}=(x_L - x_K)/|x_L - x_K|$ is the outer normal of the face $(K|L)$. Inserting the obtained expression into $\calD_{\epsilon,h}^1$, we have
    \begin{align*}
        \calD_{\epsilon,h}^1(\rho^h) = \frac{\epsilon}{2} \sum_{(K,L)\in\Sigma^h} &(u^h_L - u^h_K)  \int_{(K|L)} \nabla\mathsf{Q} (\hat{\rho}^h) \dd\calH^{d-1} \cdot \nu_{K|L} \\
        &+ o(1)|_{h\to 0} \sum_{(K,L)\in\Sigma^h} (u^h_L - u^h_K) |(K|L)|,
    \end{align*}
    where he last sum is bounded uniformly in $h>0$ by \eqref{eq:bound-on-sum-proof}.
    
    Altogether, we arrive at
    \begin{align*}
        \calD_{\epsilon,h}^1(\rho^h) = \frac{\epsilon}{2} \int_\Omega \nabla\mathsf{Q} (\hat{\rho}^h)(x) &\cdot \sum_{(K,L)\in\Sigma^h} (u^h_L - u^h_K) \nu_{K|L} \,\calH^{d-1}|_{(K|L)} (\dd x) + o(1)|_{h\to 0}.
    \end{align*}
    In this expression, one may already recognize the distributional derivative of the density $\hat{u}^h$. Indeed, from the definition of $\hat{u}^h$, we get
    $$
        D\hat{u}^h = \sum_{K\in\calT^h} u^h_K D\Ind_K = \sum_{K\in\calT^h} u^h_K n_K \calH^{d-1}|_{\partial K},
    $$
    where $n_K$ is the inner normal for the cell $K\in\calT^h$. It holds that 
    $$
        n_K \calH^{d-1}|_{\partial K} = \sum_{L\in\calT^h_K} n_{K|L} \calH^{d-1}|_{(K|L)} \qquad \text{for } K\in\calT^h,
    $$
    where $n_{K|L}$ is an inner normal to the face $(K|L)$. Using symmetry, we find
    $$
        D\hat{u}^h = \sum_{(K, L)\in\Sigma^h} u^h_K n_{K|L} \calH^{d-1}|_{(K|L)} = \frac{1}{2}\sum_{(K, L)\in\Sigma^h} (u^h_K - u^h_L)\, n_{K|L} \calH^{d-1}|_{(K|L)}.
    $$
    If $(\calT^h, \Sigma^h)$ possesses the orthogonality property, i.e. $$ n_{K|L} = \frac{x_K - x_L}{|x_K - x_L|} = - \nu_{K|L},$$ we can write
    \begin{align*}
        \calD_{\epsilon,h}^1(\rho^h) &= \epsilon \int_\Omega \nabla\mathsf{Q} (\hat{\rho}^h)(x) \cdot D\hat{u}^h (\dd x) + o(1)|_{h\to 0}.
    \end{align*}
	Moreover, since $\nabla\mathsf{Q} (\hat{\rho}^h)$ converges to $\nabla\mathsf{Q} (\rho)$ uniformly as $h\to 0$, we further obtain 
    $$
        \calD_{\epsilon,h}^1(\rho^h) = \epsilon \int_\Omega \nabla\mathsf{Q} (\rho)(x) \cdot D\hat{u}^h (\dd x) + o(1)|_{h\to 0}.
    $$
    Passing $h\to 0$ and using the convergence $D\hat{u}^h \rightharpoonup^* \nabla u$ in $\calM(\Omega; \R^d)$ then yields the assertion. \end{proof}

\begin{lemma}\label{lem:convergence-D2}
    Let the family $\{\rho^h\in\calP(\calT^h) \}_{h>0}$ be such that 
\[
	\frac{\dd\hat{\rho}^h}{\dd\calL^d}\to u\eqcolon\frac{\dd\rho}{\dd\calL^d}\quad \text{in $L^1(\Omega)$},\qquad \text{with}\qquad u\in\calM(\Omega; \R^d).
\]
Then
    $$
        \lim_{h\to 0} \calD_{\epsilon,h}^2(\rho^h) = \frac{1}{2} \int_\Omega \left| \nabla \mathsf{Q}(\rho) \right|^2 \dd\rho.
    $$    
\end{lemma}
\begin{proof}
    Using the symmetry, we rewrite $ \calD_{\epsilon,h}^2(\rho^h)$ as
    $$
         \calD_h^2(\rho^h) = \sum_{(K, L)\in\Sigma^h} \tau_{K|L}^h |q_{K|L}^h|^2 u^h_K \int_0^1 \mathfrak{h}\left(-\lambda q_{K|L}^h/\epsilon\right) (1-\lambda)\dd\lambda.
    $$
    The function $\mathfrak{h}$ has the following Taylor expansion for $s\ll 1$
    $$
        \mathfrak{h} (s) = \frac{1}{2} + \frac{s}{6} + o(s^2).
    $$
    Taking into account that $|q_{K|L}^h| \leq c_\text{pot} h $ (cf.\ estimate \eqref{eq:Cw-bound}), we have that
    \begin{align*}
        \int_0^1 \mathfrak{h}\left(-\lambda q_{K|L}^h/\epsilon\right) (1-\lambda)\dd\lambda
        = \frac{1}{4} + O(h/\epsilon)|_{h\to 0}\,.
    \end{align*}
    Substituting the last expression into $\calD_{\epsilon,h}^2$ yields
    \begin{align*}
        \calD_{\epsilon,h}^2(\rho^h) = \frac{1}{4} \sum_{(K, L)\in\Sigma^h} \tau_{K|L}^h |q_{K|L}^h|^2 u^h_K + o(1)_{h\to 0}.
    \end{align*}

    Now, notice that
    \begin{align*}
        &\Big| \big( \nabla \mathsf{Q} \bra{ \hat{\rho}^h}(x_K) \cdot (x_L - x_K) \big)^2 -  \intbar_K \left( \nabla \mathsf{Q} \bra{\hat{\rho}^h}(x) \cdot (x_L - x_K) \right)^2 \dd x \Big|  \\
        &\hspace{10em}\leq C h^2 \sup_{x\in K} \left| \nabla \mathsf{Q} \bra{ \hat{\rho}^h }(x_K) - \nabla \mathsf{Q} \bra{ \hat{\rho}^h }(x) \right|  = o \bra{ h^2 }.
    \end{align*}
    Using the representation (cf.\ \eqref{eq:q-representation-K})
    $$
        q_{K|L}^h = \intbar_K \nabla \mathsf{Q} \bra{ \hat{\rho}^h }(x) \dd x \cdot (x_L - x_K) + o(h),
    $$
    we can then rewrite $\calD_{\epsilon,h}^2$ as
    \begin{align*}
        \calD_{\epsilon,h}^2(\rho^h) &= \frac{1}{4} \sum_{(K, L)\in\Sigma^h} u^h_K \tau_{K|L}^h \intbar_K \left( \nabla \mathsf{Q} \bra{ \hat{\rho}^h }(x) \cdot (x_L - x_K) \right)^2 \dd x  + o(1)_{h\to 0} \\
        &= \frac{1}{4} \int_\Omega \hat{u}^h(x) \sum_{(K, L)\in\Sigma^h} \frac{\tau_{K|L}^h}{|K|} \Ind_K(x) \left( \nabla \mathsf{Q} \bra{ \hat{\rho}^h }(x) \cdot (x_L - x_K) \right)^2 \dd x + o(1)_{h\to 0} \\
        &= \frac{1}{4} \int_\Omega \hat{u}^h(x)\, \langle \nabla \mathsf{Q} \bra{ \hat{\rho}^h }(x), \T^h(x) \nabla \mathsf{Q} \bra{ \hat{\rho}^h }(x) \rangle \dd x + o(1)_{h\to 0},
    \end{align*}
    where we recall the tensor
    $$
        \T^h(x) = \sum_{K\in\calT^h} \Ind_K (x) \sum_{L\in\calT^h_K} \frac{\tau_{K|L}^h}{|K|} (x_L - x_K) \otimes (x_L - x_K).
    $$
    The product $\langle \nabla \mathsf{Q} \bra{ \hat{\rho}^h }(x), \T^h(x) \nabla \mathsf{Q} \bra{ \hat{\rho}^h }(x) \rangle$ has an $L^\infty$ bound uniformly in $h>0$, since for any $x\in \Omega$, there is some $K$ for which $x\in K$ and 
    \begin{align*}
        &\big| \langle \nabla \mathsf{Q} \bra{\hat{\rho}^h}(x), \T^h(x) \nabla \mathsf{Q} \bra{ \hat{\rho}^h }(x) \rangle \big|
        \leq \sum_{L\in\calT_K^h} \frac{\tau_{K|L}^h}{|K|}\left( \nabla \mathsf{Q} \bra{ \hat{\rho}^h }(x) \cdot (x_L - x_K) \right)^2 \\ 
        &\hspace{8em}\leq c_\text{pot}^2 \sum_{L\in\calT_K^h} \frac{|(K|L)| |x_K-x_L|}{|K|} \le c_\text{pot}^2\frac{C_{d-1}}{C_d\zeta^{d+1}} \sup_{h>0}\sup_{K\in\calT^h}\# \calT_K^h <\infty.
    \end{align*}
    It is left to show that, for any $f\in L^1(\Omega)$, we have the convergence
    $$
        \lim_{h\to 0} \int_\Omega f\, \langle \nabla \mathsf{Q} ( \hat{\rho}^h ), \T^h \nabla \mathsf{Q} ( \hat{\rho}^h ) \rangle \dd x = \int_\Omega f\, | \nabla \mathsf{Q} ( \rho ) |^2 \dd x.
    $$
    We consider the limit component-wise
    \begin{align*}
        \lim_{h\to 0} \int_\Omega f\, \partial_i \mathsf{Q} ( \hat{\rho}^h )\, \partial_j \mathsf{Q} ( \hat{\rho}^h )\, \T^h_{ij} \dd x 
        &= \lim_{h\to 0} \int_\Omega f\, \partial_i \mathsf{Q} ( \rho )\, \partial_j \mathsf{Q} ( \rho )\, \T^h_{ij} \dd x \\
        &+ \lim_{h\to 0} \int_\Omega f \left[ \partial_i \mathsf{Q} ( \hat{\rho}^h )\, \partial_j \mathsf{Q} ( \hat{\rho}^h ) - \partial_i \mathsf{Q} ( \rho )\, \partial_j \mathsf{Q} ( \rho) \right] \T^h_{ij} \dd x,
    \end{align*}
    where $f\, \partial_i \mathsf{Q} ( \rho )\, \partial_j \mathsf{Q} ( \rho ) \in L^1(\Omega)$ and, since $\T^h_{ij} \rightharpoonup^* 2 \delta_{ij}$ in $\sigma(L^\infty, L^1)$ by Proposition~\ref{prop:charact-by-orthogonality}, the first term converges to the expected limit. For the error term, we notice that
    \begin{align*}
        \Big| \int_\Omega f &\left[ \partial_i \mathsf{Q} ( \hat{\rho}^h )\, \partial_j \mathsf{Q} ( \hat{\rho}^h ) - \partial_i \mathsf{Q} ( \rho )\, \partial_j \mathsf{Q} ( \rho) \right] \T^h_{ij} \dd x\Big| \\
        &\leq \| \partial_i \mathsf{Q} ( \hat{\rho}^h )\, \partial_j \mathsf{Q} ( \hat{\rho}^h ) - \partial_i \mathsf{Q} ( \rho )\, \partial_j \mathsf{Q} ( \rho) \|_{\sup} \|f\|_{L^1} \|\T^h_{ij}\|_{L^\infty}
        \to 0\qquad \text{as $h\to 0$},
    \end{align*}
    due to the uniform convergence of $\nabla \mathsf{Q} ( \hat{\rho}^h )$ to $\nabla \mathsf{Q} ( \rho )$.
\end{proof}

\subsection{EDP convergence}\label{sec:EDP-summary}

\begin{definition}[Density-flux convergence]\label{def_measure_flux_convergence}
    A pair $(\rho^h, j^h)\in\mathcal{CE}_h(0,T)$ is said to converge to a  pair $(\rho, j)\in\mathcal{CE}(0,T)$ if the pair of reconstructions $(\hat \rho^h, \hat\jmath^h)\in\mathcal{CE}(0,T)$ defined as in \eqref{eq:lifted-pair} converges in the following sense:
    \begin{enumerate}
        \item $\dd \hat\rho^h_t/\dd \calL^d \to \dd\rho_t/\dd \calL^d$ in $L^1(\Omega)$ for almost every $t\in[0, T]$,
        \item $\int_{\cdot}\hat\jmath_t^h \dd t \rightharpoonup^* \int_{\cdot} j_t\dd t$ in $\calM((0, T)\times\Omega)$.
    \end{enumerate}
\end{definition}

We begin by summarizing the liminf inequalities for the tilt-independent gradient structure. 
\begin{theorem}\label{th:liminfs-tilt} Let $(\rho^h, j^h)\in\mathcal{CE}_h(0,T)$ converge to $(\rho, j)\in\mathcal{CE}(0,T)$ in the sense of Definition~\ref{def_measure_flux_convergence}. Then the following liminf inequalities hold for
\begin{enumerate}[label=(\roman*)]
    \item the dissipation potential:
    $$
        \int_0^T \frac{1}{2} \int_\Omega \left| \frac{\dd j_t}{\dd\rho} \right|^2 \dd\rho \dd t\le \liminf_{h\to 0} \int_0^T \calR_{\epsilon,h}(\rho^h_t, j^h_t) \dd t ;
    $$
    \item the Fisher information:
    $$
        \int_0^T \calD_\epsilon(\rho_t) \dd t\le \liminf_{h\to 0} \int_0^T \calD_{\epsilon,h}(
        \rho^h_t) \dd t ;
    $$
    \item the energy functional:
    $$
        \calE_\epsilon(\rho_t)\le \liminf_{h\to 0} \calE_{\epsilon,h}(\rho^h_t) \qquad\text{for all $t\in[0,T]$.}
    $$
\end{enumerate}
\end{theorem}
\begin{proof}
\emph{(i)} We need to show that the following limsup inequality holds for any $\varphi\in \calC_b^2(\Omega)$:
\begin{equation}\label{eq:limsup-Rstar}
        \limsup_{h\to 0} \calR^*_{\epsilon,h}(\rho^h, \dnabla\varphi^h) \leq \frac{1}{2} \int_\Omega |\nabla\varphi|^2 \dd\rho,
\end{equation}
where $\{\varphi^h\}_{h>0}$ is defined by $\varphi^h(K) \coloneq \varphi(x_K)$ for $K\in\calT^h$. Then the desired liminf inequality follows by the duality argument from \cite[Theorem~6.2(i)]{hraivoronska2023diffusive}.
    
    From Lemma~\ref{lem:properties-alpha}\ref{alpha-cosh}, it follows that
$$
    \calR_{\epsilon,h}^*(\rho^h, \dnabla\varphi^h) = \frac{1}{2} \sum_{(K,L)\in\Sigma^h} | \dnabla\varphi^h_{K|L} |^2 \Lambda_H(u^h_K, u^h_L)\, \tau_{K|L}^h + \frac{1}{\epsilon} \sum_{(K,L)\in\Sigma^h} O \big(| \dnabla\varphi^h_{K|L} |^3 \big) \tau_{K|L}^h.
$$
We note that $O \big(| \dnabla\varphi^h_{K|L} |^3 \big) = O(h^3)$ and, therefore,
$$
    \frac{1}{\epsilon} \sum_{(K,L)\in\Sigma^h} O \left(\left| \dnabla\varphi^h_{K|L} \right|^3 \right) \tau_{K|L}^h
    = \frac{1}{\epsilon}O(h).
$$
Using the inequality $\Lambda_H(a, b) \leq (a + b) /2$, we arrive at
$$
    \calR_{\epsilon,h}^*(\rho^h, \dnabla\varphi^h) \leq \frac{1}{2} \sum_{(K,L)\in\Sigma^h} \left| \dnabla\varphi^h_{K|L} \right|^2 \frac{\tau_{K|L}^h}{|K|} \rho^h_K + O \left(h \right).
$$
With this bound at hand, it is enough to make minor modifications of the proof of \cite[Lemma~5.14]{hraivoronska2023diffusive} for the tilt-independent dissipation potential with $\kappa^h_{K|L} = \tau_{K|L}^h / |K|$ to obtain \eqref{eq:limsup-Rstar}.

\emph{(ii)} The asserted liminf inequality follows from Theorem~\ref{th:convergence-tilted-Fisher}
and Fatou's lemma.

\emph{(iii)} As the following calculations hold for any $t\in [0,T]$, we drop the
subscript $t$. The relation between the continuous and discrete potentials yields the representation of $\calE_{\epsilon,h}$ in the integral form
$$
    \calE_{\epsilon,h}(\rho^h) = \calE_{\varepsilon}(\hat\rho^h) + O(h).
$$
Since $\calE_\epsilon$ is lower semicontinuous w.r.t.\ the narrow convergence, we then easily conclude that
$$
    \calE_\epsilon(\rho^h)\le \liminf_{h\to 0} \calE_{\epsilon,h}(\rho^h),
$$
which completes the proof.
\end{proof}

\begin{proof}[Proof of Theorem~\ref{th:EDP-convergence-tilt}]
Consider a family $\{(\rho^h, j^h)\}_{h>0}$ of GGF-solutions to Scharfetter--Gummel scheme \eqref{eq:SG}, for a fixed $\epsilon>0$, according to Definition~\ref{def_GGF_solution} and the tilt-independent structure introduced in Section~\ref{sec:tilted-gradient-structure}. Further, let $\{(\hat\rho^h, \hat\jmath^h)\}_{h>0}$ be the family of reconstructed pairs as defined in \eqref{eq:lifted-pair}. Then, the existence of a subsequential limit pair $(\rho, j) \in \mathcal{CE}(0,T)$ and the convergence specified in Theorem~\ref{th:EDP-convergence-tilt}(1) follows from the compactness arguments of Section~\ref{sec:compactness-SG}.
    
The liminf inequality from assertion (2) is proven in Theorem~\ref{th:liminfs-tilt}, which immediately implies that $\calI_\epsilon^{[s,t]}(\rho, j) \leq \liminf_{h\to 0} \calI_{\epsilon,h}^{[s,t]}(\rho^h, j^h)= 0$ for every $[s,t]\subset[0,T]$. On the other hand, the chain rule \cite[E.\ Chain rule in Section~10.1.2]{ambrosio2008gradient} yields $\calI_\epsilon^{[s,t]}(\rho, j) \geq 0$ for every $[s,t]\subset[0,T]$. Therefore, the limit pair $(\rho, j)$ is the  $(\calE, \calR, \calR^*)$-gradient flow solution of \eqref{eq:drift-aggregation-diffusion-equation} in the sense of Definition~\ref{def_GF_solution}.
\end{proof}

\section[Vanishing Diffusion Limit]{Vanishing Diffusion Limit}\label{sec:vanishing-diffusion}

This section deals with the vanishing diffusion limit for both the discrete and continuous cases, i.e.\ Theorems~\ref{th:EDP-vanishing-diffusion} and \ref{th:vanishing-diffusion-continuous}. We begin with the discrete case in Section~\ref{sec:vanishing-discrete} and proceed to prove the continuous case in Section~\ref{sec:vanishing-continuous}.

\subsection{Discrete Case}\label{sec:vanishing-discrete}

We fix a tessellation $(\calT^h, \Sigma^h)$ with some $h>0$ and consider the vanishing diffusion limit $\epsilon \to 0$. To simplify notation, we drop the superscript $h$. As mentioned in the introduction, we expect that the Scharfetter--Gummel flux \eqref{eq:SG-flux} converges to the upwind flux 
$$
    \lim_{\epsilon \to 0} \calJ_{K|L}^{\rho} = \calJ_{K|L}^{\rho,\text{up}} \coloneq \tau_{K|L}^h \left( q_{K|L}^{h,+}u_K - q_{K|L}^{h,-} u_L  \right), \qquad (K,L)\in \Sigma^h.
$$
The result of this section concerns the convergence of the Scharfetter--Gummel scheme \eqref{eq:SG} to the upwind scheme \eqref{eq:upwind-scheme} in the sense of the EDP convergence. Recall that if a pair $(\rho^{\epsilon,h}, j^{\epsilon,h})\in \mathcal{CE}_h(0, T)$ is a GGF-solutions of \eqref{eq:SG}, then $(\rho^{\epsilon,h}, j^{\epsilon,h})$ is the minimizer for the energy-dissipation functional
\begin{equation}\label{eq:EDP-discrite-epsilon}
    \calI_{\epsilon,h}^{[s,t]}(\rho^{\epsilon,h}, j^{\epsilon,h}) = \int_s^t \big\{ \calR_{\epsilon,h} (\rho^{\epsilon,h}_r, j^{h,\epsilon}_r) + \calD_{\epsilon,h} (\rho^{\epsilon,h}_r) \big\} \dd r + \calE_{\epsilon,h} (\rho^{\epsilon,h}_t) - \calE_{\epsilon,h} (\rho^{\epsilon,h}_s)
\end{equation}
with $\calR_{\epsilon,h}$, $\calD_{\epsilon,h}$, and $\calE_{\epsilon,h}$ defined in \eqref{eq:tilted-R}, \eqref{eq:tilted-Fisher}, and \eqref{eq:energy-with-interaction} respectively. The objective of this section is to get a compactness statement for $\{(\rho^{\epsilon,h}, j^{\epsilon,h})\}_{\epsilon>0}$ and to find the counterparts to $\calR_{\epsilon,h}$, $\calD_{\epsilon,h}$, and $\calE_{\epsilon,h}$ for $\epsilon = 0$. Then we complete the proof of Theorem~\ref{th:EDP-vanishing-diffusion}.

\medskip
Note that since $(\calT^h, \Sigma^h)$ is fixed and non-degenerate, we have the following useful bounds
\begin{equation}\label{eq:tessellation-bound}
    \sup_{K\in\calT^h}\sum_{L\in\calT_K^h} \frac{\tau_{K|L}^h}{|K|} \eqcolon c_\calT < \infty.
\end{equation}

We begin with the compactness result. Consider a measure $J^{\epsilon} \in \calM([0,T]\times\Sigma^h)$ defined on product measurable sets $A\times B\subset [0, T]\times\Sigma^h$ as
$$
    J^{\epsilon} (A\times B) \coloneq \int_A j_t^\epsilon(B)\dd t=\int_{A} \sum_{(K,L)\in B} j^{\epsilon}_{K|L}(t) \dd t.
$$

\begin{lemma}\label{lem:compactness-flux-epsilon}
    Let a family of pairs $\{(\rho^{\epsilon}, j^{\epsilon})\}_{\epsilon>0}\subset\mathcal{CE}_h (0, T)$ satisfy 
    \[
    	c_0\coloneq\sup_{\epsilon>0}\int_0^T \calR_{\epsilon,h}(\rho_t^\epsilon, j_t^\epsilon) \dd t<\infty.
    \]
    Then the family $\left\{ J^{\epsilon} \right\}_{\epsilon>0}$ is bounded in total variation. Moreover, 
    \[
	|J^\epsilon|(A\times \Sigma^h) \le \sqrt{c_0c_\calT\calL^1(A)}\qquad\text{for any measurable set $A\subset[0,T]$}.
	\]
\end{lemma} 
\begin{proof} 
Following the initial arguments of the proof of Lemma~\ref{lemma:compactness-for-flux-tilted}, we obtain for any $\beta\in\R$,
\begin{align*}
	\calR_{\epsilon,h}(\rho_t^\epsilon, j_t^\epsilon) \ge \beta\sum_{(K,L)\in\Sigma^h} |j_{K|L}^\epsilon|(t) - 2\sum_{(K,L)\in\Sigma^h} \tau_{K|L}^h \,\alpha_\epsilon^* \bra*{ u_K(t), u_L(t), \beta\frac{\text{sign}(j_{K|L}^\epsilon)}{2} }.
\end{align*}
If either $a=0$ or $b=0$, then $\alpha^*(a,b,x)=0$ for any $x\in\R$. If $a=b$, then
\[
	\alpha_\epsilon^*(a,a,\xi) = a \int_0^{\xi} x \dd x = a\frac{\xi^2}{2} = \alpha_0^*(a,a,\xi)\quad\text{for all $\xi\in\R$},
\]

We will now reduce the other cases to this case. Indeed, using the 1-homogeneity and concavity of~$\Lambda_H$, we have for any $\xi\in\R$ that
\begin{align}\label{eq:epsilon-uniform}\begin{aligned}
	\sum_{(K,L)\in\Sigma^h} \tau_{K|L}^h \,\alpha_\epsilon^* \bra*{ u_K, u_L, \xi} &= \sum_{(K,L)\in\Sigma^h} \alpha_\epsilon^* \bra*{ \tau_{K|L}^h u_K, \tau_{K|L}^h u_L, \xi} \\
	&\le \alpha_\epsilon^* \bra*{ \sum_{(K,L)\in\Sigma^h}\tau_{K|L}^h u_K, \sum_{(K,L)\in\Sigma^h}\tau_{K|L}^h u_L, \xi} \\
	&= \alpha_\epsilon^*\bra*{1,1,\xi}\sum_{(K,L)\in\Sigma^h}\tau_{K|L}^h u_K \le  c_\calT\frac{\xi^2}{2}.
	\end{aligned}
\end{align}
Consequently, and after integration over any measurable set $A\subset[0,T]$, we obtain the estimate
\[
	\int_0^T \calR_{\epsilon,h}(\rho_t^\epsilon, j_t^\epsilon)\dd t \ge \beta |J^\epsilon|(A\times \Sigma^h) - \frac{c_\calT}{2}\beta^2\calL^1(A).
\]
Taking the supremum over $\beta\in\R$, we arrive at the asserted estimate.
\end{proof}

\begin{lemma}\label{lem:compactness-measure-epsilon} Let a family of pairs $\{(\rho^{\epsilon}, j^{\epsilon})\}_{\epsilon>0}\subset\mathcal{CE}_h (0, T)$ satisfy 
    \[
    	c_0\coloneq\sup_{\epsilon>0}\int_0^T \calR_{\epsilon,h}(\rho_t^\epsilon, j_t^\epsilon) \dd t<\infty.
    \]
    Then there exist a limit pair $(\rho, j) \in \mathcal{CE}_h(0, T)$  and a (not relabelled) subsequence such that
    \begin{gather*}
        \rho^{\epsilon}_t \rightharpoonup \rho_t \quad \text{in $\calP(\calT^h)$\; for all $t\in [0,T]$},\\
        J^{\epsilon} \rightharpoonup^* J=\int_\cdot j_t\dd t\quad\text{weakly-$*$ in $\calM([0,T]\times\Sigma^h)$}.
    \end{gather*}
\end{lemma}
\begin{proof} The convergence for $J^\epsilon$ follows the same lines as in the proof of Lemma~\ref{lemma:compactness-for-flux-tilted}.

We now prove the convergence for $\{\rho^{\epsilon}\}_{\epsilon>0}$. Since $(\rho^{\epsilon}, j^{\epsilon})\in \mathcal{CE}_h(0, T) $, then
    \begin{align*}
        \left| \sum_{K\in\calT^h} \varphi_K \bigl(\rho^{\epsilon}_K(t) - \rho^{\epsilon}_K(s) \bigr) \right| &= \left| \int_s^t \sum_{(K,L)\in\Sigma^h} \dnabla \varphi (K,L)\, j^{\epsilon}_{K|L}(r) \dd r \right| \\
        &\leq 2\|\varphi\|_\infty |J^\epsilon|([s,t]\times\Sigma^h) \qquad\text{for any $[s, t] \subset [0, T]$.}
    \end{align*}
    Taking supremum over all $\varphi \in \calB(\calT^h)$ with $\|\varphi\|_\infty \leq 1$, we make use of Lemma~\ref{lem:compactness-flux-epsilon} to obtain
    $$
        \| \rho_t^{\epsilon} - \rho_s^{\epsilon}\|_{TV} \leq C \sqrt{|t - s|}.
    $$
	By the  Ascoli-Arzel\'a theorem, there exists a (not relabelled) subsequence of $\{\rho^{\epsilon}\}_{\epsilon>0}$ and a limit curve $\rho\in \calC([0,T];\calP(\calT^h))$, such that the asserted convergence holds.
\end{proof}

\begin{remark}\label{rem:pointwise}
Since $\calT^h$ and $\Sigma^h$ are finite discrete spaces, the weak and strong topologies coincide. In particular, the narrow convergence stated in Lemma~\ref{lem:compactness-measure-epsilon} implies the pointwise convergence. We will use this property in the proofs of the following results.
\end{remark}

In the next lemma, we establish the convergence of the Fisher information.
\begin{lemma}\label{lem:lim-fisher-epsilon} Let the family of measures $\{\rho^{\epsilon}\}_{\epsilon> 0}$ be such that $\rho^{\epsilon} \rightharpoonup \rho$ in $\calP(\calT^h)$ as $\epsilon\to 0$, then
    $$
        \lim_{\epsilon\to 0} \calD_{\epsilon,h} (\rho^{\epsilon}) = \calD_{\text{up},h} (\rho) = 2\sum_{(K, L)\in\Sigma^h} \tau_{K|L}^h\alpha_0^* \bra*{u_K, u_L, \frac{q_{K|L}^h}{2} },
    $$
    where
    $$
        \alpha_0^*(a, b, q) = \frac{1}{2} \bigl( a|q^+|^2 + b|q^-|^2 \bigr).
    $$
\end{lemma}
\begin{proof}
The limit Fisher information contains only the limit of $\calD_{\epsilon,h}^2$, since $\lim_{\epsilon\to 0} \bigl(\calD_{\epsilon,h}^0 + \calD_{\epsilon,h}^1 \bigr)= 0$. Recall that 
\begin{align*}
    \calD_{\epsilon,h}^2 (\rho^{\epsilon}) &= \frac{1}{2} \sum_{(K, L)\in\Sigma^h} \tau_{K|L}^h |q_{K|L}^h|^2\, \mathbb{h}_\epsilon \bra*{ u^{\epsilon}_K, u^{\epsilon}_L, q_{K|L}^h},
\end{align*}
with $\mathbb{h}_\epsilon$ being
$$
    \mathbb{h}_\epsilon (a, b, q) = \int_0^1 \Bigl[a\, \mathfrak{h}\left(\lambda q/\epsilon\right) + b\,\mathfrak{h}\left(-\lambda q/\epsilon\right)\Bigr](1-\lambda)\dd\lambda,\qquad \mathfrak{h}(s) = \frac{1}{4}\frac{e^s-1-s}{\sinh^2(s/2)}.
$$
It is uniformly bounded by the following argument. Since $0 \leq \mathfrak{h}\left(s\right) \leq 1$, $s\in\R$, we have that
\begin{align*}
    \calD_h^2 (\rho^{\epsilon}) &\leq \frac{1}{4}  \sum_{(K, L)\in\Sigma^h} \tau_{K|L}^h |q_{K|L}^h|^2 \bigl(u^{\epsilon}_K + u^{\epsilon}_L\bigr) \le \frac{1}{2} c_\text{pot}^2 c_\calT .
\end{align*}
Moreover, we notice that
$$
    \lim_{\epsilon\to 0} \mathfrak{h}\left(s/\epsilon\right) =  \Ind_{(0,\infty)}(s) + \frac{1}{2} \Ind_{\{0\}} (s),
$$
and, hence,
$$
    \lim_{\epsilon\to 0} \int_0^1 \mathfrak{h}\left(\lambda q/\epsilon\right) (1-\lambda)\dd\lambda = \frac{1}{2} \left( \Ind_{(0,\infty)}(q) + \frac{1}{2} \Ind_{\{0\}} (q) \right) \eqcolon \mathfrak{h}_0(q).
$$
Now we define
$$
    \Tilde{u}^{\epsilon}_{K|L} \coloneq \int_0^1 \Bigl[u^{\epsilon}_K\, \mathfrak{h}\left(\lambda q_{K|L}^h/\epsilon\right) + u^{\epsilon}_L\,\mathfrak{h}\left(-\lambda q_{K|L}^h/\epsilon\right)\Bigr](1-\lambda)\dd\lambda.
$$
Since $u^{\epsilon} \to u$ pointwise on $\calT^h$, we get
$$
    \lim_{\epsilon\to 0} \Tilde{u}^{\epsilon}_{K|L} = u_K \,\mathfrak{h}_0(q_{K|L}^h) + u_L \,\mathfrak{h}_0(-q_{K|L}^h),
$$
which concludes the proof.
\end{proof} 

Finally, we prove the convergence of the dissipation potential.
\begin{lemma}\label{lem:lim-dissipation-epsilon}
    Let the family of measure-flux pairs $\{ (\rho^{\epsilon}, j^{\epsilon})\}_{\epsilon>0}\subset\mathcal{CE}_h (0,T)$ satisfying
    \begin{enumerate}[label=(\roman*)]
        \item $\rho^{\epsilon}_t \rightharpoonup \rho_t$ in $\calP(\calT^h)$ for all $t\in [0, T]$,
        \item $\int_\cdot j_t^{\epsilon} \dd t \rightharpoonup^* \int_\cdot j_t \dd t$ weakly-$*$ in $\calM((0, T)\times \Sigma^h)$.
    \end{enumerate}
    Then,
    \[
    	\int_s^t \calR_{\text{up},h}(\rho_r,j_r)\dd r \le \liminf_{\epsilon \to 0} \int_s^t \calR_{\epsilon,h} (\rho^{\epsilon}_r, j^{\epsilon}_r)\dd r\qquad \text{for any $[s,t]\subset[0,T]$},
    \]
    where
    $$
        \calR_{\text{up},h}(\rho,j) = \sum_{(K,L)\in\Sigma^h} \tau_{K|L}^h\bra*{ u_K \Bigg| \frac{j^{+}_{K|L}}{\tau_{K|L}^hu_K} \Bigg|^2 + u_L \Bigg| \frac{j^{-}_{K|L}}{\tau_{K|L}^hu_L} \Bigg|^2 } ,\qquad u_K=\frac{\rho_K}{|K|}.
    $$
\end{lemma}
\begin{proof}
We begin by proving the convergence
\[
	\lim_{\epsilon\to 0} \calR_{\epsilon,h}^* (\rho^{\epsilon}, \xi) = \calR_{\text{up},h}^*(\rho,\xi) = \frac{1}{4}\sum_{(K,L)\in\Sigma^h} \tau_{K|L}^h \bra*{u_K |\xi_{K|L}^+|^2 + u_L |\xi_{K|L}^-|^2}
\]
for any $\xi\in\calB(\Sigma^h)$.
Since $\rho^\epsilon$ converges pointwise to $\rho$ (cf.\ Remark~\ref{rem:pointwise}) and estimate \eqref{eq:epsilon-uniform} provides
\begin{align*}
	 \sum_{(K,L)\in\Sigma^h} \tau_{K|L}^h\alpha_\epsilon^* \bra*{u^{\epsilon}_K, u^{\epsilon}_L, \xi_{K|L}} 
	\le 2\sum_{(K,L)\in\Sigma^h} \tau_{K|L}^h\alpha_\epsilon^* \bra*{u^{\epsilon}_K, u^{\epsilon}_L, \|\xi\|_\infty}
	\le \|\xi\|_\infty^2 c_\calT ,
\end{align*}
we obtain the asserted convergence by means of Lemma~\ref{lem:properties-alpha}\ref{alpha-epsilon} and the dominated convergence.
    
We now use the Legendre duality to infer the asserted liminf inequality for the dissipation potential. From the convergence result established in the first part of the proof, it follows that 
    \begin{align*}
        \int_s^t &\sum_{(K,L)\in\Sigma^h} \chi_r \xi_{K|L} j_{K|L}(r) \dd r - \int_0^T 2\sum_{(K,L)\in\Sigma^h} \tau_{K|L}^h\alpha_0^* \bra*{ u_K(r), u_L(r), \chi_r \frac{\xi_{K|L}}{2} }  \dd r \\
        &\hspace{6em}\leq \lim_{\epsilon\to 0} \int_s^t \sum_{(K,L)\in\Sigma^h} \chi_r \xi_{K|L} j^{\epsilon}_{K|L}(r) \dd r  - \limsup_{\epsilon\to 0} \int_s^t \calR_{\epsilon,h}^* \big(\rho^{\epsilon}_r, \chi_r \xi \big) \dd r \\
        &\hspace{6em}\leq \liminf_{\epsilon\to 0} \int_s^t \bigg\{ \sum_{(K,L)\in\Sigma^h} \chi_r \xi_{K|L} j^{\epsilon}_{K|L}(r) - \calR_{\epsilon,h}^* \big(\rho^{\epsilon}_r, \chi_r \xi \big) \bigg\} \dd r \\
        &\hspace{6em}\leq \liminf_{\epsilon\to 0} \int_s^t  \calR_{\epsilon,h} (\rho^{\epsilon}_r, j^{\epsilon}_r ) \dd r\qquad \text{for any $\chi\in\calB([0,T])$, $\xi\in\calB(\Sigma^h)$.}
    \end{align*}
    Now let $\eta\in \calB([0,T]\times\Sigma^h)$. We introduce the measures $\Theta_\rho^\pm$, $\Theta\in \calM([0,T]\times\Sigma^h)$ in the way that for any measurable $A \subset [0,T]$ and $B \subset \Sigma^h$ it holds that
    \begin{gather*}
        \Theta(A\times B) = \int_A \sum_{(K,L)\in B} \tau_{K|L}^h \dd t, \\ 
        \Theta_\rho^+ (A\times B) = \int_A \sum_{(K,L)\in B} \tau_{K|L}^h u_K(t)  \dd t,\qquad \Theta_\rho^- (A\times B) = \int_A \sum_{(K,L)\in B} \tau_{K|L}^h u_L(t)  \dd t.
    \end{gather*}
    Then, we rewrite
    \begin{align*}
        &\int_s^t \sum_{(K,L)\in\Sigma^h} \eta_{K|L}(r) J_{K|L}(r) \dd r - \int_s^t 2\sum_{(K,L)\in\Sigma^h} \tau_{K|L}^h\alpha_0^* \bra*{ u_K(r), u_L(r), \frac{\eta_{K|L}(r)}{2}}  \dd r \\
        &\hspace{6em}= \iint_{[s,t]\times\Sigma^h} \eta \dd J - \iint_{[s,t]\times\Sigma^h} 2\,\alpha_0^* \Bigg( \frac{\dd\Theta^+_{\rho}}{\dd\Theta}, \frac{\dd\Theta^-_{\rho}}{\dd\Theta},\frac{\eta}{2} \Bigg) \dd\Theta \eqcolon I_0^{[s,t]} (\eta).
    \end{align*}
 
    It is left to determine $\sup_{\eta\in \calB([s,t]\times\Sigma^h)} I_0^{[s,t]} (\eta)$. We note that
   \begin{align*}
       \iint_{[s,t]\times\Sigma^h} \eta \dd J 
       = \iint_{[s,t]\times\Sigma^h} \eta^+  \bigg( \bigg[ \frac{\dd J}{\dd\Theta_\rho^+} \bigg]^+ \!\!- \bigg[ \frac{\dd J}{\dd\Theta_\rho^+} \bigg]^-\bigg) \dd\Theta_\rho^+ + \eta^- \bigg( \bigg[ \frac{\dd J}{\dd\Theta_\rho^-} \bigg]^- \!\!- \bigg[ \frac{\dd J}{\dd\Theta_\rho^-} \bigg]^+ \bigg) \dd\Theta_\rho^-.
   \end{align*}
    The two negative terms can only decrease the total value, therefore the supremum over $\calB([0,T]\times\Sigma^h)$ is equivalent to taking supremum over $\eta\in \calB([0,T]\times\Sigma^h)$ satisfying
    $$
        \eta^\pm \equiv 0 \qquad \text{on } \supp\, (J^{0})^\mp.
    $$
    Because of the structure of $\alpha_0^*$ with one part depending on $\eta^+$ and the other part depending on $\eta^-$, the expression under the supremum splits into two independent parts with the supremum over $\eta^+$ and the supremum over $\eta^-$. The first part is
    \begin{align*}
        \sup_{\eta\in \calB([s,t]\times\Sigma^h)} \bigg\{  \iint_{[s,t]\times\Sigma^h} \eta^+ \bigg[ \frac{\dd J}{\dd\Theta_\rho^+} \bigg]^+ \dd\Theta_\rho^+ - \left\|\frac{\eta^+}{2}\right\|^2_{L^2([s,t]\times\Sigma^h, \Theta_\rho^+)} \bigg\} = \Bigg\| \bigg[ \frac{\dd J}{\dd\Theta_\rho^+} \bigg]^+ \Bigg\|^2_{L^2([s,t]\times\Sigma^h, \Theta_\rho^+)}.
    \end{align*}
    and the second part is
    \begin{align*}
        \sup_{\eta\in \calB([s,t]\times\Sigma^h)} \bigg\{  \iint_{[s,t]\times\Sigma^h} \eta^- \bigg[ \frac{\dd J}{\dd\Theta_\rho^-} \bigg]^- \dd\Theta_\rho^- - \left\|\frac{\eta^-}{2}\right\|^2_{L^2([s,t]\times\Sigma^h, \Theta_\rho^-)} \bigg\} = \Bigg\| \bigg[ \frac{\dd J}{\dd\Theta_\rho^-} \bigg]^- \Bigg\|^2_{L^2([s,t]\times\Sigma^h, \Theta_\rho^-)}.
    \end{align*}
    In both parts, we imply that if the supremum is finite then it equals the $L^2$-norm of the corresponding flux densities. Combining the two, we obtain
    \begin{align*}
        \sup_{\eta\in \calB([s,t]\times\Sigma^h)} I^{[s,t]}(\eta) 
        &= \Bigg\| \bigg[ \frac{\dd J}{\dd\Theta_\rho^+} \bigg]^+ \Bigg\|^2_{L^2([s,t]\times\Sigma^h, \Theta_\rho^+)} + \Bigg\| \bigg[ \frac{\dd J}{\dd\Theta_\rho^-} \bigg]^- \Bigg\|^2_{L^2([s,t]\times\Sigma^h, \Theta_\rho^-)} \\
        &= \int_s^t \sum_{(K,L)\in\Sigma^h} \tau_{K|L}^h\bra*{ u_K(r) \Bigg| \frac{j^+_{K|L}(r)}{\tau_{K|L}^h u_K(r)} \Bigg|^2 + u_L(r) \Bigg| \frac{j^-_{K|L}(r)}{\tau_{K|L}^h u_L(r)} \Bigg|^2 }  \dd r \\
        &= \int_s^t \calR_{\text{up},h} (\rho_r, j_r) \dd r,
    \end{align*}
    therewith concluding the proof.
\end{proof}

To summarize, the energy-dissipation functional corresponding to the upwind scheme comprises the driving energy 
\begin{equation}\label{eq:upwind:energy}
   \calP(\calT^h)\ni \rho\mapsto \calE_{\text{up},h} (\rho) = \sum_{K\in\calT^h} V^h_K \rho_K + \frac{1}{2} \sum_{K,L\in\calT^h\times\calT^h} W^h_{KL} \rho_K \rho_L,
\end{equation}
the dissipation potential $\calR_{\text{up},h}: \calP(\calT^h) \times \calM(\Sigma^h) \to \R_+ \cup \{+\infty\}$
\begin{equation}\label{eq:upwind:dissipation}
   \calP(\calT^h) \times \calM(\Sigma^h)\ni (\rho,j)\mapsto \calR_{\text{up},h}(\rho, j) = \sum_{(K,L)\in\Sigma^h} \tau_{K|L}^h \bra*{ u_K \Bigg|\frac{j^{+}_{K|L}}{\tau_{K|L}^hu_K} \Bigg|^2  + u_L \Bigg| \frac{j^{-}_{K|L}}{\tau_{K|L}^hu_L} \Bigg|^2 } ,
\end{equation}
and the Fisher information 
\begin{equation}\label{eq:upwind:fisher}
    \calP(\calT^h)\ni\rho\mapsto \calD_{\text{up},h} (\rho) = \sum_{(K,L)\in \Sigma^h} \tau_{K|L}^h\bra*{ u_K\Bigg| \frac{q^{+}_{K|L}}{2} \Bigg|^2  + u_L\Bigg| \frac{q^{-}_{K|L}}{2} \Bigg|^2  } .
\end{equation}
For completeness, we point out that the dual dissipation potential in this case is
\begin{equation}\label{eq:upwind:dual}
    \calP(\calT^h) \times \calB(\Sigma^h)\ni (\rho,\xi)\mapsto \calR_{\text{up},h}^* (\rho, \xi) = \sum_{(K,L)\in \Sigma^h} \tau_{K|L}^h\bra*{ u_K\Bigg|\frac{\xi_{K|L}^{+}}{2} \Bigg|^2  + u_L\Bigg| \frac{\xi_{K|L}^{-}}{2} \Bigg|^2 } .
\end{equation}

\begin{proof}[Proof of Theorem~\ref{th:EDP-vanishing-diffusion}]

Consider a family $\{(\rho^{\epsilon,h}, j^{\epsilon,h})\}_{\epsilon>0}$ of GGF-solutions to \eqref{eq:SG} according to Definition~\ref{def_GGF_solution} and the tilt-independent structure introduced in Section~\ref{sec:tilted-gradient-structure}. Lemma~\ref{lem:compactness-flux-epsilon} and Lemma~\ref{lem:compactness-measure-epsilon} provide the existence of a subsequential limit pair $(\rho^{\text{up},h}, j^{\text{up},h}) \in \mathcal{CE}_h (0, T)$ and the convergence specified in Theorem~\ref{th:EDP-vanishing-diffusion}(1).
    
The liminf inequality for the energy-dissipation functionals from assertion (2) is proven in Lemma~\ref{lem:lim-fisher-epsilon} and Lemma~\ref{lem:lim-dissipation-epsilon}. With a simple chain rule, we easily deduce $\calI_{\text{up},h}^{[s,t]}(\rho^{\text{up},h}, j^{\text{up},h})\ge 0$ for every $[s,t]\subset[0,T]$, and hence, the limit pair $(\rho^{h,\epsilon}, j^{h,\epsilon})$ is the GGF solution of the upwind scheme \eqref{eq:upwind-scheme}.
\end{proof}

\subsection{Continuous case}\label{sec:vanishing-continuous} Recall that for $\epsilon>0$, a gradient flow solution $(\rho^\epsilon,j^\epsilon)$ of \eqref{eq:drift-aggregation-diffusion-equation} satisfies
\[
	\calI^{[s,t]}_\epsilon (\rho^\epsilon, j^\epsilon)=\int_s^t \bigl\{\calR(\rho_r^\epsilon,j_r^\epsilon) +  \calD_\epsilon(\rho_r^\epsilon)\bigr\} \dd r + \calE_\epsilon(\rho_t^\epsilon) - \calE_\epsilon(\rho_s^\epsilon) = 0\qquad\text{for all $[s,t]\subset[0,T]$,}
\]
with Fisher information
\[
	\calD_\epsilon(\rho) = 2\epsilon^2\int_\Omega \Bigl| \nabla \sqrt{u}\Bigr|^2\dd x + \epsilon\int_\Omega \nabla u \cdot \nabla \mathsf{Q}(\rho)\dd x + \frac{1}{2} \int_\Omega \lvert\nabla \mathsf{Q}(\rho)\rvert^2 \dd\rho, \qquad u = \frac{\dd \rho}{\dd\calL^d}.
\]
In particular, $\sqrt{u^\epsilon}\in H^1(\Omega)$ for every $\epsilon>0$.

As in the previous results, we will pass to the liminf in each of the terms in the energy-dissipation functional $\calI_\epsilon$. Due to the joint lower semicontinuity of the dissipation potential $\calR$ w.r.t.\ weak-$*$ convergence and the fact that $\calE_{\text{agg}}\le \calE_\epsilon$, the only difficulty here is in proving the liminf inequality for the Fisher information $\calD_\epsilon$, as it is unclear that the first two terms vanish in the limit.

However, since the chain rule $\nabla v^2 = 2 v\,\nabla v\in L^1(\Omega)$ holds for $v\in H^1(\Omega)$, the Fisher information $\calD_\epsilon$ takes the alternative form
\[
	\calD_\epsilon(\rho) = \frac{1}{2} \int_\Omega \left|2\epsilon \nabla \sqrt{u} + \sqrt{u}\,\nabla \mathsf{Q}(\rho)\right|^2 \dd x,\qquad u = \frac{\dd \rho}{\dd\calL^d},\;\sqrt{u}\in H^1(\Omega).
\]
Moreover, by defining the $\R^d$-valued measure
\begin{equation}\label{eq:g-FI-agg}
	g_t^\epsilon := \sqrt{u_t^\epsilon}\Bigl(2\epsilon \nabla \sqrt{u_t^\epsilon} + \sqrt{u_t^\epsilon}\nabla \mathsf{Q}(\rho_t^\epsilon)\Bigr)\calL^d = \bigl(\epsilon \nabla u_t^\epsilon + u_t^\epsilon\nabla \mathsf{Q}(\rho_t^\epsilon) \bigr)\calL^d\in \calM(\Omega;\R^d),
\end{equation}
for every $t\in[0,T]$, we can further express $\calD_\epsilon(\rho^\epsilon)$ as
\[
	\calD_\epsilon(\rho_t^\epsilon) = \frac{1}{2}\int_\Omega \left|\frac{\dd g_t^\epsilon}{\dd \rho_t^\epsilon} \right|^2 \dd\rho_t^\epsilon = \calR(\rho_t^\epsilon,g_t^\epsilon).
\]
Therefore, if $\rho_t^\epsilon \rightharpoonup^* \rho_t$ weakly-$*$ in $\calP(\Omega)$ and $g_t^\epsilon \rightharpoonup^* g_t$ weakly-$*$ in $\calM(\Omega;\R^d)$ for every $t\in[0,T]$, then the weak-$*$ lower semicontinuity of $\calR$ yields
\[
	\calR(\rho_t,g_t) \le \liminf_{\epsilon\to 0} \calR(\rho_t^\epsilon,g_t^\epsilon) = \liminf_{\epsilon\to 0} \calD_\epsilon(\rho_t^\epsilon).
\]
Hence, it suffices to show that $g_t^\epsilon \rightharpoonup^* \rho_t \nabla \mathsf{Q}(\rho_t)$ weakly-$*$ in $\calM(\Omega;\R^d)$ for every $t\in[0,T]$.

\begin{lemma}\label{lem:FI-vanishing-diffusion-continuous}
	Let $\{\rho^\epsilon\}_{\epsilon>0}\subset\calC([0,T];\calP(\Omega))$, $\rho\in\calC([0,T];\calP(\Omega))$ be such that $\rho_t^\epsilon\rightharpoonup^* \rho_t$ weakly-$*$ in $\calP(\Omega)$ for every $t\in[0,T]$ and the interaction potential $W$ satisfy \eqref{ass:smooth-potential}. Then for every $t\in[0,T]$, the sequence $\{g_t^\epsilon\}_{\epsilon>0}\subset \calM(\Omega;\R^d)$ defined in \eqref{eq:g-FI-agg} satisfies
	\[
		g_t^\epsilon \rightharpoonup^* g_t:=\rho_t  \nabla \mathsf{Q}(\rho_t)\quad\text{weakly-$*$ in $\calM(\Omega;\R^d)$}.
	\]
	In particular, we have
	\[
		\int_s^t \calD_{\text{agg}}(\rho_r) \dd r\le \liminf_{\epsilon\to 0} \int_s^t \calD_\epsilon(\rho_r^\epsilon)\dd r\qquad\text{for every $[s,t]\subset[0,T]$}.
	\]
\end{lemma}
\begin{proof}
	Let $\varphi\in\calC_c^1(\Omega;\R^d)$ be arbitrary and $t\in[0,T]$. Then
	\[
		\langle \varphi,g_t^\epsilon\rangle = \int_\Omega \varphi\cdot\bigl(\epsilon \nabla u_t^\epsilon + u_t^\epsilon\nabla \mathsf{Q}(\rho_t^\epsilon) \bigr)\dd x = - \epsilon \int_\Omega \text{div} \varphi\,\dd\rho_t^\epsilon + \int_\Omega \varphi\cdot\nabla\mathsf{Q}(\rho_t^\epsilon)\dd\rho_t^\epsilon,
	\]
	and therefore
	\begin{align*}
		|\langle \varphi,g_t^\epsilon - g_t\rangle | \le \epsilon \|\text{div}\varphi\|_{\sup} + \|\varphi\|_{\sup} \|\nabla \mathsf{Q}(\rho_t^\epsilon)-\nabla \mathsf{Q}(\rho_t)\|_{\sup} + |\langle \varphi\cdot\nabla \mathsf{Q}(\rho_t),\rho_t^\epsilon-\rho_t\rangle|
	\end{align*}
	From the assumptions placed on the potentials $V$ and $W$, one easily deduces the uniform convergence $\|\nabla \mathsf{Q}(\rho_t^\epsilon)-\nabla \mathsf{Q}(\rho_t)\|_{\sup} \to 0$ as $\epsilon\to 0$. Clearly, the other terms also converge to zero.
	
	Using the weak-$*$ lower semicontinuity of $\calR$, we then obtain
	\[
		\calR(\rho_t,g_t) \le \liminf_{\epsilon\to 0} \calD_\epsilon(\rho_t^\epsilon)\qquad\text{for every $t\in[0,T]$}.
	\]
	Since $\calD_\epsilon(\rho_t^\epsilon)\ge 0$ for $t\in[0,T]$, an application of Fatou's lemma then yields the result.
\end{proof}

Following the same strategy as in the previous sections, we obtain a compactness result for the solutions of \eqref{eq:drift-aggregation-diffusion-equation}.

\begin{lemma}\label{lem:compactness-continuous-epsilon} Let a family of pairs $\{(\rho^{\epsilon}, j^{\epsilon})\}_{\epsilon>0}\subset\mathcal{CE} (0, T)$ satisfying
    \[
    	c_0\coloneq\sup_{\epsilon>0}\int_0^T \calR_{\epsilon}(\rho_t^\epsilon, j_t^\epsilon) \dd t<\infty.
    \]
    Then there exist a limit pair $(\rho, j) \in \mathcal{CE}(0, T)$  and a (not relabelled) subsequence such that
    \begin{gather*}
        \rho^{\epsilon}_t \rightharpoonup^* \rho_t \quad \text{weakly-$*$ in $\calP(\Omega)$\; for all $t\in [0,T]$},\\
       \int_\cdot j_t^\epsilon\dd t \eqcolon J^{\epsilon} \rightharpoonup^* J=\int_\cdot j_t\dd t\quad\text{weakly-$*$ in $\calM([0,T]\times\Omega;\R^d)$}.
    \end{gather*}
\end{lemma}
\begin{proof}
	An application of Jensen's inequality immediately yields
	\begin{align*}
	\sup_{\epsilon>0}\||j_\cdot^\epsilon|(\Omega)\|_{L^2((0,T))}^2 \le 2\sup_{\epsilon>0}\int_0^T\calR(\rho_t^\epsilon,j_t^\epsilon)\dd t = 2 c_0.
\end{align*}
In particular, the sequence $\{t\mapsto |j_t^\epsilon|(\Omega)\}_{\epsilon>0}$ is equi-integrable, and the weak-$*$ compactness of $\{J^\epsilon\}_{\epsilon>0}$ can be proven as in Lemma~\ref{lemma:compactness-for-flux-tilted}.

 We now prove the asserted weak-$*$ convergence for the sequence $\{\rho^\epsilon\}_{\epsilon>0}\subset \calC([0,T];\calP(\R^d))$. Since $(\rho^\epsilon,j^\epsilon)$ satisfies the continuity equation \eqref{eq:CE}, for any $\varphi\in\calC_c^1(\R^d)$ with $\|\nabla \varphi\|_{L^\infty}\le 1$:
    \begin{align*}
    	|\langle \varphi,\rho_t^\epsilon - \rho_s^\epsilon\rangle| = \left|\int_s^t \langle \nabla \varphi, j_r^\epsilon\rangle\dd r\right| \le \int_s^t |j_r^\epsilon|(\Omega)\dd r \le \sqrt{|t-s|}\||j_\cdot^\epsilon|(\Omega)\|_{L^2((0,T))}.
    \end{align*}
    Taking the supremum over Lipschitz functions $\varphi$ satisfying $\|\nabla \varphi\|_{L^\infty}\le 1$ then gives
    \[
    	W_1(\rho_t^\epsilon,\rho_s^\epsilon) \le c_0\sqrt{|t-s|}\qquad\text{for all $\epsilon>0$ and $[s,t]\subset[0,T]$},
    \]
    where $W_1$ is the 1-Wasserstein distance. The Ascoli-Arzel\'a theorem then provides the existence of a limit curve $\rho\in \calC([0,T];\calP(\Omega))$ and a subsequence such that the convergence holds.
\end{proof}

We now conclude with the proof of Theorem~\ref{th:vanishing-diffusion-continuous}.

\begin{proof}[Proof of Theorem~\ref{th:vanishing-diffusion-continuous}]

Consider a family $\{(\rho^{\epsilon}, j^{\epsilon})\}_{\epsilon>0}$ of gradient flow solutions to \eqref{eq:drift-aggregation-diffusion-equation} according to Definition~\ref{def_GF_solution}. Lemma~\ref{lem:compactness-continuous-epsilon} provides the existence of a subsequential limit pair $(\rho, j) \in \mathcal{CE} (0, T)$ and the convergence specified in Theorem~\ref{th:vanishing-diffusion-continuous}(1).

To show the liminf inequality for the energy-dissipation functionals from assertion (2), we begin by noticing that 
\[
	\int_s^t \calR(\rho_r^\epsilon,j_r^\epsilon)\dd r = \frac{1}{2}\iint_{[s,t]\times \Omega} \left| \frac{\dd J^\epsilon}{\dd R^\epsilon} \right|^2 \dd R^\epsilon,\qquad R^\epsilon = \int_\cdot \rho_t^\epsilon\dd t,
\]
where the right-hand side is jointly weakly-$*$ lower semicontinuous as a functional on $\calM([s,t]\times\Omega)\times \calM([s,t]\times\Omega;\R^d)$. Since $(R^\epsilon,J^\epsilon)\rightharpoonup^* (R,J)$ weakly-$*$ in $\calM([s,t]\times\Omega)\times \calM([s,t]\times\Omega;\R^d)$ with $R = \int_\cdot \rho_t\dd t$ and $J=\int_\cdot j_t\dd t$, we then conclude that
\[
	\liminf_{\epsilon\to 0}\int_s^t \calR(\rho_r^\epsilon,j_r^\epsilon)\dd r \ge \frac{1}{2}\iint_{[s,t]\times \Omega} \left| \frac{\dd J}{\dd R} \right|^2 \dd R = \int_s^t \calR(\rho_r,j_r)\dd r.
\]
Together with Lemma~\ref{lem:FI-vanishing-diffusion-continuous} and the fact that $\calE_{\text{agg}}\le \calE_\epsilon$, we easily deduce the asserted liminf inequality $\calI_{\text{agg}}^{[s,t]}(\rho,j) \le \liminf_{\epsilon\to 0}\calI_{\epsilon}^{[s,t]}(\rho^\epsilon, j^\epsilon)=0$ for every $[s,t]\subset[0,T]$.

Finally, the chain rule \cite[E.\ Chain rule in Section~10.1.2]{ambrosio2008gradient} yields $\calI_{\text{agg}}^{[s,t]}(\rho,j) \geq 0$ for every $[s,t]\subset[0,T]$. Therefore, the limit pair $(\rho, j)$ is an $(\calE_{\text{agg}}, \calR, \calR^*)$-gradient flow solution of \eqref{eq:aggregation} in the sense of Definition~\ref{def_GF_solution}
\end{proof}

\section{From the Upwind Scheme to the Aggregation Equation}\label{sec:upwind-to-aggregation}

In this section, we complete the commutative diagram in Figure~\ref{diagram:with-theorems} by studying the variational convergence of the upwind scheme \eqref{eq:upwind-scheme} to the aggregation equation \eqref{eq:aggregation}. We mentioned earlier that we could not consider general tessellations in this section, thus, we restrict to Cartesian grids. Moreover, we assume \eqref{ass:smooth-potential} for the interaction potential $W$. On the other hand, we can handle any initial data $\rho_{\text{in}}^h\in \calP(\calT^h)$ satisfying $\hat{\rho}^h_\text{in} \rightharpoonup^* \rho_\text{in}$ weakly-$*$ in $\calP(\Omega)$ without any additional assumptions.

We work with $(\calE_{\text{up},h}, \calR_{\text{up},h}, \calR_{\text{up},h}^*)$-generalized gradient flow solutions of the upwind scheme \eqref{eq:upwind-scheme}, where $\calE_{\text{up},h}$, $\calR_{\text{up},h}$, and $\calR_{\text{up},h}^*$ are defined in \eqref{eq:upwind:energy}, \eqref{eq:upwind:dissipation}, and \eqref{eq:upwind:dual}, respectively. The strategy should be familiar to the reader by now. We begin with the necessary compactness result in Lemma~\ref{lem:compactness-upwind}. The convergence of the dual dissipation potential $\calR_{\text{up},h}^*$ and, consequently, the Fisher information $\calD_{\text{up},h}$ given in \eqref{eq:upwind:fisher}  is established in Theorem~\ref{th:fisher-upwind-to-aggregation}. We conclude this section with the proof of Theorem~\ref{th:upwind-to-aggregation}.

\medskip
We begin this section with a compactness result.

\begin{lemma}\label{lem:compactness-upwind}
	The family $\{J^h\}_{h>0}$ is weakly-* compact in $\calM((0, T)\times\Omega; \R^d)$ and the family $\{t\mapsto |\hat{\jmath}^h_t|\}_{h>0}$ is equi-integrable. In particular, there exists a (not relabelled) subsequence of $\{ (\hat{\rho}^h, \hat{\jmath}^h) \}_{h>0}$ and a pair $(\rho, j)\in \mathcal{CE}(0,T)$ such that
    \begin{gather*}
        \hat{\rho}^h_t \to \rho_t\quad\text{weakly-$*$ in $\calP(\Omega)$\; for all $t \in [0, T]$},\\
      \int_{\cdot} \hat\jmath^h_t \dd t \eqcolon J^h \rightharpoonup^* J=\int_\cdot j_t \dd t\quad\text{weakly-* in $\calM((0, T)\times\Omega; \R^d)$.}
    \end{gather*}
\end{lemma}
\begin{proof}
    The weak-$*$ compactness of $\{J^h\}_{h>0}$ and equi-integrability of the family $\{t\mapsto |\hat{\jmath}^h_t|\}_{h>0}$ can be proven as in Lemma~\ref{lemma:compactness-for-flux-tilted}. Indeed, using the dissipation potential $\calR_{\text{up},h}$ (cf.\ \eqref{eq:upwind:dissipation}) instead, we obtain
    \begin{align*}
    	\sup_{h>0}\||\hat\jmath_\cdot^h|(\Omega)\|_{L^2((0,T))}^2\le 2C_\tau d^2 \sup_{h>0}\int_0^T \calR_{\text{up},h}(\rho_t^h,j_t^h) \dd t =: c_0 <\infty.
    \end{align*}
    For the pointwise weak-$*$ convergence of $\{\hat{\rho}_t^h\}_{h>0}$, we simply mimic the proof of Lemma~\ref{lem:compactness-continuous-epsilon}.
\end{proof}

\begin{theorem}\label{th:fisher-upwind-to-aggregation}
    Let $\{ (\calT^h, \Sigma^h)\}_{h>0}$ be a family of  Cartesian tessellations with edge-length $h>0$. Let the family $\{ \rho^h\in\calP(\calT^h) \}_{h>0}$ satisfy $\hat{\rho}^h \rightharpoonup^* \rho$ weakly-$*$ in $\calP(\Omega)$. If the family of discrete functions $\{\varphi^h\in \calB(\calT^h)\}_{h>0}$ is such that for some $\varphi \in C_b^1(\Omega)$:
    $$
        \dnabla \varphi^h(K,L) = \nabla \varphi(x_K) \cdot (x_L - x_K) + o(h),
    $$
    then
    $$
        \lim_{h\to 0} \calR_{\text{up},h}^*(\rho^h, \dnabla\varphi^h) = \frac{1}{2} \int_\Omega |\nabla \varphi(x)|^2 \rho(\dd x).
    $$
    
    Consequently, if the interaction potential $W$ satisfies assumption \eqref{ass:smooth-potential}, then
    $$
        \lim_{h\to 0} \calD_{\text{up},h}(\rho^h) = \frac{1}{2} \int_\Omega \left| \nabla \mathsf{Q} (\rho) \right|^2 \dd\rho,
    $$
    with $\mathsf{Q}(\rho)=\nabla V + \nabla W \ast \rho$.
\end{theorem}
\begin{proof}
    Using symmetry, we rewrite the functional as
    \begin{align*}
        \calR_{\text{up},h}^*(\rho^h, \dnabla\varphi^h) &= \frac{1}{4} \sum_{(K,L)\in\Sigma^h} \Big( u^h_K \big|(\dnabla \varphi^h(K,L))^+ \big|^2 + u^h_L \big| (\dnabla \varphi^h(K,L))^- \big|^2 \Big)\tau_{K|L}^h \\
        &= \frac{1}{2} \sum_{(K,L)\in\Sigma^h} \big| (\dnabla \varphi^h(K,L))^+ \big|^2 u^h_K \tau_{K|L}^h.
    \end{align*}
    Since the mapping $\R \ni q \mapsto q^+$ is Lipschitz, we have that
    $$
        (\dnabla \varphi^h(K,L))^+ = \bigl(\nabla \varphi(x_K) \cdot (x_L - x_K)\bigr)^+ + o(h).
    $$
    Inserting this expression into the functional yields
    \begin{align*}
        \calR_{\text{up},h}^*(\rho^h, \dnabla\varphi^h) &= \frac{1}{2}\sum_{(K,L)\in\Sigma^h} \tau_{K|L}^h u^h_K\left| \bigl( \nabla \varphi (x_K) \cdot (x_L - x_K) \bigr)^+ \right|^2    + o(h^2) \sum_{(K,L)\in\Sigma^h}  \frac{\tau_{K|L}^h}{|K|}\rho^h_K \\
        &= \frac{1}{2}\sum_{K\in\calT^h} \Big \langle \nabla \varphi (x_K), \sum_{L\in\calT^h_K} \tau_{K|L}^h (x_L - x_K) \otimes (x_L - x_K)\, \imath^\varphi_K(L) \nabla \varphi (x_K) \Big \rangle \,u^h_K + o(1), 
    \end{align*}
    where we set
    $$
        \imath^\varphi_K \coloneq \Ind_{\{ M\in\calT^h: \nabla  \varphi (x_K) \cdot (x_M - x_K) > 0 \}} + \frac{1}{2} \Ind_{\{ M\in\calT^h:\nabla  \varphi (x_K) \cdot (x_M - x_K) = 0 \}}
    $$
    The indicator $\imath^\varphi_K$ means that for any cell $K\in\calT^h$ the sum goes only over the faces $(K|L)$ for which $\nabla \varphi (x_K) \cdot (x_L - x_K) > 0$. For the Cartesian grid, all the neighboring cells $\calT^h_K$ can be grouped in pairs $M, L \in \calT^h_K$  such that $x_L - x_K = - (x_M - x_K)$ and $x_L - x_K = \pm h e_i$ for some basis vector $e_i$, $i \in \left\{ 1, \dots, d \right\}$. We illustrate this idea in Figure~\ref{fig_explanation} below.
    \begin{figure}[ht]
    \centering
    \includegraphics[width=0.35\linewidth, height=5cm]{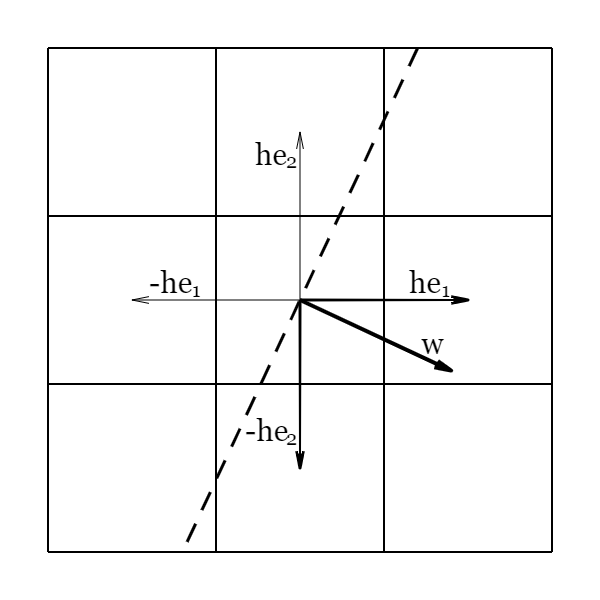} 
    \caption{Consider the 2-d Cartesian grid. Let the central cell be the cell $K$ and $w$ be a vector starting at $x_K$. Then the line orthogonal to $w$ (dashed line) splits the set of vectors $\{(x_L - x_K)\}_{L\in\calT^h_K}$ into two groups: one such that $w\cdot (x_L - x_K) > 0$ (i.e.\ $\{he_1, -he_2\}$) and the other such that $w\cdot (x_L - x_K) < 0$, (i.e, $\{-he_1, he_2\}$).}
    \label{fig_explanation}
    \end{figure}

    This means that for any $\nabla\varphi(x_K)$ that is not parallel to any basis vector $\{e_i\}_{i=1}^d$, the indicator $\imath^\varphi_{K}$ "chooses" all the basis vectors with either plus or minus sign. Hence, the tensor takes the form
    $$
        \sum_{L\in\calT^h_K} \tau_{K|L}^h (x_L - x_K) \otimes (x_L - x_K) \, \imath_K^{\varphi}(L) = h^d \sum_{i=1}^d e_i \otimes e_i = |K| \Id.
    $$
    If $\nabla\varphi(x_K)$ is parallel to some $e_i$ for some $i\in\{1,\dots,d\}$, then $\imath^\varphi_K$ includes both $he_i$ and $-he_i$ with the coefficient $1/2$, which does not change the form of the tensor.
    
    The expression above then simplifies to 
    $$
        \calR_{\text{up},h}^*(\rho^h, \dnabla\varphi^h) = \frac{1}{2} \sum_{K\in\calT^h} \left| \nabla \varphi (x_K) \right|^2 |K| u^h_K + o(1).
    $$
    Since $\nabla \varphi$ is uniformly continuous on $\Omega$, it holds that
    $$
        \left| \nabla \varphi (x_K) \right|^2 = \intbar_K \left| \nabla \varphi (x) \right|^2 \dd x + o(1).
    $$
    Therefore, the functional admits an integral form
    $$
        \calR_h(\rho^h, \dnabla\varphi^h) = \frac{1}{2} \int_\Omega |\nabla \varphi(x)|^2 \hat{\rho}^h(\dd x) + o(1)\quad \xrightarrow{h\to 0}\quad \frac{1}{2} \int_\Omega |\nabla \varphi(x)|^2 \rho(\dd x).
    $$
    
    As for the convergence of the Fisher information, the assumptions on $V$ and $W$ give
    \[
    	|\nabla \mathsf{Q} (\hat{\rho}^h) (x_K)|^2 = \intbar_K |\nabla \mathsf{Q} (\hat{\rho}^h) (x)|^2\dd x + o(1),
    \]
    and therefore,
    \begin{align*}
        \calD_{\text{up},h}(\rho^h) &= \frac{1}{2} \int_\Omega \left| \nabla \mathsf{Q} (\hat{\rho}^h) (x) \right|^2 \hat{\rho}^h (\dd x) + o(1).
    \end{align*}
    The assertion then follows from the weak-$*$ convergence $\hat{\rho}^h \rightharpoonup^* \rho$ in $\calP(\Omega)$ and the uniform convergence $\nabla \mathsf{Q} (\hat{\rho}^h) \to \nabla \mathsf{Q}(\rho)$ in $\calC(\overline\Omega)$.
\end{proof}

\begin{proof}[Proof of Theorem~\ref{th:upwind-to-aggregation}]

Consider a family $\{(\rho^h, j^h)\}_{h>0}$ of GGF-solutions to the upwind scheme \eqref{eq:upwind-scheme} according to Definition~\ref{def_GGF_solution} and the generalized gradient structure obtained as the EDP limit in Section~\ref{sec:vanishing-diffusion}. Let $\{(\hat\rho^h, \hat\jmath^h)\}_{h>0}$ be defined as in \eqref{eq:lifted-pair}. Then, the existence of a subsequential limit pair $(\rho, j) \in \mathcal{CE}(0,T)$ and the convergence specified in Theorem~\ref{th:upwind-to-aggregation}(1) follow from Lemma~\ref{lem:compactness-upwind}.
    
The convergence of the Fisher information is proven in Theorem~\ref{th:fisher-upwind-to-aggregation}. The liminf inequality for the dissipation potential follows from the limit of the dual dissipation potential shown in Theorem~\ref{th:fisher-upwind-to-aggregation} and a duality argument from \cite[Theorem~6.2(i)]{hraivoronska2023diffusive}. In this way, the assertion (2) is proven and it immediately follows that $\calI_{\text{agg}}^{[s,t]}(\rho, j) \leq \liminf_{h\to 0} \calI_{\text{up},h}^{[s,t]}(\rho^h, j^h)= 0$ for every $[s,t]\subset[0,T]$. On the other hand, the chain rule \cite[E.\ Chain rule in Section~10.1.2]{ambrosio2008gradient} yields $\calI_{\text{agg}}^{[s,t]}(\rho,j) \geq 0$ for every $[s,t]\subset[0,T]$. Therefore, the limit pair $(\rho, j)$ is an  $(\calE_{\text{agg}}, \calR, \calR^*)$-gradient flow solution of \eqref{eq:aggregation} in the sense of Definition~\ref{def_GF_solution}.
\end{proof}

\appendix

\section{The \texorpdfstring{$\cosh$}{cosh}-gradient structure, its tilt-dependence and its de-tilting}\label{appendix:cosh-tilting}

We link in this appendix a different possible gradient structure based on dissipation potentials based on the `$\cosh$' to the gradient structure defined by the dual dissipation potential~\eqref{eq:tilted-Rstar} and discussed in Section~\ref{sec:gf-SG}. We will see in Section~\ref{sec:cosh-structure-SG} that in the `$\cosh$' case, the edge activity $\vartheta^{h,\rho}$ depends on the potentials $V^h$, $W^h$ and hence also $\rho^h$. The latter dependency would enforce stronger assumptions on the tesselation. Additionally, we point out that those dependencies are scaled by factors of $\epsilon^{-1}$ and hence become singular in the vanishing diffusion limit $\epsilon\to 0$. Hence, one expects additional complications in proving EDP convergence for such a gradient structure.

Also from the modeling point of view, this dependence of the dissipation potential on the driving energy can be considered a drawback and unphysical, which we discuss in Section~\ref{sec:cosh-structure-tilt-dependence}. An in-depth discussion of tilt-dependent gradient systems, where changes in the driving energy can lead to changes in the dissipation potential, is carried out in \cite{peletier2023cosh}. Fortunately for the Scharfetter--Gummel scheme, it is possible to derive a tilt-independent gradient structure, which is better suited for proving EDP convergence uniform in the diffusivity $\epsilon\ll 1$. We present a de-tilting of the `$\cosh$' dissipation potential towards a tilt-independent dissipation potential in Section~\ref{sec:tilted-gradient-structure}.

\subsection{The \texorpdfstring{$\cosh$}{cosh}-gradient structure}\label{sec:cosh-structure-SG}

We show that the Scharfetter--Gummel scheme \eqref{eq:SG} defines a random walk that possesses a `cosh' gradient structure by following a similar construction of~\cite{hraivoronska2023diffusive}.
To handle the nonlocal aggregation part in the energy, we follow the strategy introduced in~\cite{EFLS2016} and define a \emph{local equilibrium} to arrive at a suitable gradient flow formulation incorporating, such that the scheme would indeed fit into the frame developed in~\cite{hraivoronska2023diffusive}.
From the discrete energy $\calE_h$ given in \eqref{eq:energy-with-interaction}, we identify the discrete gradient of its variational derivative as
\begin{equation}\label{eq:discrete_energy_variation}
	\dnabla \calE'_{\epsilon,h}(\rho^h) = \epsilon \dnabla \bigl(\log \rho^h - \log\pi^{\epsilon,h,\rho} \bigr),
\end{equation}
with 
\begin{equation}\label{eq:discrete_energy_variation_2}
	\pi^{\epsilon,h,\rho}_K = \frac{\abs{K} e^{- \mathsf{Q}_K^{h,\rho}/\epsilon}}{Z^{\epsilon,h,\rho}},\qquad \mathsf{Q}_K^{h,\rho} = V_K^h + \sum_{M\in \calT^h} W_{KM}^h \rho_M^h,
\end{equation}
and $Z^{\epsilon,h,\rho} = \sum_{K\in \calT^h} \abs{K}  e^{-\mathsf{Q}_K^{h,\rho}/\epsilon}$ is the normalization such that $\pi^{\epsilon,h,\rho} \in \calP(\calT^h)$. We note that the contribution of the variation of $Z^{\epsilon,h,\rho}$ in~\eqref{eq:discrete_energy_variation} cancels due to the discrete gradient.

The `cosh' dual dissipation potential is given for all $\rho^h \in \calP(\calT^h)$ and $\xi^h \in \calB(\Sigma^h)$ by
\begin{equation}\label{eq:def:cosh:Rstar}
    \overline\calR_{\epsilon,h}^*(\rho^h, \xi^h) = \frac{1}{2} \sum_{(K,L)\in\Sigma^h} \Psi^*_\epsilon (\xi^h_{K|L}) \sqrt{\bar u^h_K \bar u^h_L} \, \kappa^{\epsilon,h,\rho}_{K|L} \pi^{\epsilon,h,\rho}_K, \qquad \bar u_K^h = \frac{\rho_K^h}{\pi^{\epsilon,h,\rho}_K} ,
\end{equation}
where $\Psi_\epsilon^*(s) = 4 \epsilon^2 (\cosh(s/2 \epsilon) - 1)$.
The idea is then to choose a jump kernel $\kappa^{\epsilon,h,\rho} : \Sigma^h \to [0, \infty)$ in such a way that it satisfies the \emph{local detailed balance condition}
\begin{equation}\label{eq:local:DBC}
 \kappa^{\epsilon,h,\rho}_{K|L} \pi^{\epsilon,h,\rho}_K = \kappa^{\epsilon,h,\rho}_{L|K} \pi^{\epsilon,h,\rho}_L \qquad\text{ for all } (K,L)\in\Sigma^h \text{ and all } \rho^h \in \calP(\calT^h) .
\end{equation}
and allows representing the flux in the gradient form \eqref{eq:flux-force-SG}.

One possibility is to define the jump kernel as
\begin{equation}\label{eq:SG-jump-kernel}
    \kappa^{\epsilon,h,\rho}_{K|L} \coloneq \frac{1}{|K|}\frac{\tau_{K|L}^h}{\exp\bra[\big]{-\mathsf{Q}_K^{h,\rho}/\epsilon}} \frac{2\,q_{K|L}^h / \epsilon}{\exp(\mathsf{Q}^{h,\rho}_L / \epsilon) - \exp(\mathsf{Q}^{h,\rho}_K / \epsilon)}, \qquad (K,L)\in\Sigma^h,
\end{equation}
where we recall that $\tau_{K|L}^h\coloneq |(K|L)| / |x_L - x_K|$ is the transmission coefficient and 
\begin{equation}\label{eq:q-for-aggregation}
	q_{K|L}^h \coloneq V^h_L - V^h_K + \sum_{M\in\calT^h} \rho^h_M (W^h_{ML} - W^h_{MK}) = \mathsf{Q}_L^{h,\rho}-\mathsf{Q}_K^{h,\rho},\qquad (K,L)\in \Sigma^h.
\end{equation}
Notice that the pair $(\kappa^{\epsilon,h,\rho}, \pi^{\epsilon,h,\rho})$ satisfies the local detailed balance condition~\eqref{eq:local:DBC}, since $\tau_{K|L}^h = \tau_{L|K}^h$ and $q_{K|L}^h = - q_{L|K}^h$.
The {\em edge conductivity} is then given by
\begin{equation}\label{eq:theta-cosh}
    \vartheta^{\epsilon,h,\rho}_{K|L} \coloneq  \frac{\tau_{K|L}^h}{Z^{\epsilon,h,\rho}} \frac{2\, q_{K|L}^h / \epsilon}{\exp(\mathsf{Q}^{h,\rho}_L / \epsilon) - \exp(\mathsf{Q}^{h,\rho}_K / \epsilon)}.
\end{equation}
The kernel defined in \eqref{eq:SG-jump-kernel} satisfies the bound
\begin{equation}\label{eq:scheme-cosh:bounded}
	\sup_{h>0} \sup_{K\in\calT^h} h^2 \sum_{L\in\calT^h_K} \kappa^{\epsilon,h,\rho}_{K|L} < \infty,\qquad \text{where}\;\;\calT^h_K\coloneq\bigl\{L\in\calT^h: (K,L)\in \Sigma^h\bigr\},
\end{equation}
provided $\{(\calT^h, \Sigma^h)\}_{h>0}$ satisfy \eqref{ass:tessellation}. Indeed,  for any $(K,L)\in\Sigma^h$, it holds that 
\begin{align*}
	\kappa^{\epsilon,h,\rho}_{K|L} = \frac{|(K|L)|}{|K||x_K-x_L|} \frac{2\, q_{K|L}^h/\epsilon}{\exp(q_{K|L}^h/\epsilon) - 1  }
    \leq \frac{C_{d-1} h^{d-1}}{C_d \zeta^{d+1} h^{d+1}} \Big(2 - \frac{q_{K|L}^h}{\epsilon} + o(h) \Big) = O(h^{-2}).
\end{align*}
It is not difficult to see that the non-degeneracy assumption~\eqref{ass:tessellation} implies that \cite{hraivoronska2023diffusive}
\[
	\sup_{h>0} \sup_{K\in\calT^h} \#\calT^h_K < \infty,
\]
and thus also the asserted bound~\eqref{eq:scheme-cosh:bounded}. 
However, we point out that the bound is non-uniform in $\epsilon\ll 1$. 

\medskip
To apply the strategy from~\cite{hraivoronska2023diffusive} directly, it is left to show that the choice of $\kappa^{\epsilon,h,\rho}$ in~\eqref{eq:SG-jump-kernel} indeed gives rise to the Scharfetter--Gummel flux~\eqref{eq:SG-flux}.

\begin{lemma}\label{lem:flux-force-for-cosh}
    For any $\rho^h\in\calP(\calT^h)$, $K\in\calT^h$, and $(K,L)\in\Sigma^h$, we have the identity \eqref{eq:flux-force-SG},
    where $\calJ^{h,\rho}$ is the Scharfetter--Gummel flux given in \eqref{eq:SG-flux} and $\overline\calR_{\epsilon,h}^*$ is the `cosh' dual dissipation potential with edge conductivity $\vartheta^{\epsilon,h,\rho}$ defined in \eqref{eq:theta-cosh}. 
    
    In particular, the Scharfetter--Gummel scheme \eqref{eq:SG} possesses the `cosh' gradient flow structure with \eqref{eq:energy-with-interaction} as the driving energy.
\end{lemma}
\begin{proof}
We begin by rewriting the Scharfetter--Gummel flux in~\eqref{eq:SG-flux} using the density $\bar u^h = \dd\rho^h / \dd\pi^{\epsilon,h,\rho}$ with the reference measure $\pi^{h,\rho}$ depending on $\mathsf{Q}^{h,\rho}$:
\begin{equation}\label{eq:SG-flux-V-dependent}
   \calJ_{K|L}^{h,\rho} = \frac{\epsilon \tau_{K|L}^h}{Z^{h,\rho}} \left( \Bern( q_{K|L}^h / \epsilon) \,\bar u^h_K e^{-\mathsf{Q}^{h,\rho}_K/ \epsilon} - \Bern(- q_{K|L}^h / \epsilon )\, \bar u^h_L e^{-\mathsf{Q}^{h,\rho}_L/ \epsilon} \right).
\end{equation}
The expression \eqref{eq:SG-flux-V-dependent} can be simplified, since
\begin{align*}
    \Bern(q_{K|L}^h / \epsilon)\exp(-\mathsf{Q}^{h,\rho}_K / \epsilon)
    = \frac{q_{K|L}^h \exp(-\mathsf{Q}^{h,\rho}_K / \epsilon)}{\epsilon \left( \exp(q_{K|L}^h / \epsilon) - 1 \right)}
    = \frac{ q_{K|L}^h}{\epsilon \bra[\big]{\exp(\mathsf{Q}^{h,\rho}_L / \epsilon) - \exp(\mathsf{Q}^{h,\rho}_K / \epsilon) }}
\end{align*}
and, similarly, 
$$
    \Bern(-q_{K|L}^h / \epsilon) \exp(-\mathsf{Q}^{h,\rho}_L / \epsilon) = \frac{ q_{K|L}^h}{\epsilon \bra[\big]{\exp(\mathsf{Q}^{h,\rho}_L / \epsilon) - \exp(\mathsf{Q}^{h,\rho}_K / \epsilon) }},
$$
therefore
$$
    \calJ_{K|L}^{h,\rho} = \frac{\tau_{K|L}^h}{Z^{h,\rho}} \frac{ q_{K|L}^h}{\exp(\mathsf{Q}^{h,\rho}_L / \epsilon) - \exp(\mathsf{Q}^{h,\rho}_K / \epsilon)} \bra*{ \bar u^h_K - \bar u^h_L } = \frac{\epsilon}{2}\bra*{ \bar u^h_K - \bar u^h_L } \vartheta_{K|L}^{\epsilon,h,\rho} .
$$
On the other hand, we note that for every $(K,L)\in\Sigma^h$ and $\xi^h\in \calB(\Sigma^h)$:
\[
  D_2\overline\calR^*_{\epsilon,h}(\rho^h,\xi^h)(K,L) = \epsilon \sinh\bra*{\frac{\xi^h_{K|L}}{2\epsilon}} \sqrt{\bar u^h_K  \bar u^h_L}  \,\vartheta^{\epsilon,h,\rho}_{K|L}.
\]
Recall from~\eqref{eq:discrete_energy_variation} and~\eqref{eq:def:cosh:Rstar} that $\dnabla \calE'_{\epsilon,h}(\rho^h) = \epsilon \dnabla \log(\bar u^h)$. Inserting $\xi^h = - \dnabla \calE'_{\epsilon,h}(\rho^h)$, we obtain
 \begin{align*}
     D_2 \overline\calR_{\epsilon,h}^* \big(\rho^h, - \dnabla \calE'_{\epsilon,h}(\rho^h)\bigr) (K, L) &= \epsilon \sinh\bra[\bigg]{\frac{1}{2} \log \frac{\bar u^h_K}{\bar u^h_L}} \sqrt{\bar u^h_K \bar u^h_L} \,\vartheta^{\epsilon,h,\rho}_{K|L} = \calJ_{K|L}^{h,\rho},
 \end{align*}
 i.e.\ identity \eqref{eq:flux-force-SG} holds as asserted.
 \end{proof}

\begin{remark}\label{rem:SG-cosh}
Since the classical Scharfetter--Gummel scheme has the `cosh' gradient-flow formulation, one can ask if it is possible to use the framework of \cite{hraivoronska2023diffusive} to prove the convergence. The necessary assumptions on the invariant measure $\pi^{\epsilon,h,\rho}$ and the jump intensities $\kappa^{\epsilon,h,\rho}$ hold true based on the notion of local detailed balance as defined in~\eqref{eq:local:DBC}. However, the zero-local-average assumption 
\begin{equation}\label{assumption_local_average}
    \sum_{L\in\calT^h_K} \vartheta^{\epsilon,h,\rho}_{K|L} (x_K - x_L) = 0\qquad\text{for all $K\in\calT^h$ with $\overline{K} \cap \partial\Omega = \emptyset$ does not hold.}
\end{equation}
In addition, the nonlinear dependency of $\vartheta^{\epsilon,h,\rho}$ on $\rho$ seems to make satisfying~\eqref{assumption_local_average}, even only asymptotically, very hard and may require strong assumptions on the tessellations to work around. 

As a last remark, we emphasize that the edge conductivity $\vartheta^{\epsilon,h,\rho}$ defined in \eqref{eq:theta-cosh} depends non-uniformly on the diffusion parameter $\epsilon>0$, which makes it difficult to pass to the limit  $\epsilon \to 0$.
\end{remark}

\subsection{The tilt-dependence of the \texorpdfstring{$\cosh$}{cosh}-gradient structure}
\label{sec:cosh-structure-tilt-dependence}

The disadvantages of the `cosh' gradient structure mentioned earlier are due to tilt-dependence as defined in~\cite{peletier2023cosh}. To clarify this further, we decompose the free energy into entropy and potential energies by writing
\begin{equation}\label{eq:def:discrete_free_energy}
	\calE_{\epsilon,h}(\rho^h) = \epsilon \calS_h(\rho^h) + \calV_h^V(\rho^h) + \calW_h^W(\rho^h),
\end{equation}
where $V^h:\calT_h \to \R$ and $W^h:\calT_h\times \calT_h\to \R$ symmetric are given and we set
\begin{gather*}
	\calS_h(\rho_h) \coloneq \sum_{K\in\calT^h} \phi(u^h_K )|K|  , \qquad\text{ where } \qquad u^h_K \coloneq \frac{\rho^h_K}{|K|} ; \\
	\calV_h^V(\rho^h) \coloneq \sum_{K\in\calT^h} V^h_K \rho^h_K \qquad\text{and}\qquad \calW_h^W(\rho^h) \coloneq \frac{1}{2} \sum_{K, L\in\calT^h\times\calT^h} W^h_{KL} \rho^h_K \rho^h_L .
\end{gather*}
Then, we can provide a gradient structure for the Scharfetter--Gummel scheme for all possible potential energies $V^h$ and interaction energies $W^h$ altogether by introducing the set of {\em tilts}
\begin{equation}\label{eq:def:tilts}
	\calF_h \coloneq \set*{ \calV_h^V + \calW_h^W \,\Big|\, V^h : \calT^h \to \R ,\; W^h: \calT^h\times \calT^h \to \R \text{ symmetric}} . 
\end{equation}
We can then recast Lemma~\ref{lem:flux-force-for-cosh} as a derivation of a gradient structure with tilting~\cite[Definition 1.16]{peletier2023cosh} of the type $(\calT^h,\Sigma^h,\dnabla,\calS_h,\overline\calR_{\epsilon,h},\calF_h)$. By recalling that for $\calV_h^V + \calW_h^W \in \calF_h$, we find $\mathsf{Q}^{h,\rho} = (\calV_h^V)'(\rho^h) + (\calW_h^W)'(\rho^h)$ as defined in~\eqref{eq:discrete_energy_variation_2} and obtain from~\eqref{eq:def:cosh:Rstar} the dissipation potential with tilting defined by
\begin{equation}\label{eq:def:cosh:R}
	\overline\calR_{\epsilon,h}(\rho^h,j^h;\calV_h^V + \calW_h^W) \coloneq \frac{1}{2}\sum_{(K,L)\in \Sigma^h} \Psi_\epsilon\Biggl(\frac{j_{K|L}^h}{\sqrt{\bar u^h_K \bar u^h_L} \vartheta_{K|L}^{\epsilon,h,\rho}}\Biggr)\sqrt{\bar u^h_K \bar u^h_L} \vartheta_{K|L}^{\epsilon,h,\rho},\qquad \bar u_K^h = \frac{\rho_K^h}{\pi^{\epsilon,h,\rho}}.
\end{equation}
In particular, it depends on the potential energies $V^h,W^h$ through $\vartheta^{\epsilon,h,\rho}$ defined in~\eqref{eq:theta-cosh} and hence is tilt-dependent. Its undesirable properties explained in Remark~\ref{rem:SG-cosh} are a direct consequence of the dependency of the gradient structure on the potentials and the diffusivity $\epsilon>0$.

\subsection{Detilting of the \texorpdfstring{$\cosh$}{cosh}-gradient structure}
\label{sec:tilted-gradient-structure}

In this section, we deduce the dual dissipation potential~\eqref{eq:tilted-Rstar} from the one including the `$\cosh$' in~\eqref{eq:def:cosh:Rstar} by a de-tilting construction explained in~\cite[Remark 1.17]{peletier2023cosh}. 
In this way, we prove the following Lemma. 
\begin{lemma}\label{lem:flux-force-for-tilt}
	The Scharfetter--Gummel with flux-force relation~\eqref{eq:SG-flux} is induced by a gradient structure with tilting $(\calT^h,\Sigma^h,\dnabla, \calS_h,\calR_{\epsilon,h},\calF_h)$ with tilt set~$\calF_h$ given in~\eqref{eq:def:tilts}. Moreover, the dissipation potential~$\calR_{\epsilon,h}$ is tilt-independent and given by
	\begin{equation}\label{eq:tilted-R}
		\calR_{\epsilon,h}(\rho^h, j^h) = 2\sum_{(K,L)\in \Sigma^h} \tau_{K|L}^h \,\alpha_\epsilon\bra*{ u^h_K ,u^h_L, \frac{j^h_{K|L}}{\tau_{K|L}^h}} , \qquad u_K^h\coloneq \frac{\rho_K^h}{|K|}, 
	\end{equation}
	where $\alpha_\epsilon$ is the Legendre dual of $\alpha_\epsilon^*$ given in~\eqref{eq:def:alpha*} with respect to the third variable.
\end{lemma}
\begin{proof}
	First, we make the tilt-dependence of the `$\cosh$'-dual dissipation potential $\overline\calR_h^*$ from~\eqref{eq:def:cosh:Rstar} explicit, for which use the primal dissipation potential defined in~\eqref{eq:def:cosh:R} and can rewrite~\eqref{eq:def:cosh:Rstar} as
	\begin{equation*}
		\overline\calR^*_{\epsilon,h}(\rho^h,\xi^h; \calV_h^V+\calW_h^W) =
		\frac{1}{2} \sum_{(K,L)\in\Sigma^h} \Psi^*_\epsilon(\xi^h_{K|L})\, \sqrt{\bar u^h_K \bar u^h_L} \vartheta_{K|L}^{\epsilon,h,\rho},\qquad \bar u_K^h = \frac{\rho_K^h}{\pi_K^{\epsilon,h,\rho}}.
	\end{equation*}
	Note, that the tilt-dependence comes through $\vartheta^{\epsilon,h,\rho}$ in terms of $\mathsf{Q}^{h,\rho}$ defined in~\eqref{eq:theta-cosh} and ~\eqref{eq:discrete_energy_variation_2}, respectively. Specifically, we fix $(K,L)\in \Sigma^h$ and identify $q_{K|L}^h = \dnabla\mathsf{Q}^{h,\rho}$ in $\vartheta^{\epsilon,h,\rho}_{K|L}$ from~\eqref{eq:theta-cosh} to obtain
	\begin{equation*}
		\sqrt{\bar u^h_K \bar u^h_L} \vartheta_{K|L}^{\epsilon,h,\rho} = \tau_{K|L}^h \sqrt{u_K^h u_L^h} \frac{\dnabla \mathsf{Q}^{h,\rho}_{K|L}}{\exp\bra[\big]{\dnabla \mathsf{Q}^{h,\rho}_{K|L}/(2\epsilon)}-\exp\bra[\big]{-\dnabla \mathsf{Q}^{h,\rho}_{K|L}/(2\epsilon)}} \,.
	\end{equation*}
	In this way, we can write	
	\begin{equation*}
		\overline\calR^*_{\epsilon,h}(\rho^h,\xi^h; \calV_h^V+\calW_h^W) = \frac{1}{2} \sum_{(K,L)\in\Sigma^h} \overline{\mathsf{R}}^*_{\epsilon,h}(K|L,\rho^h_K,\rho^h_L,\xi^h_{K|L};\mathsf{Q}^{h,\rho}_K - \mathsf{Q}^{h,\rho}_L) \,,
	\end{equation*}
	with
	\begin{equation}\label{eq:def:cosh:edge:R*}
		\overline{\mathsf{R}}^*_{\epsilon,h}(K|L;\rho^h_K,\rho^h_L;\xi^h_{K|L};q_{K|L}^h)  \coloneq \Psi^*_\epsilon(\xi^h_{K|L})\, \tau_{K|L}^h \sqrt{ u_K^h u_L} \frac{q_{K|L}^h}{e^{q_{K|L}^h/(2\epsilon)}- e^{-q_{K|L}^h/(2\epsilon)}}\,.
	\end{equation}
	Following the construction from~\cite[Remark 1.17]{peletier2023cosh}, we can make use of the fact that the solution is given for the specific force $\xi_{K|L}^h=- \epsilon\dnabla \log \rho^h(K,L)-q_{K|L}^h$, which is the negative discrete gradient of the discrete free energy from~\eqref{eq:def:discrete_free_energy}. With this, one can define an \emph{evolution-equivalent} dissipation potential $\tilde{\mathsf{R}}^*_{\epsilon,h}$ (cf.~\cite[Eqn.\ (1.64)]{peletier2023cosh}) by first differentiating the dual dissipation potential $\overline{\mathsf{R}}^*_{\epsilon,h}$ in~\eqref{eq:def:cosh:edge:R*} w.r.t.\ the $\xi$-variable denoted with $D_3$, then use the substitution for $q_{K|L}^h$ along solutions of the scheme and finally integrate w.r.t.\ the $\xi$-variable again to obtain
	\begin{equation*}
		\tilde{\mathsf{R}}^*_{\epsilon,h}(K|L;\rho^h_K,\rho^h_L;\Xi^h_{K|L})  \coloneq \int_0^{\Xi^h_{K|L}} D_3 \overline{\mathsf{R}}^*_{\epsilon,h}(K|L;\rho^h_K,\rho^h_L;\xi^h_{K|L};-\xi_{K|L}^h- \epsilon\dnabla \log \rho^h(K,L)) \dd \xi_{K|L}^h .
	\end{equation*}
	Instead of calculating the integral, we check that $\tilde{\mathsf{R}}^*_{\epsilon,h}$ agrees with $D_2 \calR_{\epsilon,h}^*(\rho^h,\xi^h)$ from~\eqref{eq:tilted-Rstar}, which amounts to verifying the identity
	\begin{equation}\label{eq:detilt}
		D_2 \calR_{\epsilon,h}^*(\rho^h,\xi^h)(K,L) \stackrel{!}{=} D_2\overline\calR_{\epsilon,h}^*\bra*{\rho^h,\xi^h; -\xi^h-\epsilon\dnabla \calS_h(\rho)}(K,L) . 
	\end{equation}
	By substituting once more $\dnabla \mathsf{Q}^{h,\rho}_{K|L} = q_{K|L}^h = -\xi_{K|L}^h- \epsilon\dnabla \log \rho^h(K,L)$, which amounts in using the identity~\eqref{eq:qKL:energy}, we observe that
	\begin{align*}
		\MoveEqLeft D_2\overline\calR_{\epsilon,h}^*\bra*{\rho^h,\xi^h; - \xi^h -\dnabla \calS_h(\rho)}(K,L)\\
		&= \epsilon \tau_{K|L}^h \sinh\bra[\bigg]{\frac{\xi^h_{K|L}}{2\epsilon}} \sqrt{u_K^h u_L^h}\frac{\log\bra[\big]{u^h_K e^{-\xi^h_{K|L} / 2\epsilon}} - \log\bra[\big]{u^h_L e^{-\xi^h_{K|L} / 2\epsilon} }}{e^{-\xi^h_{K|L}/2\epsilon -\dnabla\log \sqrt{u^h}(K,L) }-e^{\xi_{K|L}/2\epsilon+\dnabla\log \sqrt{u^h}(K,L)}} \\
		&= \alpha_\epsilon\bra*{u^h_K,u^h_L, \xi^h_{K|L}/2} = D_2\calR_{\epsilon,h}^*(\rho^h,\xi^h)(K,L),
	\end{align*}
	which verifies the claimed identity~\eqref{eq:detilt}. Therefore, the solution property is a consequence of~\eqref{eq:SG:KR} and the remaining statements about the tilt-independence in Lemma~\ref{lem:flux-force-for-tilt} follow by construction as argued in~\cite[Remark 1.17]{peletier2023cosh}.
\end{proof}

\section{Properties of the Tilted Dual Dissipation Potential}
\label{appendix:function-alpha}

The following lemma contains some properties and an integral representation of the harmonic-logarithm mean $\Lambda_H$ introduced in~\eqref{eq:def:HarmLogMean}. 

\begin{definition}[Mean]
	A function $M:\R_+ \times \R_+ \to \R_+$ is a \emph{mean} if it is
	\begin{enumerate}
		\item positively one-homogeneous: $M(\lambda s,\lambda t) = \lambda M(s,t)$ for all $s,t\in \R_+$ and $\lambda >0$;
		\item bounded by $\min\set*{s,t}\leq M(s,t)\leq \max\set*{s,t}$ for all $s,t\in \R_+$;
		\item jointly concave.
	\end{enumerate}
\end{definition}
\begin{lemma}[Harmonic-logarithmic mean]\label{lem:harm_log_mean}
	The \emph{logarithmic mean} $\Lambda: \R_+ \times \R_+ \to \R_+$,
	\[ 
	\Lambda(s,t) = \int_0^1 s^\tau t^{1-\tau} \dd \tau = 
	\begin{cases}
		\frac{s-t}{\log s - \log t}  , & s\ne t ;\\
		s , & s=t . 
	\end{cases}
	\]
	is a mean between the geometric and arithmetic mean
	\[
	  \sqrt{st } \leq \Lambda(s,t) \leq  \frac{s+t}{2} ,
	\]
	with derivatives bounded
	\[
	\partial_1 \Lambda(s,t) = \partial_2\Lambda(t,s) \qquad\text{and}\qquad \partial_1 \Lambda(s,t) = \frac{\Lambda(s,t)(s-\Lambda(s,t))}{s(s-t)} \ . 
	\]
	The \emph{harmonic-logarithmic mean} $\Lambda_H : \R_+ \times \R_+ \to \R_+$ defined by 
	\[
	\Lambda_H(s,t) = \frac{1}{\Lambda\bra*{1/s, 1/t}}  = \frac{st}{\Lambda(s, t)}
	\] 
	is a mean between the harmonic and geometric mean
	\[
	  \frac{2}{\frac{1}{s}+ \frac{1}{t}} \leq \Lambda_H(s,t) \leq \sqrt{st} 
	\]	
	with the integral representations
	\[
	\Lambda_H(a,b) = \int_0^1 \frac{\dd\tau}{\tau/s+ (1-\tau)/t} = \int_0^\infty \frac{s\, t\dd\tau}{(\tau +s) (\tau+t)} 
	\]
	and derivatives
	\[
	\partial_1 \Lambda_H(s,t)=\partial_2 \Lambda_H(t,s) = \frac{t\left( \Lambda(s, t) - t \right)}{\Lambda(s, t)}.
	\]
\end{lemma}
\begin{proof}
	See, for instance~\cite{Bhatia2008} for many properties of the logarithmic mean, from which the analogous ones of the harmonic-logarithmic mean follow.
\end{proof}
The tilt-independent dual dissipation potential $\calR_{\epsilon,h}^*$ in~\eqref{eq:tilted-Rstar} is given in terms of the function $\alpha^*_\epsilon$ defined in~\eqref{eq:def:alpha*}, which we recall here for convenience 
\begin{equation*}
    \alpha_\epsilon^*(a, b, \xi) = \epsilon \int_0^\xi \sinh\left( \frac{x}{\epsilon} \right) \Lambda_H (a e^{-x/\epsilon}, b e^{x/\epsilon}) \dd x= \epsilon^2 \alpha_1 \left(a, b, \frac{\xi}{\epsilon} \right).
\end{equation*}
Below we prove useful properties of $\alpha_\epsilon^*$.
\begin{lemma}\label{lem:properties-alpha}
The function $\alpha_\epsilon^*:\R_+\times\R_+\times\R\to \R_+$ in~\eqref{eq:def:alpha*} has the following useful properties:
    \begin{enumerate}[label=(\alph*)]
        \item $\alpha_\epsilon^* (a, b, \xi)$ is convex in $\xi$ for fixed $a,b>0$, with $\min\set*{a,b} \leq \partial_{\xi}^2 \alpha_\epsilon^* (a, b, \xi) \leq \max\set*{a,b}$;
        \item  
        $\alpha_\epsilon^* (a, b, \xi)$ is positively one-homogeneous and jointly concave in $(a,b)$ for fixed $\xi$;
        \item \label{alpha-cosh} $\alpha_\epsilon^*$ satisfies the following bound:
        $$
        \alpha_\epsilon^* (a, b, \xi) \leq \epsilon^2 \sqrt{ab} \bigg( \cosh\bigg( \bigg| \frac{\xi}{\epsilon} \bigg| \bigg) - 1 \bigg)= \frac{1}{4}\sqrt{ab}\,\Psi^*(2\xi).
    $$
    Moreover, the expansion for $\abs{\xi}\ll 1$ is given by
        $$
  	        \alpha_\epsilon^*(a,b,\xi) = \Lambda_H(a,b) \frac{\xi^2}{2} + O\bra*{\frac{\abs{\xi}^3}{\epsilon}};
        $$
        \item \label{alpha-epsilon} It holds that
        $$
            \alpha_\epsilon^*(a,b,\xi) \to \frac{1}{2} \big( a (\xi^+)^2 + b (\xi^-)^2 \big) \eqcolon \alpha_0^*(a,b,\xi) \qquad\text{as } \epsilon \to 0 \:,
        $$
        where $\xi^\pm$ is the positive and negative part of $\xi$, respectively. Moreover,
        $$
            | \alpha_\epsilon^*(a,b,\xi) - \alpha_0^* (a,b,\xi) | = O(C_{a, b, \xi} \, \epsilon),
        $$
        where the constant $C_{a, b, \xi} < \infty$ depends on $a, b, \xi$.
        \item The function $\beta_\epsilon: \R_+\times \R_+\to \R_+$ defined for the argument $\xi = - \epsilon \log \sqrt{b/a}$ in $\alpha_\epsilon^*$ has the representation
        \begin{align*}
            \beta_\epsilon(a, b) \coloneq \alpha_\epsilon^* (a, b, -\epsilon \log \sqrt{b/a})
            &= \frac{\epsilon^2}{4} \int_a^b \frac{ab}{z} \left[ \frac{1}{\Lambda(z,a)} - \frac{1}{\Lambda(z,b)} \right] \dd z;
        \end{align*}
        
    \item \label{item:beta:bound}  The function $\beta_\epsilon:\R_+\times\R_+\to\R_+$ defined in (e) is jointly convex, continuous with
		$$
		\beta_\epsilon (a, 0) \coloneq  \frac{\epsilon^2}{4} \frac{\pi^2}{6} a \quad \text{and, symmetrically, } \quad
		\beta_\epsilon (0, b) \coloneq  \frac{\epsilon^2}{4} \frac{\pi^2}{6} b,
		$$
		and satisfies the following bounds:
		$$
		\frac{\epsilon^2}{4} (\sqrt{a}-\sqrt{b})^2 \leq \frac{\epsilon^2}{4} \frac{\bra*{a-b}^2}{a+b} \leq \beta_\epsilon(a, b) \leq \frac{\epsilon^2}{2} (\sqrt{a}-\sqrt{b})^2;
		$$
		Moreover, the function $\R_+ \times \R_+ \ni (a, b) \mapsto \beta_\epsilon (a^2, b^2)$ is differentiable.
    
		\item  \label{alpha-decomposition} The function $\alpha_\epsilon^*\bigl(a,b,-\epsilon \log\sqrt{b/a} + q / 2 \bigr)$ has the expansion
		\begin{align*}
			\alpha_\epsilon^*\left(a,b,-\epsilon \log\sqrt{b/a} + q/2 \right) &= \beta_\epsilon(a, b) + \frac{\epsilon}{4}(a-b)\,q + \frac{q^2}{4} \mathbb{h}_\epsilon (a, b, q)
		\end{align*}
		with 
		$$
		\mathbb{h}_\epsilon (a, b, q) \coloneq \int_0^1 \Bigl[a\, \mathfrak{h}\left(\lambda q/\epsilon\right) + b\,\mathfrak{h}\left(-\lambda q/\epsilon\right)\Bigr](1-\lambda)\dd\lambda,\qquad \mathfrak{h}(s) = \frac{1}{4}\frac{e^s-1-s}{\sinh^2(s/2)}.
		$$
    \end{enumerate}
  \end{lemma}
  \begin{proof}
  \emph{(a)} From the representation of $\alpha^*_1$ in terms of the harmonic-logarithmic mean, it follows that
  $$
    \partial_\xi \alpha^*_1(a, b, \xi) = \sinh(\xi) \Lambda_H (ae^{-\xi}, b e^\xi) = \sinh(\xi) \frac{ab}{\Lambda (ae^{-\xi}, b e^\xi)}.
  $$
It also holds
\[
  \partial_\xi^2 \alpha^*_1(a,b,\xi) = \frac{a\, b}{\bra*{a e^{-\xi} - b e^{\xi}}^2} \bra*{ a (e^{-2\xi}-1)+ b (e^{2\xi} -1) + (a-b)\bra*{ \log\frac{e^{-\xi}}{b} - \log\frac{e^{\xi}}{a}} } \:,
\]
which can be rewritten with the help of the function
\[
  g(x) = \frac{x \log x - x +1}{(x-1)^2} 
\]
as 
\[
  \partial_\xi^2 \alpha^*_1(a,b,\xi) = a \, g\bra*{\frac{a}{b}e^{-2\xi}} + b \, g\bra*{ \frac{b}{a} e^{2\xi}} \:.
\]
The convexity follows now by observing that
\[
  \forall x\in [0,1] : 0\leq g(x) \leq 1 \quad\text{and}\quad g(x) + g(x^{-1}) = 1 
\]
and hence the bound
\[
  \min\{a,b\} \leq \partial_\xi^2 \alpha^*_1(a,b,\xi) \leq \max\{ a ,b \} \:,
\]
implying the convexity in $\xi$ for fixed $a,b>0$.

\emph{(b)} The positive one-homogeneity and joint concavity follow from the properties of $\Lambda_H$.

\emph{(c)} Let $\xi>0$. Using the inequality between the harmonic-logarithmic and geometric mean, we obtain
\begin{align*}
    \alpha_1 (a, b, \xi) = \int_0^\xi \sinh(x) \Lambda_H (a e^{-x}, b e^x) \dd x 
    \leq \int_0^\xi \sinh(x) \sqrt{ab} \dd x = \sqrt{ab} \big( \cosh(\xi) - 1 \big).
\end{align*}
If $\xi < 0$, then
\begin{align*}
    \alpha_1 (a, b, \xi) = \int_0^{|\xi|} \sinh(x) \Lambda_H(ae^x, be^{-x}) \dd x \leq \sqrt{ab} \big( \cosh(|\xi|) - 1 \big). 
\end{align*}
Combining the two cases and considering $\alpha_\epsilon$, we get
$$
    \alpha_\epsilon (a, b, \xi) \leq \epsilon^2 \sqrt{ab} \bigg( \cosh\bigg( \bigg| \frac{\xi}{\epsilon} \bigg| \bigg) - 1 \bigg).
$$

As for the asymptotic expansion, we obtain, by definition of $\alpha^*_1$, 
$$
    \alpha^*_1(a, b, \xi) = \partial_\xi^2 \alpha^*_1(a, b, \xi) |_{\xi=0} \frac{\xi^2}{2} + O\left( |\xi|^3 \right)
    = \Lambda_H (a, b) \frac{\xi^2}{2} + O\left( |\xi|^3 \right).
$$
Then it follows directly that
$$
    \alpha^*_\epsilon(a, b, \xi) = \epsilon^2 \alpha^*_1 \left(a, b, \frac{\xi}{\epsilon} \right) = \Lambda_H (a, b) \frac{\xi^2}{2} + O\left( \frac{|\xi|^3}{\epsilon} \right).
$$

\emph{(d)} We rewrite $\alpha^*_\epsilon$ as
\begin{align*}
    \alpha^*_\epsilon(a, b, \xi) &=\epsilon^2 \int_0^{\xi/\epsilon} \sinh(x) \Lambda_H (ae^{-x}, be^x) \dd x
    = \frac{\epsilon^2}{2} \int_0^{\xi/\epsilon} \big( \Lambda_H (a, be^{2x}) - \Lambda_H (ae^{-2x}, b) \big) \dd x \\
    &= \frac{\epsilon}{2} \int_0^\xi \big( \Lambda_H (a, be^{2x/\epsilon}) - \Lambda_H (ae^{-2x/\epsilon}, b) \big)\dd x.
\end{align*}
For $x > 0$, it holds that 
\begin{align*}
    \frac{\epsilon}{2} \Lambda_H (a, be^{2x/\epsilon})
    = \frac{\epsilon}{2} \frac{ab e^{2x/\epsilon}}{a - be^{2x/\epsilon}} \Big( \log \frac{a}{b} - \frac{2x}{\epsilon} \Big)
    =  \frac{ab}{a e^{-2x/\epsilon}- b} \Big( \frac{\epsilon}{2} \log \frac{a}{b} - x\Big)
    \xrightarrow{\epsilon \to 0} ax,
\end{align*}
and
\begin{align*}
    - \frac{\epsilon}{2} \Lambda_H (ae^{-2x/\epsilon}, b)
    = - \frac{ab}{a - b e^{2x/\epsilon}} \Big( \frac{\epsilon}{2} \log \frac{a}{b} - x\Big) \xrightarrow{\epsilon \to 0} 0.
\end{align*}
For $x < 0$, similarly, we obtain
$$
    \frac{\epsilon}{2} \Big( \Lambda_H (a, be^{2x/\epsilon}) - \Lambda_H (ae^{-2x/\epsilon}, b)  \Big) \xrightarrow{\epsilon \to 0} bx.
$$
Combining the two cases yields
$$
    \lim_{\epsilon\to 0} \alpha^*_\epsilon(a, b, \xi)
    = \Ind_{\xi > 0} \int_0^\xi ax \dd x  + \Ind_{\xi < 0} \int_0^\xi bx \dd x = \frac{1}{2} \big( a (\xi^+)^2 + b (\xi^-)^2 \big).
$$

\emph{(e)} Direct calculation shows
\begin{align*}
    \beta_\epsilon(a, b) &= \alpha^*_\epsilon \left(a, b, -\epsilon \log \sqrt{\frac{b}{a}} \right) 
    = \epsilon^2 \alpha^*_1 \left(a, b, \log \sqrt{\frac{a}{b}} \right) \\
    &= \epsilon^2 \int_0^{\log\sqrt{a/b}} \sinh(x) \Lambda_H(ae^{-x}, be^x) \dd x 
    = \frac{\epsilon^2}{4} \int_1^{a/b} \left(\sqrt{y} - \frac{1}{\sqrt{y}} \right) \frac{ab}{\Lambda\left( \frac{a}{\sqrt{y}}, b\sqrt{y} \right)} \frac{1}{y} \dd y \\
    &= \frac{\epsilon^2}{4} \int_1^{a/b} \frac{ab}{y} \left[ \frac{1}{\Lambda \left(a/y, b \right)} - \frac{1}{\Lambda \left(a, b y \right)} \right] \dd y 
    = \frac{\epsilon^2}{4} \int_a^b \frac{ab}{z} \left[ \frac{1}{\Lambda(z, a)} - \frac{1}{\Lambda(z, b)} \right] \dd y.
\end{align*}
      
\emph{(f)} The joint convexity of $\beta_\epsilon$ follows from \emph{(a)} and \emph{(b)}. It is clear that $\beta_\epsilon$ is continuously differentiable in $\R_+ \times \R_+$ since it is defined as an integral of a bounded continuous function. However, on the boundary $\{0\}\times [0, +\infty) \cup [0, +\infty) \times \{0\}$ some partial derivatives become $-\infty$. In the case of $(a, b) \mapsto \beta_\epsilon(a^2, b^2)$, the directional derivatives are continuous and bounded:
\begin{align*}
    0 \geq \partial_1 \beta_1 (a^2, 1) &= - 2a \int_{a^2}^1 \frac{\log z}{z (z - 1)} \dd z \geq - 2a \int_{a^2}^1 \frac{1}{z \sqrt{z}} \dd z = 4a \frac{1}{\sqrt{z}}\Big|_{a^2}^1 \\
    &= 4a \left( 1 - \frac{1}{a} \right) = 4 (a - 1) > -\infty.
\end{align*}

As for the bounds, we begin with the
 \textit{upper bound.} Using the inequality that the harmonic-logarithmic mean is less or equal to the geometric mean yields
\begin{align*}
    \beta_\epsilon (a, b)
         %&= \alpha^*_\epsilon (a, b, -\epsilon \log \sqrt{b/a})
         \leq \epsilon^2 \sqrt{ab} \int_0^{-\log \sqrt{b/a}} \sinh(x) \dd x 
         &= \epsilon^2 \sqrt{ab} \left( \cosh\left( -\log \sqrt{b/a} \right) - 1 \right) \\
         &=  \frac{\epsilon^2}{2}\left( \sqrt{a} - \sqrt{b} \right)^2.
\end{align*}
     
     \textit{Tight lower bound.} Since $\beta_1$ is positively one-homogeneous it is enough to prove that
     $$
        \beta_1(a, 1) \geq \gamma(a) \coloneq \frac{1}{4}
     \frac{(a-1)^2}{a+1} \qquad \forall a\geq 0.
     $$
     For $a = 0$ the inequality holds, since $\beta_1(0, 1) = \frac{1}{4} \frac{\pi^2}{6} \geq \frac{1}{4} = \gamma(0)$. It is left to consider $a > 0$.
     
     We notice that $\beta_1(1, 1) = 0 = \gamma(1)$. Now we aim to compare the derivatives $\partial_a \beta_1(a,1) $ and $\partial_a \gamma(a)$ for $a\in (0,1)$ and $a\in(1,\infty)$. The derivative of $\gamma$ is
     $$
        \partial_a \gamma(a) = \frac{1}{4} \frac{(a-1)(a+3)}{(a+1)^2} = \int_1^a \frac{2}{(z+1)^3} \dd z
     $$
     We use the representation of $\beta_1$ from \emph{(e)} and apply the change of variables $y = z/a$ in the first part of the integral
     \begin{align*}
         \partial_a \beta_1(a, 1) 
         &= \frac{1}{4} \partial_a \left[ \int_1^{1/a} \frac{1}{y \Lambda(y,1)} \dd y - a \int_a^1 \frac{1}{z \Lambda(z,1)} \dd z \right] \\
         &= \frac{a}{\Lambda(1/a, 1)} \left( -\frac{1}{a^2} \right) - \int_a^1 \frac{1}{z \Lambda(z,1)} \dd z + \frac{1}{\Lambda(a, 1)} \\
         &= \int_1^a \frac{1}{z \Lambda(z,1)} \dd z
         = \int_1^a \frac{\log z}{z (z - 1)} \dd z .
     \end{align*}
     Therefore,
     $$
        \partial_a \left(\beta_1(a,1) - \gamma(a) \right) = \int_1^a \left[ \frac{\log z}{z (z - 1)}  - \frac{2}{(z+1)^3} \right] \dd z.
     $$
     We are left to show that the integrand is positive, and then the bound follows. For $z>1$, the integrand is positive, if and only if
     $$
        \log z \geq \frac{8z(z-1)}{(z+1)^3},
     $$
     which can be shown again by comparing the derivatives
     \begin{align*}
         \frac{1}{z} - 8\frac{-z^2 + 4z -1}{(z+1)^4} = \frac{(z-1)^2 (z^2 + 14z + 1)}{z (z+1)^4} > 0 \qquad \forall z >1.
     \end{align*}
    
\textit{Rough lower bound.} This lower bound follows from the inequality between the geometric and arithmetic means
$$
    \frac{(a-b)^2}{a+b} = \left( \sqrt{a} - \sqrt{b} \right)^2 \left( 1 + \frac{2\sqrt{ab}}{a +b} \right) \geq 2 \left( \sqrt{a} - \sqrt{b} \right)^2.
$$
        
\emph{(g)} We apply the second-order Taylor expansion for a function $f$:
\[
    f(y) = f(x) + f'(x)(y-x) + (y-x)^2\int_0^1 f''((1-\lambda)x + \lambda y)(1-\lambda)\dd\lambda
\]
to expand the function $\alpha^*_\epsilon$, obtaining
\begin{align*}
    \alpha^*_\epsilon \Big(a,b,-\epsilon \log\sqrt{\frac{b}{a}} + \frac{q}{2} \Big) &= \alpha^*_\epsilon \Big(a,b,-\epsilon \log\sqrt{\frac{b}{a}}  \Big) + \frac{q}{2} \,\partial_\xi(\alpha^*_\epsilon) \Big(a,b,-\epsilon \log\sqrt{\frac{b}{a}} \Big) \\
    &\hspace{1cm}+ \frac{q^2}{4} \int_0^1 (\partial_\xi^2\alpha^*_\epsilon) \Big (a,b,-\epsilon \log\sqrt{\frac{b}{a}} + 
    \lambda \frac{q}{2} \Big)(1-\lambda)\dd\lambda.
\end{align*}
After some manipulation, we find that
\begin{align*}
    (\partial_\xi\alpha^*_\epsilon)\Big(a,b,-\epsilon \log\sqrt{\frac{b}{a}} \Big) &= \frac{\epsilon}{2}(a-b),\\
(\partial_\xi^2\alpha^*_\epsilon) \Big(a,b,-\epsilon \log\sqrt{\frac{b}{a}} + 
    \frac{q}{2} \Big) &= a\, \mathfrak{h} \Big(\frac{q}{\epsilon} \Big) + b\,\mathfrak{h}\left(-\frac{q}{\epsilon}\right),
\end{align*}
with
\[
    \mathfrak{h}(s) = \frac{1}{4}\frac{e^s-1-s}{\sinh^2(s/2)}.
\]
Hence,
\begin{align*}
    \alpha^*_\epsilon \Big(a,b,-\epsilon \log\sqrt{\frac{b}{a}} + \frac{q}{2} \Big) &= \beta_\epsilon (a,b) + \frac{\epsilon}{4}(a-b)\,q \\
    &\hspace{1cm} + \frac{q^2}{4} \int_0^1 \Bigl[a\, \mathfrak{h}\left(\lambda q/\epsilon\right) + b\,\mathfrak{h}\left(-\lambda q/\epsilon\right)\Bigr](1-\lambda)\dd\lambda,
\end{align*}
therewith concluding the proof.
\end{proof}

\bibliographystyle{abbrv}
\bibliography{ref}

\end{document}